\pdfoutput=1
\documentclass{amsart}
\usepackage[utf8]{inputenc}

\usepackage[margin=1in]{geometry}
\usepackage[expansion=false]{microtype}

\usepackage{amsmath, amsfonts, amssymb, amsthm, amsopn, mathtools}
\usepackage{eucal}

\usepackage{tikz-cd}
\usetikzlibrary{patterns}
\usepackage{bm}
\usepackage{rotating}
\usepackage{xparse}

\usepackage[pdfusetitle,unicode,colorlinks]{hyperref}

\numberwithin{equation}{subsection}

\theoremstyle{plain}
\newtheorem{theorem}[equation]{Theorem}
\newtheorem{proposition}[equation]{Proposition}

\newtheorem{lemma}[equation]{Lemma}
\newtheorem{corollary}[equation]{Corollary}

\theoremstyle{definition}
\newtheorem{definition}[equation]{Definition}
\newtheorem{example}[equation]{Example}
\newtheorem{remark}[equation]{Remark}

\newtheorem{convention}[equation]{Convention}

\let\scr=\mathcal
\let\bb=\mathbf
\let\phi=\varphi

\let\into=\hookrightarrow
\let\onto=\twoheadrightarrow

\def\AA{\scr A}
\def\BB{\scr B}
\def\CC{\scr C}
\def\DD{\scr D}

\def\GG{\scr G}
\def\II{\scr I}
\def\JJ{\scr J}

\def\LL{\scr L}
\def\OO{\scr O}
\def\PP{\scr P}

\def\RR{\scr R}
\def\SS{\scr S}

\def\WW{\scr W}
\def\XX{\scr X}

\def\ZZ{\scr Z}

\def\AAA{\widehat{\AA}}
\def\BBB{\widehat{\BB}}

\def\SSS{\widehat{\SS}}

\def\SSSS{\hathat{\SS}}
\def\BBBB{\hathat{\BB}}

\def\bU{\mathbf{U}}
\def\bV{\mathbf{V}}
\def\bW{\mathbf{W}}

\DeclareMathOperator{\id}{id}
\DeclareMathOperator{\ev}{ev}
\DeclareMathOperator{\Ind}{Ind}
\DeclareMathOperator{\PSh}{PSh}
\DeclareMathOperator{\IPSh}{\mathsf{PSh}}
\DeclareMathOperator{\Cat}{Cat}
\DeclareMathOperator{\Shv}{Sh}
\DeclareMathOperator{\Cart}{Cart}
\DeclareMathOperator{\Cocart}{Cocart}
\DeclareMathOperator{\RFib}{RFib}
\DeclareMathOperator{\LFib}{LFib}
\DeclareMathOperator{\ILFib}{\mathsf{LFib}}
\DeclareMathOperator{\IRFib}{\mathsf{RFib}}
\DeclareMathOperator{\RTop}{RTop}
\DeclareMathOperator{\LTop}{LTop}
\DeclareMathOperator{\Set}{Set}
\DeclareMathOperator{\Sub}{Sub}
\DeclareMathOperator{\Tw}{Tw}
\DeclareMathOperator{\Fun}{Fun}
\DeclareMathOperator{\Map}{map}

\DeclareMathOperator{\Image}{Im}
\DeclareMathOperator{\Grpd}{Grpd}
\DeclareMathOperator{\Grp}{Grp}
\DeclareMathOperator{\Pyk}{Pyk}
\DeclareMathOperator{\const}{const}
\DeclareMathOperator{\diag}{diag}

\DeclareMathOperator{\QCoh}{QCoh}
\DeclareMathOperator{\D}{D}
\DeclareMathOperator{\MH}{H}
\DeclareMathOperator{\SH}{SH}

\DeclareMathOperator{\Sch}{Sch}
\DeclareMathOperator{\pr}{pr}

\newcommand{\op}{\mathrm{op}}
\newcommand{\core}{\simeq}
\newcommand{\gp}{\mathrm{gpd}}
\newcommand{\ord}[1]{\langle{#1}\rangle}
\newcommand{\BigRTop}{\widehat{\RTop}}
\newcommand{\map}[1]{\Map_{#1}}
\newcommand{\Eq}[1]{\mathrm{eq}_{#1}}
\newcommand{\Over}[2]{#1_{\hspace{-1pt}/#2}}
\newcommand{\Under}[2]{#1_{\hspace{-1pt}#2/}}
\newcommand{\sslash}{\mathbin{/\mkern-6mu/}}
\newcommand{\I}[1]{\mathsf{#1}}
\newcommand{\Cech}{\check C}
\newcommand{\Coh}{\bb{H}}
\newcommand{\ibot}{{\begin{sideways}$\!\Vdash$\end{sideways}}}
\renewcommand{\smallint}{\textstyle\int}
\newcommand{\iFun}[2]{{[#1,#2]}}
\newcommand{\Comma}[3]{{#1}\downarrow_{#2}{#3}}
\newcommand{\Cocomma}[3]{{#1}\diamond_{#2}{#3}}
\newcommand{\Simp}[1]{#1_{\Delta}}
\newcommand{\CatS}{\Cat_{\infty}}
\newcommand{\CatSS}{\widehat{\Cat}_{\infty}}
\newcommand{\CatSSS}{\hathat{\Cat}_{\infty}}

\NewDocumentCommand{\Gen}{m o}{%
	\IfNoValueTF{#2}{%
		\langle #1\rangle%
	}{%
		\langle #1\rangle_{#2}%
	}%
}

\NewDocumentCommand{\Univ}{o}{%
	\IfNoValueTF{#1}{%
		\I{\Omega}%
	}{%
		\I{\Omega}_{#1}%
	}%
}
\NewDocumentCommand{\UnivHat}{o}{%
	\IfNoValueTF{#1}{%
		\widehat{\I{\Omega}}%
	}{%
		\widehat{\I{\Omega}}_{#1}%
	}%
}

\let\lim=\relax
\DeclareMathOperator*{\lim}{lim}
\DeclareMathOperator*{\colim}{colim}

\makeatletter
\g@addto@macro\bfseries{\boldmath}
\makeatother

\makeatletter
\newcommand{\hathatInternal}[2]{%
	\begingroup%
	\let\macc@kerna\z@%
	\let\macc@kernb\z@%
	\let\macc@nucleus\@empty%
	\widehat{\raisebox{#2}{\vphantom{\ensuremath{#1}}}\smash{\widehat{#1}}}%
	\endgroup%
}
\makeatother
\newcommand{\hathat}[1]{\mathchoice
	{\hathatInternal{#1}{.3ex}}
	{\hathatInternal{#1}{.2ex}}
	{\hathatInternal{#1}{-1.5pt}}
	{\hathatInternal{#1}{1pt}}
}

\title{Yoneda's lemma for internal higher categories}
\author{Louis Martini}
\address{Norwegian University of Science and Technology (NTNU)\\
Alfred Getz' vei 1\\
7034 Trondheim\\
Norway}
\email{\href{mailto:louis.o.martini@ntnu.no}{louis.o.martini@ntnu.no}}
\date{\today}
\begin{document}
\begin{abstract}
We develop some basic concepts in the theory of higher categories internal to an arbitrary $\infty$-topos. We define internal left and right fibrations and prove a version of the Grothendieck construction and of Yoneda's lemma for internal categories.
\end{abstract}
\maketitle
\setcounter{tocdepth}{1}
\tableofcontents

\section{Introduction}

\subsection*{Motivation}
In various areas of geometry, one of the principal strategies is to study geometric objects by means of algebraic invariants such as cohomology, $K$-theory and (stable or unstable) homotopy groups. Usually these invariants are constructed in a functorial way, in the sense that they define (possibly higher) presheaves on a suitable (higher) category $\CC$ of the geometric objects of interest. Usually, $\CC$ is equipped with a Grothendieck topology that encodes the topological behaviour of the geometric objects contained in $\CC$, and one typically expects reasonable algebraic invariants to respect this topology in that they should define (higher) sheaves on $\CC$. In that way, one can study the global nature of these objects by means of their local behaviour. For example, if $A$ is an abelian group and if $n\geq 0$ is an integer, there is a higher sheaf $\Coh^n(-, A)$ for the \'etale topology on the category $\Sch$ of schemes that takes values in the $\infty$-category $\Grp(\SS)$ of group objects in $\infty$-groupoids, such that for any scheme $X$ the group $\pi_0\Coh^n(X, A)\in \Grp(\Set)$ is the $n$th \'etale cohomology group of $X$ with coefficients in $A$~\cite[\S~7]{htt}.

In modern flavours of geometry, the study of categorified invariants has come more and more into focus. For example, one of the fundamental invariants that one can associate to a scheme $X$ is its unbounded derived category $\D(X)$. In order to turn this construction into a sheaf on $\Sch$, i.e.\ in order to be able to \emph{glue} derived categories, one first passes to an $\infty$-categorical enhancement: there is a sheaf $X\mapsto \QCoh(X)$ on $\Sch$ that sends $X$ to the symmetric monoidal stable $\infty$-category of quasi-coherent sheaves on $X$ such that the homotopy category of $\QCoh(X)$ recovers the derived category $\D(X)$ of $X$. The properties of $\QCoh$ and their relations to the geometry of the scheme $X$ have been investigated by numerous authors, for example by Lurie in~\cite{lurie2018} and by Ben-Zvi, Francis and Nadler in~\cite{benzvi2010}. Furthermore, the sheaf $\QCoh$ plays a fundamental role in geometric representation theory; in~\cite{gaitsgory2014} Gaitsgory studies section-wise module $\infty$-categories over the sheaf $\QCoh$, i.e.\ the sheaf of $\infty$-categories that are acted on by $\QCoh$.
Another categorical invariant on schemes comes from motivic homotopy theory: the assignment that sends a scheme $X$ to its unstable motivic homotopy $\infty$-category $\MH(X)$ defines a sheaf for the Nisnevich topology on $\Sch$, and similarly the assignment $X\mapsto \SH(X)$ in which $\SH(X)$ is the stable motivic homotopy $\infty$-category on $X$ defines a sheaf for the Nisnevich topology on $\Sch$~\cite{hoyois2017}. The latter plays a prominent role in the formalisation of the six operations in motivic homotopy theory, see for example~\cite{cisinski2019b} for an overview.

It has been long known that for many choices of an (a priori higher) category $\AA$, the datum of an $\AA$-valued sheaf on $\Sch$ is equivalent to that of an \emph{$\AA$-object} internal to the $\infty$-topos $\Shv_{\SS}(\Sch)$, or the $1$-topos $\Shv_{\Set}(\Sch)$ in the case that $\AA$ is actually a $1$-category. For example, a sheaf of abelian groups on $\Sch$ is simply an abelian group internal to $\Shv_{\Set}(\Sch)$, and the collection of structure sheaves $\OO_X$ for $X\in\Sch$ are encoded by a single ring object in $\Shv_{\Set}(\Sch)$, the sheaf represented by the affine line. By making use of the internal logic of the $\infty$-topos of sheaves on $\Sch$, one can therefore study such invariants in the same way as one studies their \emph{non-parametrised} counterpart in the $\infty$-category of $\infty$-groupoids, see for example the PhD thesis of Blechschmidt~\cite{blechschmidt2017} for an application of these ideas to algebraic geometry in the $1$-categorical case. The study of higher invariants, i.e.\  sheaves on $\Sch$ that take values in a reasonable $\infty$-category $\AA$, thus naturally leads to the emerging field of homotopy type theory~\cite{hott2013}. In fact, it is now known~\cite{shulman2019} that much of homotopy type theory has a model in an arbitrary $\infty$-topos. 

By extending the same line of thought to the case of categorical invariants on schemes, the datum of a sheaf of $\infty$-categories on $\Sch$ is equivalent to that of a category internal to the $\infty$-topos $\Shv_{\SS}(\Sch)$. One should therefore be able to formulate and study the properties of such sheaves by means of studying their categorical properties when viewed internally in the $\infty$-topos $\Shv_{\SS}(\Sch)$. Such an internal perspective on category theory is not only useful for the study of categorical invariants on schemes, but in other situations as well. For example, work of Barwick, Glasman and Haine~\cite{barwick2020, barwick2019} shows that the pro\'etale $\infty$-topos $S_{\text{pro\'et}}$ of any coherent scheme $S$ can be naturally regarded as a \emph{pyknotic $\infty$-category}, i.e.\ as a category internal to the $\infty$-topos $\Pyk(\SS)$ of pyknotic $\infty$-groupoids. Moreover, Sebastian Wolf shows in~\cite{wolf2020} that internally in $\Pyk(\SS)$ the $\infty$-topos $S_{\text{pro\'et}}$ turns out to simply be given by a presheaf category, which amounts to a significant simplification of the structure of $S_{\text{pro\'et}}$.
Although bare homotopy type theory is insufficient to argue synthetically about such structures (owing to the presence of non-invertible arrows), Riehl and Shulman~\cite{shulman2017} have proposed an additional layer to homotopy type theory that rectifies this problem and that is powerful enough to support a synthetic formulation of basic higher categorical constructions.

In the present paper, our goal is to study sheaves of $\infty$-categories from a semantic point of view: as categories internal to an $\infty$-topos $\BB$. On the one hand, this is meant to provide tools for studying categorical invariants in geometry, such as those that are mentioned above. On the other hand, the study of categories internal to an $\infty$-topos $\BB$ should provide categorical semantics for the type theory developed by Riehl and Shulman. We do not aim to turn this into a formal statement, nevertheless we believe that developing the semantical side of the story might serve as a bridge to foster future development of the syntactic theory as well.

\subsection*{Main results}
The present paper constitutes the first in a series of papers aimed at developing the theory of categories internal to an arbitrary $\infty$-topos $\BB$, hereafter referred to as \emph{$\BB$-categories}. Arguably, Yoneda's lemma is at the very core of any flavour of category theory, so in this paper we will focus on a proof of this result in the context of internal higher categories.

We define the $\infty$-category $\Cat(\BB)$ of $\BB$-categories as the full subcategory of the $\infty$-category of simplicial objects in $\BB$ that satisfy the \emph{Segal condition} and \emph{univalence}, mimicking the definition of complete Segal spaces given by Rezk~\cite{rezk2001} as previously done by Lurie~\cite{lurie2009b} and Rasekh~\cite{rasekh2018}. Since complete Segal spaces provide a model for the $\infty$-category of $\infty$-categories, the latter can be identified with $\CatS$, and one obtains an equivalence $\Cat(\BB)\simeq\Shv_{\CatS}(\BB)$ between the $\infty$-category of $\BB$-categories and the $\infty$-category of $\CatS$-valued sheaves on $\BB$. Consequently, the sheaf $\Over{\BB}{-}$ that sends an object $A\in\BB$ to the slice $\infty$-category $\Over{\BB}{A}$ defines a (large) $\BB$-category $\Univ$ that we will refer to as the \emph{internal universe} of $\BB$ and which represents the reflection of $\BB$ within itself. 
We furthermore show that $\Cat(\BB)$ is cartesian closed, hence that for any two $\BB$-categories $\I{C}$ and $\I{D}$ there is a $\BB$-category $\iFun{\I{C}}{\I{D}}\in\Cat(\BB)$ of functors between $\I{C}$ and $\I{D}$. 

The first of the two main results in this paper is an internal version of the \emph{Grothendieck construction}: For any $\BB$-category $\I{C}$, we define a large $\BB$-category $\ILFib_{\I{C}}$ whose objects are given by left fibrations over $\I{C}$. We then show:
\theoremstyle{plain}
\newtheorem*{thm:GrothendieckConstruction}{Theorem~\ref{thm:internalGrothendieck}}
\begin{thm:GrothendieckConstruction}
	There is a canonical equivalence
	\begin{equation*}
		\iFun{\I{C}}{\Univ}\simeq \ILFib_{\I{C}}
	\end{equation*}
	that is natural in $\I{C}$.
\end{thm:GrothendieckConstruction}
By making use of the Grothendieck construction, we derive our second main result, an internal version of Yoneda's lemma. For any $\BB$-category $\I{C}$, we construct a mapping bifunctor $\map{\I{C}}(-,-)\colon\I{C}^{\op}\times\I{C}\to \Univ$ whose transpose gives rise to the Yoneda embedding $h\colon\I{C}\to\iFun{\I{C}^{\op}}{\Univ}$. Let $\ev\colon \I{C}^{\op}\times {\iFun{\I{C}^{\op}}{\Univ}}\to\Univ$ be the evaluation functor, i.e.\  the counit of the adjunction $\I{C}^{\op}\times- \dashv \iFun{\I{C}^{\op}}{-}$. We then show:
\newtheorem*{thm:Yoneda}{Theorem~\ref{thm:YonedaLemma}}
\begin{thm:Yoneda}
	There is a commutative diagram
	\begin{equation*}
		\begin{tikzcd}
			\I{C}^{\op}\times {\iFun{\I{C}^{\op}}{\Univ}}\arrow[dr, "\ev"'] 
			\arrow[r, "h\times \id"] & {\iFun{\I{C}^{\op}}{\Univ}^{\op}}\times{\iFun{\I{C}^{\op}}{\Univ}}\arrow[d, "{\map{\iFun{\I{C}^{\op}}{\Univ}}(-,-)}"]  \\
			& \Univ
		\end{tikzcd}
	\end{equation*}
	of (large) $\BB$-categories.
\end{thm:Yoneda}
In particular, Theorem~\ref{thm:YonedaLemma} implies that the Yoneda embedding $h$ is a fully faithful functor.

\subsection*{Related work}
The idea of developing category theory internal to some topos is not new; in as early as 1963, Lawvere formulated axioms for a theory of categories~\cite{lawvere1963}, and in subsequent years B\'enabou developed the theory of internal categories in a presheaf topos through the notion of fibred categories. For a detailed exposition of the theory of categories in a $1$-topos, the reader may consult~\cite{johnstone2002}. In the world of higher category theory, Riehl and Shulman~\cite{shulman2017} have proposed a synthetic approach to the theory of $\infty$-categories.  The theory of categories in an $\infty$-topos as presented herein is expected to provide categorical semantics to their type theory (cf.~\cite[Remark~A.14]{shulman2017}). We therefore believe that our version of Yoneda's lemma can be formally deduced from their work. However, our approach is different in that it is derived from the internal Grothendieck construction, which is not featured in~\cite{shulman2017} and which is a useful result in its own right. On the semantic side of the story, Barwick, Dotto, Glasman, Nardin and Shah~\cite{shah2016} have developed the theory of $\infty$-categories that are parametrised by a base $\infty$-category. Such parametrised $\infty$-categories are a special case of the more general notion of categories in an arbitrary $\infty$-topos; they precisely correspond to categories in presheaf $\infty$-topoi. Lastly, we note that Rasekh has previously worked out some aspects of the theory of internal higher categories in~\cite{rasekh2018}. 

\subsection*{Acknowledgments}
I would like to thank my advisor Rune Haugseng for his support and help throughout the process of writing this paper. I furthermore thank Simon Pepin Lehalleur for introducing me to the world of higher category theory and for his invaluable guidance during the early stages of this project. I would also like to thank Mathieu Anel for helpful discussions.

\section{Preliminaries}
\subsection{General conventions and notation}
Throughout this paper we freely make use of the language of higher category theory. We will generally follow a model-independent approach to higher categories. This means that as a general rule, all statements and constructions that are considered herein will be invariant under equivalences in the ambient $\infty$-category, and we will always be working within such an ambient $\infty$-category. For example, this means that all constructions involving $\infty$-categories and functors between $\infty$-categories will be assumed to take place in the $\infty$-category of $\infty$-categories. In the same vein, a set will be a discrete $\infty$-groupoid, and a $1$-category will be an $\infty$-category all of whose mapping $\infty$-groupoids are discrete. These conventions in particular imply that we will understand the adjective \emph{unique} in the homotopical sense, i.e.\  as the condition that there is a contractible $\infty$-groupoid of choices.

We denote by $\Delta$ the simplex category, i.e.\ the category of non-empty totally ordered finite sets with order-preserving maps. Every natural number $n\in\mathbb N$ can be considered as an object in $\Delta$ by identifying $n$ with the totally ordered set $\ord{n}=\{0,\dots n\}$. For $i=0,\dots,n$ we denote by $\delta^i\colon \ord{n-1}\to \ord{n}$ the unique injective map in $\Delta$ whose image does not contain $i$. Dually, for $i=0,\dots n$ we denote by $\sigma^i\colon \ord{n+1}\to \ord{n}$ the unique surjective map in $\Delta$ such that the preimage of $i$ contains two elements. Furthermore, if $S\subset n$ is an arbitrary subset of $k$ elements, we denote by $\delta^S\colon \ord{k}\to \ord{n}$ the unique injective map in $\Delta$ whose image is precisely $S$. In the case that $S$ is an interval, we will denote by $\sigma^S\colon \ord{n}\to \ord{n-k}$ the unique surjective map that sends $S$ to a single object. If $\CC$ is an $\infty$-category, we refer to a functor $C_{\bullet}\colon\Delta^{\op}\to\CC$ as a simplicial object in $\CC$. We write $C_n$ for the image of $n\in\Delta$ under this functor, and we write $d_i$, $s_i$, $d_S$ and $s_S$ for the image of the maps $\delta^i$, $\sigma^i$, $\delta^S$ and $\sigma^S$ under this functor. Dually, a functor $C^{\bullet}\colon \Delta\to\CC$ is referred to as a cosimplicial object in $\CC$. In this case we denote the image of $\delta^i$, $\sigma^i$, $\delta^S$ and $\sigma^S$ by $d^i$, $s^i$, $d^S$ and $\sigma^S$.

The $1$-category $\Delta$ embeds fully faithfully into the $\infty$-category of $\infty$-categories by means of identifying posets as $0$-categories and order-preserving maps between posets with functors between such $0$-categories. We denote by $\Delta^n$ the image of $n\in\Delta$ under this embedding.

\subsection{Set theoretical foundations}
Once and for all we will fix three Grothendieck universes $\bU\in\bV\in\bW$ that contain the first infinite ordinal $\omega$. A set is \emph{small} if it is contained in $\bU$, \emph{large} if it is contained in $\bV$ and \emph{very large} if it is contained in $\bW$. An analogous naming convention will be adopted for $\infty$-categories and $\infty$-groupoids. The large $\infty$-category of small $\infty$-groupoids is denoted by $\SS$, and the very large $\infty$-category of large $\infty$-groupoids by $\SSS$. The (even larger) $\infty$-category of very large $\infty$-groupoids will be denoted by $\SSSS$. Similarly, we denote the large $\infty$-category of small $\infty$-categories by $\CatS$, the very large $\infty$-category of large $\infty$-categories by $\CatSS$, and the even larger $\infty$-category of very large $\infty$-categories by $\CatSSS$.

\subsection{$\infty$-topoi}
A large $\infty$-category $\BB$ is said to be an $\infty$-topos if there exists a small $\infty$-category $\CC$ such that $\BB$ arises as a left exact and accessible localisation of $\PSh_{\SS}(\CC)=\Fun(\CC^{\op},\SS)$, see~\cite[\S~6]{htt} for alternative characterisations and the basic theory. An algebraic morphism between two $\infty$-topoi $\AA$ and $\BB$ is a functor $f^\ast\colon \AA\to\BB$ that commutes with small colimits and finite limits. We denote by $\Fun^\ast(\AA,\BB)$ the full subcategory of $\Fun(\AA,\BB)$ that is spanned by the algebraic morphisms between $\AA$ and $\BB$. Dually, a geometric morphism between $\infty$-topoi is a functor $f_\ast\colon \BB\to\AA$ that admits a left adjoint $f^\ast$ which defines an algebraic morphism (i.e.\ which commutes with finite limits). We denote by $\Fun_\ast(\BB,\AA)$ the full subcategory of $\Fun(\BB,\AA)$ spanned by the geometric morphisms. By the adjoint functor theorem, one finds $\Fun^\ast(\AA,\BB)^{\op}\simeq\Fun_\ast(\BB,\AA)$. We let $\RTop$ be the subcategory of $\CatSS$ that is spanned by the $\infty$-topoi and geometric morphisms, and we denote by $\LTop$ the subcategory of $\CatSS$ that is spanned by the $\infty$-topoi and algebraic morphisms. There is an equivalence $\RTop^{\op}\simeq\LTop$ that sends an $\infty$-topos to itself and a geometric morphism to its left adjoint. The $\infty$-category $\SS$ of small $\infty$-groupoids is a final object in $\RTop$; for any $\infty$-topos $\BB$ we denote by $\Gamma\colon \BB\to\SS$ the unique geometric morphism and refer to this functor as the \emph{global sections} functor. Explicitly, this functor is given by $\map{\BB}(1,-)$ where $1\in\BB$ denotes a final object. Dually, we denote the unique algebraic morphism from $\SS$ to $\BB$ by $\const\colon \SS\to\BB$ and refer to this map as the \emph{constant sheaf functor}.

\subsection{Universe enlargement}
\label{sec:universeEnlargement}
For any two large $\infty$-categories $\CC$ and $\AA$, an \emph{$\AA$-valued presheaf} on $\CC$ is a functor $\CC^{\op}\to\AA$ and an \emph{$\AA$-valued sheaf} an $\AA$-valued presheaf that preserves small limits (whenever they exist). We denote the $\infty$-category of $\AA$-valued presheaves on $\CC$ by $\PSh_{\AA}(\CC)$ and the full subcategory spanned by the $\AA$-valued sheaves on $\CC$ by $\Shv_{\AA}(\CC)$.

For any $\infty$-topos $\BB$, we define its \emph{universe enlargement} $\BBB=\Shv_{\SSS}(\BB)$ as the very large $\infty$-category of $\SSS$-valued sheaves on $\BB$. By~\cite[Remark~6.3.5.17]{htt} this is an $\infty$-topos relative to $\bV$. Moreover, one can make the assignment $\BB\mapsto \BBB$ into a functor as follows:

Consider the functor $\PSh_{\SSS}(-)\colon \RTop\to \CatSSS$ that acts by sending a map $f_\ast\colon \BB\to\AA$ in $\RTop$ to the map $(-)\circ f^\ast\colon \PSh_{\SSS}(\BB)\to\PSh_{\SSS}(\AA)$. Since $f^\ast$ commutes with small colimits, the functor $(-)\circ f^\ast$ restricts to a functor $\BBB\to\AAA$ that we will denote by $f_\ast$ as well. As a consequence, if $\int \PSh_{\SSS}(-)\to \RTop$ is the cocartesian fibration that is classified by the functor $\PSh_{\SSS}(-)$, then the full subcategory of $\int \PSh_{\SSS}(-)$ that is spanned by pairs $(\BB, A)$ with $\BB\in\RTop$ and $A\in \BBB\subset \PSh_{\SSS}(\BB)$ is stable under cocartesian arrows and therefore defines a cocartesian subfibration of $\int \PSh_{\SSS}(-)$ over $\RTop$. Moreover, by making use the adjunction $f^\ast\dashv f_\ast$ one obtains a commutative diagram
\begin{equation*}
\begin{tikzcd}
\BB \arrow[d, hookrightarrow] \arrow[r, "f_\ast"] & \AA \arrow[d, hookrightarrow]\\
\BBB\arrow[r, "\hat{f}_{\ast}"] & \AAA,
\end{tikzcd}
\end{equation*}
hence the same argumentation implies that the full subcategory of $\int \PSh_{\SSS}(-)$ spanned by pairs $(\BB, A)$ with $\BB\in\RTop$ and $A\in \BB$ defines a cocartesian subfibration of $\int\PSh_{\SSS}(-)$ over $\RTop$ too. Consequently, one obtains a functor
\begin{equation*}
\RTop\to \CatSSS,\quad \BB\mapsto \BBB
\end{equation*}
together with a natural transformation
\begin{equation*}
\begin{tikzcd}[column sep={3cm,between origins}]
\RTop\arrow[r, bend left,"\BB\mapsto \BB"{name=U}, start anchor=east, end anchor=west, shift left=.5em]\arrow[r, bend right, "\BB\mapsto \BBB"{name=L, below}, start anchor=east, end anchor=west, shift right=.5em]\arrow[from=U, to=L, Rightarrow, shorten=3mm] & \CatSSS
\end{tikzcd}
\end{equation*}
that is given by the inclusion $\BB\into\BBB$.

By~\cite[Remark~6.4.6.18]{htt} the functor $ f_\ast\colon\BBB\to\AAA$ defines a geometric morphism between $\infty$-topoi relative to the universe $\bV$, and the associated left adjoint $ f^\ast$ is obtained as the restriction of the functor of left Kan extension $(f^\ast)_!\colon\PSh_{\SSS}(\AA)\to\PSh_{\SSS}(\BB)$ to $\AAA$. Since the functor $\BB\mapsto \BBB$ above therefore takes values in the $\infty$-category $\BigRTop$ of $\infty$-topoi relative to the universe $\bV$, passing to opposite $\infty$-categories therefore results in a functor
\begin{equation*}
\LTop\to \CatSSS,\quad \BB\mapsto \BBB
\end{equation*}
that sends the geometric morphism $f^\ast\colon\AA\to\BB$ to $f^\ast\colon\AAA\to\BBB$. By construction, one furthermore obtains a commutative square
\begin{equation*}
\begin{tikzcd}
\BB \arrow[d, hookrightarrow] \arrow[from=r, "f^\ast"'] & \AA \arrow[d, hookrightarrow]\\
\BBB\arrow[from=r, "{f}^{\ast}"'] & \AAA,
\end{tikzcd}
\end{equation*}
hence an analogous argument as above shows that the inclusion $\BB\into\BBB$ defines a natural transformation
\begin{equation*}
\begin{tikzcd}[column sep={3cm,between origins}]
\LTop\arrow[r, bend left,"\BB\mapsto \BB"{name=U}, start anchor=east, end anchor=west, shift left=.5em]\arrow[r, bend right, "\BB\mapsto \BBB"{name=L, below}, start anchor=east, end anchor=west, shift right=.5em]\arrow[from=U, to=L, Rightarrow, shorten=3mm] & \CatSSS.
\end{tikzcd}
\end{equation*}
Note that if $f^\ast$ admits a further left adjoint $f_!$, then the map $ f^\ast\colon\AAA\to\BBB$ is given by precomposition with $f_!$.

\begin{remark}
	\label{rem:transitivityUniverseEnlargement}
	If $\BB$ is an $\infty$-topos, there are a priori two ways to define the universe enlargement $\BBBB$ relative to the universe $\bW$: either by applying the above construction to the pair $\bU\in\bW$, i.e.\ by defining $\BBBB=\smash{\Shv_{\SSSS}(\BB)}$, or by applying this construction first to the pair $\bU\in \bV$ and then to the pair $\bV\in\bW$, i.e.\ by setting $\BBBB=\smash{\Shv_{\SSSS}(\BBB})$, where the right-hand side now denotes the $\infty$-categories of functors $\BBB^{\op}\to\SSSS$ that commute with $\bV$-small limits. It turns out that either approach results in the same object: in fact, upon identifying $\bU$ with a regular cardinal in $\bV$, we may identify $\BB$ with the $\infty$-category $\widehat{\Ind}_{\bU}(\BB)$, hence~\cite[Proposition~5.3.5.10]{htt} implies that the inclusion $\BB\into\BBB$ induces an equivalence
	\begin{equation*}
	 \Fun^{\bU\text{-filt.}}(\BBB, \SSSS^{\op})\simeq \Fun(\BB, \SSSS^{\op})
	\end{equation*}
	in which the left-hand side denotes the full subcategory of $\Fun(\BBB,\SSSS^{\op})$ that is spanned by those functors that preserve $\bU$-filtered colimits. Now~\cite[Proposition~5.5.1.9]{htt} implies that the above equivalence restricts to an equivalence $\smash{\Shv_{\SSSS}(\BBB)\simeq\Shv_{\SSSS}(\BB)}$, noting that its proof does not require $\SSSS^{\op}$ to be presentable (relative to the universe $\bV$) but merely to admit $\bV$-small colimits.
\end{remark}

Recall that the assignment $A\mapsto \Over{\BB}{A}$ defines a fully faithful functor $\BB\into\Over{\RTop}{\BB}$. By~\cite[Corollary~9.9]{haugseng2017} there is a functorial equivalence
\begin{equation*}
\PSh_{\SSS}(\Over{\BB}{-})\simeq \Over{\PSh_{\SSS}(\BB)}{-}
\end{equation*}
of functors $\BB^{\op}\to \CatSSS$ that is given on each object $A\in \BB$ by the left Kan extension of the functor $\Over{\BB}{A}\to \Over{\PSh_{\SSS}(\BB)}{A}$ that is induced by the Yoneda embedding $\BB\into\PSh_{\SSS}(\BB)$ along the Yoneda embedding $\Over{\BB}{A}\into\PSh_{\SSS}(\Over{\BB}{A})$.

\begin{lemma}
	\label{lem:equivalenceSheavesSlices}
	For every $A\in \BB$, the equivalence $\PSh_{\SSS}(\Over{\BB}{A})\simeq \Over{\PSh_{\SSS}(\BB)}{A}$ restricts to an equivalence
	\begin{equation*}
	\widehat{\Over{\BB}{A}}\simeq \Over{\BBB}{A}.
	\end{equation*}
\end{lemma}
\begin{proof}
	In the commutative diagram
	\begin{equation*}
	\begin{tikzcd}
	& \widehat{\Over{\BB}{A}}\arrow[r, hookrightarrow]\arrow[d, dashed, "\phi"] & \PSh_{\SSS}(\Over{\BB}{A})\arrow[d, "\simeq"]\\
	\Over{\BB}{A}\arrow[r]\arrow[ur, hookrightarrow] & \Over{\BBB}{A}\arrow[r, hookrightarrow] & \Over{\PSh_{\SSS}(\BB)}{A}
	\end{tikzcd}
	\end{equation*}
	the task is to find the dashed arrow $\phi$ that completes the diagram and to show that this functor is an equivalence of $\infty$-categories. To that end, note that one has $\widehat{\Over{\BB}{A}}\simeq \widehat{\Ind}_{\bU}(\Over{\BB}{A})$, which by~\cite[Proposition~5.3.5.10]{htt} implies that composition with the Yoneda embedding gives rise to an equivalence
	\begin{equation*}
	\Fun^{\bU}(\widehat{\Over{\BB}{A}},\mathcal D)\simeq \Fun(\Over{\BB}{A},\mathcal D)
	\end{equation*}
	for any $\infty$-category $\DD$ which admits $\bU$-small filtered colimits. Here the left-hand side denotes the $\infty$-category of $\bU$-continuous functors, i.e.\ of those functors that commute with $\bU$-filtered colimits. This result implies that the functor $\phi$ in the diagram above is well-defined and makes the diagram indeed commute. By construction, $\phi$ must be fully faithful. On the other hand, combining the equivalence of $\infty$-categories $\BBB\simeq\widehat{\Ind}_{\bU}(\BB)$ with the fact that the projection $\Over{\BBB}{A}\to\BBB$ creates colimits shows that every object in $\Over{\BBB}{A}$ is obtained as the colimit of a functor $\JJ\to \Over{\BB}{A}\to \Over{\BBB}{A}$ where $\JJ$ is a $\bU$-filtered $\infty$-category. Since $\phi$ commutes with $\bU$-filtered colimits, this shows that this functor is essentially surjective.
\end{proof}
\begin{proposition}
	\label{prop:equivalenceSheavesSlices}
	There is a canonical equivalence
	\begin{equation*}
	\widehat{\Over{\BB}{-}}\simeq \Over{\BBB}{-}
	\end{equation*}
	of functors $\BB^{\op}\to\CatSSS$.
\end{proposition}
\begin{proof}
	By Lemma~\ref{lem:equivalenceSheavesSlices}, the functorial equivalence
	\begin{equation*}
	\PSh_{\SSS}(\Over{\BB}{-})\simeq \Over{\PSh_{\SSS}(\BB)}{-}
	\end{equation*}
	restricts objectwise to an equivalence on the level of sheaves, hence the result follows.
\end{proof}
We finish this section by discussing the preservation of structure under universe enlargement:
\begin{proposition}
	\label{prop:universeEnlargementStructurePreservation}
	For any $\infty$-topos $\BB$ the inclusion $\BB\into\BBB$ commutes with small limits and colimits, and the internal mapping bifunctor on $\BBB$ restricts to the internal mapping bifunctor on $\BB$.
\end{proposition}
\begin{proof}
	Since the Yoneda embedding $h\colon\BB\into\PSh_{\SSS}(\BB)$ commutes with small limits and since $\BBB$ is a localisation of $\PSh_{\SSS}(\BB)$, the embedding $\BB\into\BBB$ preserves small limits. The case of small colimits is proved in~\cite[Remark~6.3.5.17]{htt}. Lastly, since the product bifunctor on $\BBB$ restricts to the product bifunctor on $\BB$, it suffices to show that for $A,B\in\BB\into\BBB$ their internal mapping object $\iFun{A}{B}\in\BBB$ is contained in $\BB$. Now $\BBB$ is a left exact localisation of $\PSh_{\SSS}(\BB)$ and therefore an exponential ideal in $\PSh_{\SSS}(\BB)$, hence it suffices to show that the internal mapping object $\iFun{h(A)}{h(B)}\in\PSh_{\SSS}(\BB)$ is representable. This follows from the computation
	\begin{equation*}
		\iFun{h(A)}{h(B)}\simeq\map{\PSh_{\SSS}(\BB)}(h(-)\times h(A), h(B))\simeq\map{\BB}(-\times A, B)\simeq\map{\BB}(-, \iFun{A}{B})
	\end{equation*}
	in which we make repeated use of the Yoneda lemma and in which the object $\iFun{A}{B}\in\BB$ on the right-hand side denotes the internal mapping object in $\BB$.
\end{proof}

\subsection{Factorisation systems}
\label{sec:factorisationSystems}
Let $\CC$ be an $\infty$-category. Given two maps $f\colon a\to b$ and $g\colon c\to d$ in $\CC$, we say that $f$ and $g$ are orthogonal if the commutative square
\begin{equation*}
\begin{tikzcd}
\map{\CC}(b, c)\arrow[r, "g_\ast"]\arrow[d, "f^\ast"] & \map{\CC}(b, d)\arrow[d, "f^\ast"]\\
\map{\CC}(a, c)\arrow[r, "g^\ast"]& \map{\CC}(a,d)
\end{tikzcd}
\end{equation*}
is cartesian. We denote the orthogonality relation between $f$ and $g$ by $f\bot g$, and we will say that $f$ is left orthogonal to $g$ and $g$ is right orthogonal to $f$. In particular, $f$ and $g$ being orthogonal implies that any lifting square
\begin{equation*}
\begin{tikzcd}
a\arrow[r]\arrow[d, "f"] & c\arrow[d, "g"]\\
b\arrow[r]\arrow[ur, dotted] & d
\end{tikzcd}
\end{equation*}
has a unique solution.
If $\CC$ is cartesian closed, we furthermore say that $f$ and $g$ are \emph{internally} orthogonal if the square
\begin{equation*}
\begin{tikzcd}
{[b,c]}\arrow[r, "g_\ast"]\arrow[d, "f^\ast"] & {[b,d]}\arrow[d, "f^\ast"]\\
{[a,c]} \arrow[r, "g^\ast"]& {[a,d] }
\end{tikzcd}
\end{equation*}
is cartesian. By definition, this is equivalent to $c\times f\bot g$ for every $c\in \CC$. We will denote the internal orthogonality relation between $f$ and $g$ by $f\ibot g$.

If $\CC$ has a terminal object $1\in \CC$, then an object $c\in\CC$ is said to be \emph{local} with respect to the map $f\colon a\to b$ in $\CC$ if the terminal map $\pi_c\colon c\to 1$ is right orthogonal to $f$, i.e.\ if $f\bot \pi_c$ holds. Similarly, $c$ is \emph{internally local} with respect to $f$ if $f\ibot \pi_c$ holds.

If $S$ is an arbitrary family of maps in $\CC$, we will denote by $S^{\bot}$ the collection of maps that are right orthogonal to any map in $S$, and by $\prescript{\bot}{}{S}$ the collection of maps that are left orthogonal to any map in $S$.
\begin{definition}
	Let $\CC$ be an $\infty$-category. A \emph{factorisation system} is a pair $(\LL,\RR)$ of families of maps in $\CC$ such that
	\begin{enumerate}
		\item Any map $f$ in $\CC$ admits a factorisation $f\simeq rl$ with $r\in \RR$ and $l\in \LL$.
		\item $\LL^{\bot}=\RR$ as well as $\prescript{\bot}{}{\RR}=\LL$.
	\end{enumerate}
\end{definition}
The following proposition summarises some properties of factorisation systems whose proof is a straightforward consequence of the definition:
\begin{proposition}
	\label{prop:propertiesFactorisationSystems}
	Let $\CC$ be an $\infty$-category and let $(\LL,\RR)$ be a factorisation system in $\CC$. Then
	\begin{enumerate}
		\item The intersection $\LL\cap\RR$ is precisely the collection of equivalences in $\CC$;
		\item if $g\in \LL$, then $fg\in \LL$ if and only $f\in \LL$; dually, if $f\in \RR$ then $fg\in\RR$ if and only if $g\in \RR$;
		\item $\RR$ is stable under pullbacks and $\LL$ is stable under pushouts;
		\item both $\RR$ and $\LL$ are stable under taking retracts.
		\item $\RR$ is stable under all limits that exist in $\Fun(\Delta^1,\CC)$, and dually $\LL$ is stable under all colimits that exist in $\Fun(\Delta^1,\CC)$.\qed
	\end{enumerate}
\end{proposition}
If $(\LL,\RR)$ is a factorisation system in an $\infty$-category $\CC$, then $\RR$ defines a full subcategory of the arrow $\infty$-category $\Fun(\Delta^1, \CC)$. The factorisation of maps in $\CC$ then defines a left adjoint to this inclusion~\cite[Lemma~5.2.8.19]{htt}. More precisely, if $f\colon d\to c$ is a map in $\CC$ and if $rl\colon d\to e\to c$ is the factorisation of $f$ into maps $l\in \LL$ and $r\in \RR$, then the assignment $f\mapsto r$ extends to a functor $\Fun(\Delta^1,\CC)\to \RR$ that is left adjoint to the inclusion. The unit of this adjunction is then given by the square
\begin{equation*}
\begin{tikzcd}
d\arrow[r, "l"]\arrow[d, "f"] & e\arrow[d, "r"]\\
c\arrow[r, "\id"] & c.
\end{tikzcd}
\end{equation*}
By dualisation, this also shows that the inclusion $\LL\into \Fun(\Delta^1, \CC)$ admits a right adjoint.

Note that the fact that the inclusion $\RR\into\Fun(\Delta^1,\CC)$ admits a left adjoint moreover proves that for any $c\in \CC$ the induced inclusion $\Over{\RR}{c}\into\Over{\CC}{c}$ is reflective. If $f\colon d\to c$ is an object in $\Over{\CC}{c}$ and if $f\simeq rl$ is its factorisation, then the map $r$ is the image of $f$ under the localisation functor and the map $l$ is the component of the counit of the adjunction at $f$. In particular, this shows that the map $f\colon d\to c$ is contained in $\LL$ if and only if it is sent to the terminal object by the localisation functor $\Over{\CC}{c}\to \Over{\RR}{c}$, and the essential image of the inclusion $\Over{\RR}{c}\into \Over{\CC}{c}$ is spanned by those objects in $\Over{\CC}{c}$ that are local with respect to the class of maps in $\Over{\CC}{c}$ that are sent to $\LL$ via the projection functor $(\pi_c)_!\colon \Over{\CC}{c}\to \CC$.

\begin{remark}
	\label{rem:localEquivalences}
	Suppose that $(\LL,\RR)$ is a factorisation system in an $\infty$-category $\CC$ that has a final object $1\in\CC$, and let $f\colon c\to d$ be a map in $\CC$. Let $L\colon \CC\to \Over{\RR}{1}$ be a left adjoint to the inclusion, and consider the commutative diagram
	\begin{equation*}
		\begin{tikzcd}
		c\arrow[d]\arrow[r, "f"] & d\arrow[d]\\
		L(c)\arrow[d]\arrow[r, "L(f)"] & L(d) \arrow[d] \\
		1\arrow[r, "\id"] & 1
		\end{tikzcd}
	\end{equation*}
	in which the two vertical compositions are determined by the factorisation of the two terminal maps $\pi_c\colon c\to 1$ and $\pi_d\colon d\to 1$ into maps in $\LL$ and $\RR$. If $f$ is contained in $\LL$, then item~(2) of Proposition~\ref{prop:propertiesFactorisationSystems} implies that $L(f)$ must be contained in $\LL$ as well. On the other hand, item~(2) of Proposition~\ref{prop:propertiesFactorisationSystems} also implies that $L(f)$ is contained in $\RR$. Hence $L(f)$ must be an equivalence. In other words, the functor $L$ sends maps in $\LL$ to equivalences in $\Over{\RR}{1}$. The converse is however not true in general, i.e.\ not every map that is sent to an equivalence by $L$ must necessarily be contained in $\LL$. A notable exception is the case where $\pi_d\colon d\to 1$ is already contained in $\LL$. In this case, the map $d\to L(d)$ is an equivalence, hence $L(f)$ being an equivalence does imply that $f$ is contained in $\LL$.
\end{remark}
Lastly, let us discuss how a factorisation system can be \emph{generated} by a set of maps:
\begin{proposition}[{\cite[Proposition~5.5.5.7]{htt}}]
	\label{prop:factorisationSystemGenerated}
	Let $\CC$ be a presentable $\infty$-category and let $S$ be a small set of maps in $\CC$. Then there is a factorisation system $(\LL,\RR)$ in $\CC$ with $\RR=S^{\bot}$ and $\LL=\prescript{\bot}{}{\RR}$.\qed
\end{proposition}
In the situation of Proposition~\ref{prop:factorisationSystemGenerated}, the assignment $S\mapsto \LL$ can be viewed as a certain closure operation that is referred to as \emph{saturation}. Recall the definition of a saturated class:
\begin{definition}
	Let $\CC$ be a presentable $\infty$-category and let $S$ be a class of maps in $\CC$. Then $S$ is \emph{saturated} if 
	\begin{enumerate}
		\item $S$ contains all equivalences in $\CC$ and is closed under composition;
		\item $S$ is closed under small colimits in $\Fun(\Delta^1,\CC)$;
		\item $S$ is closed under pushouts.
	\end{enumerate}
\end{definition}
By Proposition~\ref{prop:propertiesFactorisationSystems}, the left class in any factorisation system is saturated. Now if $S$ is a small set of maps and if $(\LL,\RR)$ is the induced factorisation system in $\CC$ as provided by Proposition~\ref{prop:factorisationSystemGenerated}, then $\LL$ is the \emph{universal} saturated class that contains $S$, in the following sense:
\begin{proposition}
	\label{prop:saturatedClosure}
	Let $\CC$ be a presentable $\infty$-category and let $S$ be a small set of maps in $\CC$. Let $(\LL,\RR)$ be the associated factorisation system. Then $\LL$ is the smallest saturated class of maps that contains $S$.
\end{proposition}
\begin{proof}
	To begin with, note that the property of a class being saturated is preserved under taking arbitrary intersections, hence the \emph{smallest} saturated class containing $S$ is well-defined and is explicitly given by the intersection
	\begin{equation*}
	\overline S = \bigcap_{S\subset T} T
	\end{equation*}
	over all saturated classes of maps that contain $S$. We need to show that any saturated $T\supset S$ contains $\LL$ as well. But since $\LL=\prescript{\bot}{}{\RR}$ and as $(\LL, \RR)$ is a factorisation system, this is equivalent to $T^{\bot}\subset \RR$, which in turn follows immediately from $S\subset T$ and $S^{\bot}=\RR$.
\end{proof}

An analogous construction can be carried out when replacing orthogonality by \emph{internal orthogonality} in the case where $\CC$ is cartesian closed:
\begin{proposition}[{\cite[Proposition~3.2.9]{anel2020}}]
	\label{prop:factorisationSystemInternallyGenerated}
	Let $\CC$ be a presentable and cartesian closed $\infty$-category and let $S$ be a small set of maps in $\CC$. Then there is a factorisation system $(\LL,\RR)$ in $\CC$ such that $\RR= S^{\ibot}$ and $\LL=\prescript{\ibot}{}{\RR}=\prescript{\bot}{}{\RR}$.\qed
\end{proposition}
Since a map $r$ in a cartesian closed $\infty$-category $\CC$ is internally right orthogonal to a map $l$ if and only if $r$ is right orthogonal to $c\times l$ for any $c\in \CC$, Proposition~\ref{prop:factorisationSystemGenerated} implies:
\begin{proposition}
	\label{prop:saturatedClosureInternal}
	Let $\CC$ be a presentable and cartesian closed $\infty$-category and let $S$ be a small set of maps in $\CC$. Let $(\LL, \RR)$ be the factorisation system provided by Proposition~\ref{prop:factorisationSystemInternallyGenerated}. Then $\LL$ is the smallest saturated class of maps that contains the set $\{c\times f~\vert~c\in \CC,~f\in S\}$.\qed
\end{proposition}

\begin{example}
	\label{ex:effectiveEpimorphismMonomorphism}
	Let $\CC$ be a presentable and cartesian closed $\infty$-category. We say that a map $f\colon c\to d$ is a \emph{monomorphism} if it is internally right orthogonal to the codiagonal $1\sqcup 1\to 1$, where $1$ is the final object in $\CC$. By construction, a map in $\CC$ is a monomorphism if and only if its diagonal is an equivalence. Dually, we say that a map is a \emph{strong epimorphism} if it is internally left orthogonal to every monomorphism. By Proposition~\ref{prop:factorisationSystemInternallyGenerated}, one obtains a factorisation system in which the left class are strong epimorphisms and the right class are monomorphisms. If $\CC$ is an $\infty$-topos, then strong epimorphisms are precisely \emph{covers}, i.e.\ effective epimorphisms in the terminology of~\cite{htt}. This follows from the fact that covers and monomorphisms form a factorisation system in any $\infty$-topos~\cite[Example~5.2.8.16]{htt}, combined with the fact that the left class of a factorisation system is uniquely determined by the right class.
\end{example}

\section{$\BB$-categories}
In this chapter, we introduce the language and some basic concepts of the theory of categories internal to an $\infty$-topos $\BB$. We will confine ourselves to discussing only those aspects of the theory that will be needed for the discussion of the Grothendieck construction and Yoneda's lemma in chapter~\ref{chap:yoneda}. We set up the general framework of higher internal category theory in \S~\ref{sec:simpObjects}--\ref{sec:categoriesAsSheaves}. In \S~\ref{sec:objectsMorphisms} we feature a brief discussion of the objects and morphisms of a $\BB$-category. We define and investigate the universe $\Univ$ for $\BB$-groupoids in \S~\ref{sec:universe}, and finally \S~\ref{sec:fullyFaithful} and \S~\ref{sec:subcategories} contain a discussion of fully faithful functors and full subcategories.

\subsection{Simplicial objects in an $\infty$-topos}
\label{sec:simpObjects}
Let $\BB$ be an arbitrary $\infty$-topos and let $\Simp{\BB}$ denote the $\infty$-topos of simplicial objects in $\BB$. By postcomposition with the adjunction $(\const\dashv \Gamma)\colon\BB\to\SS$ one obtains an induced adjunction $(\const\dashv\Gamma)\colon \Simp{\BB}\to \Simp\SS$ on the level of simplicial objects. This defines a functor
\begin{equation*}
	\Simp{(-)}\colon \RTop\to \Over{\RTop}{{\Simp\SS}}
\end{equation*}
from the $\infty$-category of $\infty$-topoi with geometric morphisms as maps to the slice $\infty$-category of $\infty$-topoi over $\Simp\SS$.

We define the \emph{tensoring} of $\Simp{\BB}$ over $\SS_{\Delta}$ as the bifunctor
\begin{equation*}
-\otimes -\colon \Simp\SS\times\Simp\BB\to\Simp\BB
\end{equation*}
that is given by the composition $(-\times -)\circ (\const\times \id_{\Simp\BB})$. Dually, we define the \emph{powering} of $\Simp\BB$ over $\SS_{\Delta}$ as the bifunctor
\begin{equation*}
(-)^{(-)}\colon \SS_{\Delta}^{\op}\times\Simp\BB\to \Simp\BB
\end{equation*}
that is given by the composition $\iFun{-}{-}\circ (\const\times\id_{\Simp\BB})$, where $\iFun{-}{-}$ denotes the internal mapping object in $\Simp{\BB}$. Let $\Fun_{\BB}(-,-)\colon\Simp\BB^{\op}\times\Simp\BB\to\Simp\SS$ be the bifunctor given by $\Gamma\circ\iFun{-}{-}$. We then obtain equivalences
\begin{equation*}
\map{\Simp\BB}(-, (-)^{(-)})\simeq \map{\Simp\BB}(-\otimes-, -)\simeq \map{\SS_{\Delta}}(-, \Fun_{\BB}(-,-)).
\end{equation*}

\begin{remark}
	\label{rem:InternalHomConstantSimplicialObjects}
	For any object $A$ in the $\infty$-topos $\BB$, we may regard $A$ as a constant simplicial object via the diagonal functor $\BB\into \BB_{\Delta}$. Since products in $\Simp{\BB}$ are computed objectwise, the endofunctor $A\times -$ on $\Simp{\BB}$ is equivalent to the functor that is given by postcomposing simplicial objects with the product functor $A\times -\colon \BB\to \BB$. The latter admits a right adjoint $\iFun{A}{-}\colon \BB\to \BB$, and since postcomposition with an adjunction induces an adjunction on the level of functor $\infty$-categories, the uniqueness of adjoints implies that the internal mapping objects functor $\iFun{A}{-}\colon \Simp{\BB}\to\Simp{\BB}$ is obtained by applying the internal mapping objects functor of $\BB$ levelwise to simplicial objects in $\BB$. More precisely, the restriction of the internal mapping object bifunctor $\iFun{-}{-}$ on $\Simp{\BB}$ along the inclusion $\BB^{\op}\times\Simp{\BB}\into \Simp{\BB}^{\op}\times\Simp{\BB}$ is equivalent to the transpose of the composite functor
	\begin{equation*}
	\BB^{\op}\times\Simp{\BB}\times\Delta^{\op}\xrightarrow{\id_{\BB^{\op}}\times\ev_{\Delta^{\op}}} \BB^{\op}\times\BB\xrightarrow{\iFun{-}{-}} \BB
	\end{equation*}
	in which $\ev_{\Delta^{\op}}$ denotes the evaluation functor (i.e\ the counit of the adjunction $-\times \Delta^{\op}\dashv \Fun(\Delta^{\op}, -)$) and in which the second functor is the internal mapping object bifunctor on $\BB$.
	
	In particular, this argument shows that the internal mapping objects bifunctor on $\Simp{\BB}$ restricts to the internal mapping objects bifunctor on $\BB$, which justifies our choice of notation.
\end{remark}

Observe that the diagonal functor $\iota\colon\BB\into \Simp{\BB}$ admits a right adjoint that is given by the evaluation functor $(-)_0\colon \Simp{\BB}\to\BB$. Let us denote by $\Delta^{\bullet}\colon\Delta\into\SS_{\Delta}$ the Yoneda embedding. Restricting the powering bifunctor along $\Delta^{\bullet}$ then defines a functor $(-)^{\Delta^{\bullet}}\colon\Simp{\BB}\to \PSh_{\Simp{\BB}}(\Delta)$. The computation
\begin{align*}
	\map{\BB}(-, ((-)^{\Delta^{\bullet}})_0)&\simeq\map{\Simp{\BB}}(\iota(-), (-)^{\Delta^\bullet})\\
	&\simeq \map{\SS_{\Delta}}(\Delta^{\bullet}, \Fun_{\Simp{\BB}}(\iota(-),-))\\
	&\simeq \Gamma\iFun{\iota(-)}{(-)_{\bullet}}\\
	&\simeq \map{\BB}(-, (-)_{\bullet})
\end{align*}
in which the penultimate equivalence follows from Yoneda's lemma and Remark~\ref{rem:InternalHomConstantSimplicialObjects} now shows :
\begin{proposition}
	\label{prop:simplicialPowering}
	The composite functor
	\begin{equation*}
	\Simp{\BB}\xrightarrow{(-)^{\Delta^{\bullet}}} \PSh_{\Simp{\BB}}(\Delta)\xrightarrow{(-)_0} \Simp{\BB}
	\end{equation*}
	in which the second arrow denotes postcomposition with the evaluation functor $(-)_0\colon \Simp{\BB}\to \BB$ is equivalent to the identity functor on $\Simp{\BB}$.\qed
\end{proposition}

\begin{lemma}
	\label{lem:iteratedLocalisationGenerators}
	Let $S$ be a saturated class of maps in $\BB$ that contains the maps $s^0\colon\Delta^1\otimes C\to C$ for every $C\in\Simp\BB$. Then $S$ contains the projection $\Delta^n\otimes A\to  A$ for every $n\geq 0$ and every $A\in \BB$.
\end{lemma}
\begin{proof}
	We will use induction on $n$, the case $n=1$ being true by assumption. Let us therefore assume that for an arbitrary integer $n\geq 1$ the projection $s^0\colon\Delta^n\otimes A\to A$ is contained in $S$. Then the composition $(\Delta^1\times\Delta^n)\otimes A\to \Delta^n\otimes A\to A$ in which the first map is induced by $s^0\colon\Delta^1\to\Delta^0$ is contained in $S$ as well. Let $\alpha\colon \Delta^{n+1}\to \Delta^1\times\Delta^n$ be the map that is defined by $\alpha(0)=(0,0)$ and $\alpha(k)=(1,k-1)$ for all $1\leq k\leq n+1$, and let $\beta\colon \Delta^1\times\Delta^n\to \Delta^{n+1}$ be defined by $\beta(0,k)=0$ and $\beta(1,k)=k+1$ for all $0\leq k\leq n$. Then the composition $\beta\alpha$ is equivalent to the identity on $\Delta^{n+1}$, and we therefore obtain a retract diagram
	\begin{equation*}
		\begin{tikzcd}
			\Delta^{n+1}\otimes A\arrow[d] \arrow[r, "\alpha\otimes\id"] & (\Delta^1\times\Delta^n)\otimes A\arrow[d]\arrow[r, "\beta\otimes\id"] & \Delta^{n+1}\otimes A\arrow[d]\\
			A\arrow[r, "\id"] & A\arrow[r, "\id"] & A
		\end{tikzcd}
	\end{equation*}
	which shows that the map $\Delta^{n+1}\otimes A\to A$ is contained in $S$ as well.
\end{proof}

\begin{proposition}
	\label{prop:characterisationGroupoids}
	For any simplicial object $C\in \Simp\BB$, the following are equivalent:
	\begin{enumerate}
		\item $C$ is contained in the essential image of the diagonal functor $\iota\colon \BB\into\Simp\BB$.
		\item $C$ is internally local with respect to the map $s^0\colon\Delta^1\to\Delta^0$.
	\end{enumerate}
\end{proposition}
\begin{proof}
	If $C$ is internally local with respect to $s^0\colon\Delta^1\to\Delta^0$, then Lemma~\ref{lem:iteratedLocalisationGenerators} and Proposition~\ref{prop:simplicialPowering} imply that the simplicial maps $C_0\to C_n$ are equivalences in $\BB$ for all $n\geq 0$, which implies that $C$ is a constant simplicial object and therefore in the essential image of the diagonal functor. Conversely, let $A\in\BB$ be an arbitrary object. By making use of the adjunction $\colim_{\Delta^{\op}}\dashv \iota$ and the fact that the functor $\colim_{\Delta^{\op}}$ commutes with finite products as $\Delta^{\op}$ is a sifted category, the object $\iota(A)$ is internally local to $s^0$ whenever $A$ is internally local to $\colim_{\Delta^{\op}}(s^0)\colon \colim_{\Delta^{\op}}\Delta^1\to \colim_{\Delta^{\op}}\Delta^0$ in $\BB$. Since the colimit of a representable presheaf on a small $\infty$-category is always the final object $1\in\SS$, the latter map must be an equivalence, hence the result follows.
\end{proof}

\begin{definition}
	A simplicial object in $\BB$ is said to be a \emph{$\BB$-groupoid} if it satisfies the equivalent conditions from Proposition~\ref{prop:characterisationGroupoids}. We denote by $\Grpd(\BB)\into\Simp\BB$ the full subcategory spanned by the $\BB$-groupoids.
\end{definition}

\subsection{Categories in an $\infty$-topos}
We now proceed by defining a $\BB$-category to be a simplicial object in $\BB$ that satisfies the \emph{Segal conditions} and \emph{univalence}. To that end, let us recall the following combinatorial constructions:
\begin{definition}
	\label{def:walkingMorphismEquivalence}
	For any $n\geq 1$, let $I^n=\Delta^1\sqcup_{\Delta^0}\cdots\sqcup_{\Delta^0}\Delta^1\subset\Delta^n$ denote the \emph{spine} of $\Delta^n$, i.e\ the simplicial sub-$\infty$-groupoid of $\Delta^n$ that is spanned by the inclusions $d^{\{i-1,i\}}\colon\Delta^1\subset\Delta^n$ for $i=1,\dots,n$.
	
	Furthermore, let $E^1$ be the simplicial $\infty$-groupoid that is defined by the pushout square
	\begin{equation*}
	\begin{tikzcd}
	\Delta^1\sqcup \Delta^1\arrow[r, "s^0\sqcup s^0"] \arrow[d, "d^{\{0,2\}}\sqcup d^{\{1,3\}}"']\arrow[dr, "\lrcorner", very near end, phantom] & \Delta^0\sqcup \Delta^0\arrow[d]\\
	\Delta^3\arrow[r] & E^1.
	\end{tikzcd}
	\end{equation*}
	We refer to $E^1$ as the \emph{walking equivalence}.
\end{definition}
\begin{remark}
	Many authors define the walking equivalence as the simplicial set that arises as the nerve of the category with two objects and a unique isomorphism between them. The simplicial set $E^1$ from Definition~\ref{def:walkingMorphismEquivalence} is different in that it is comprised of a map together with \emph{separate} left and right inverses.
\end{remark}
\begin{remark}
	\label{rem:HomotopyColimitsSimplicialSets}
	In the situation of Definition~\ref{def:walkingMorphismEquivalence}, the colimits that define $I^n$ and $E^1$ can be computed as the ordinary (i.e\ $1$-categorical) pushouts in the $1$-category of simplicial sets. Indeed, colimits in $\Simp\SS$ are computed levelwise and so are ordinary colimits in $\Simp\Set$, hence it suffice to show that for any integer $k\geq 0$ the colimits in $\SS$ of the diagrams in $\Set\into\SS$ that define $I^n_k$ and $E^1_k$ can be computed by the ordinary colimits of these diagrams in $\Set$. As the Quillen model structure on the $1$-category of simplicial sets is left proper, ordinary pushouts along monomorphisms of simplicial sets present homotopy pushouts in $\SS$. Now $I^n$ is an iterated pushout along monomorphisms in $\Simp\SS$, hence $I^n_k$ is an iterated pushout along monomorphisms in $\Set\into\SS$, hence the claim follows for $I^n$. Regarding the simplicial $\infty$-groupoid $E^1$, note that $E^1$ fits into the commutative diagram
	\begin{equation*}
	\begin{tikzcd}
	& \Delta^1\arrow[r]\arrow[d, hookrightarrow,"d^{\{1,3\}}"] \arrow[dr, phantom, "\lrcorner", very near end]& \Delta^0\arrow[d]\\
	\Delta^1\arrow[r, hookrightarrow, "d^{\{0,2\}}"] \arrow[d]\arrow[dr, phantom, "\lrcorner", very near end]& \Delta^3\arrow[r] \arrow[d]\arrow[dr, phantom, "\lrcorner", very near end]& \Delta^3\sqcup_{\Delta^1}\Delta^0 \arrow[d]\\
	\Delta^0\arrow[r] & \Delta^0\sqcup_{\Delta^1}\Delta^3\arrow[r] & E^1
	\end{tikzcd}
	\end{equation*}
	of simplicial $\infty$-groupoids. Since the pushouts in the lower left and in the upper right corner are computed by the ordinary pushouts of simplicial sets, the claim now follows from the straightforward observation that the composition $\Delta^1\into\Delta^3\to \Delta^3\sqcup_{\Delta^1}\Delta^0$ is a monomorphism.
\end{remark}

\begin{definition}
	An object $C\in \Simp\BB$ is a \emph{$\BB$-category} if
	\begin{description}
		\item[Segal conditions] $C$ is internally local with respect to the map $I^2\into\Delta^2$, and
		\item[univalence] $C$ is internally local with respect to the map $E^1\to \Delta^0$.
	\end{description}
	 We denote the full subcategory of $\Simp\BB$ that is spanned by the $\BB$-categories by $\Cat(\BB)$.
\end{definition}

\begin{lemma}
	\label{lem:internalExternalSegalConditions}
	The following sets generate the same saturated class of maps in $\Simp\BB$:
	\begin{enumerate}
	\item $S=\{I^2\otimes K\into \Delta^2\otimes K~\vert~K\in\Simp\BB\}$;
	\item $T=\{I^n\otimes A\into\Delta^n\otimes A~\vert~n\geq2,~A\in\BB\}$;
	\item $U=\{I^n\otimes K\into\Delta^n\otimes K~\vert~n\geq2,~K\in\Simp\BB\}$.
	\end{enumerate}
\end{lemma}
\begin{proof}
	Let $\CC$ be a small $\infty$-category such that $\BB$ arises as a left exact accessible localisation of $\PSh_{\SS}(\CC)$. Then $\Simp{\BB}$ is a localisation of $\PSh_{\SS}(\Delta\times\CC)$, which means that any simplicial object $K\in\Simp{\BB}$ is a small colimit of objects in $\Delta\times\CC$. We may therefore always assume without loss of generality that $K$ is of the form $\Delta^k\otimes A$ for some $A\in \BB$ and some $k\geq 0$. Moreover, note that if $i\colon K\to L$ is a map in $\Simp\SS$, a map $f\colon C\to D$ in $\Simp\BB$ is right orthogonal to $i\otimes\id_{A}$ if and only if $f_\ast\colon \Fun_{\BB}(A, C)\to\Fun_{\BB}(A,D)$ is right orthogonal to $i$. This allows us to reduce to the case where $\BB\simeq \SS$, where $A\simeq 1$ and where $K\simeq\Delta^k$ for an arbitrary $k\geq 0$.
	
	We now claim that a map in $\Simp\SS$ is contained in the saturation $\overline{T}$ of $T$ if and only if it is contained in the saturation $\overline{T}^\prime$ of the set
	\begin{equation*}
	T^\prime=\{\Lambda^n_i\into\Delta^n~\vert~ n\geq 2,0<i<n\}
	\end{equation*}
	of inner horn inclusions. In fact, by~\cite[Proposition~2.13]{joyal2008} the spine inclusions $I^n\into\Delta^n$ are contained in $\overline{T}^\prime$, and the converse is proved in~\cite[Lemma~3.5]{Joyal2007}. As a consequence, every inner anodyne map between simplicial sets is contained in $\overline{T}$. Therefore,~\cite[Proposition~2.3.2.4]{lurie2009b} implies that for any $n\geq 2$ and any $k\geq 0$, the map $I^n\times \Delta^k\into\Delta^n\times\Delta^k$ is contained in $\overline{T}$ as well. To complete the proof, we therefore only need to show that $\overline{T}$ is contained in the saturation $\overline{S}$ of $S$. To that end, observe that $\overline{S}$ contains all maps of the form $I^2\times\partial\Delta^k\into \Delta^2\times\partial \Delta^k$ and therefore all maps of the form
	\begin{equation*}
		(I^2\times\Delta^k)\sqcup_{I^2\times \partial\Delta^k} (\Delta^2\times\partial\Delta^k)\into\Delta^2\times\Delta^k
	\end{equation*}
	as well. By~\cite[Proposition~2.3.2.1]{htt}, this implies that $T^\prime$ is contained in $\overline{S}$ as well, hence the claim follows.
\end{proof}
\begin{lemma}
	\label{lem:internalExternalUnivalence}
	Let $S$ be a strongly saturated class of maps in $\BB$ that contains the maps $E^1\otimes A\to A$ and $I^n\otimes A\into\Delta^n\otimes A$ for all $A\in \BB$ and all $n\geq 0$. Then $S$ contains the maps $E^1\otimes C\to C$ for all $C\in \Simp\BB$. 
\end{lemma}
\begin{proof}
	Using a similar argumentation as in Lemma~\ref{lem:internalExternalSegalConditions}, we may assume without loss of generality that $\BB\simeq \SS$, that $A$ is the final simplicial $\infty$-groupoid and furthermore that $C\simeq\Delta^n$. As furthermore by Lemma~\ref{lem:internalExternalSegalConditions} the map $E^1\times I^n\into E^1\times\Delta^n$ is contained in $S$ for all $n\geq 0$, we need only show that the induced map $E^1\times I^n\to I^n$ is contained in $S$. Using that $I^n$ is a colimit of a diagram involving only $\Delta^0$ and $\Delta^1$, we may further restrict to the case $n=1$, in which case the statement was proven by Rezk in~\cite[Proposition~12.1]{rezk2001}.
\end{proof}

\begin{proposition}
	\label{prop:ConditionsInternalCategories}
	Let $C$ be a simplicial object in $\BB$. The following conditions are equivalent:
	\begin{enumerate}
		\item $C$ is a $\BB$-category;
		\item \label{cond2} for all $n\geq 2$ the maps
		\begin{equation*}
		C_n\to C_1\times_{C_0} \cdots\times_{C_0} C_1
		\end{equation*}
		as well as the map
		\begin{equation*}
		C_0\to (C_0\times C_0)\times_{C_1\times C_1} C_3
		\end{equation*}
		are equivalences.
	\end{enumerate}
\end{proposition}
\begin{proof}
	A simplicial object $C$ in $\BB$ satisfies the Segal condition if and only if it is local with respect to the collection of maps $I^2\otimes E\into\Delta^2\otimes E$ for any simplicial object $E\in\Simp\BB$. On the other hand, the first map of condition~(2) is an equivalence if and only if $C$ is local with respect to all maps $I^n\otimes A\into\Delta^n\otimes A$ for arbitrary $A\in\BB$. By Lemma~\ref{lem:internalExternalSegalConditions}, these two conditions are equivalent. Similarly, Lemma~\ref{lem:internalExternalUnivalence} implies that $C$ is univalent if and only if the second map of condition~(2) is an equivalence.
\end{proof}

\begin{remark}
	\label{rem:generalizedDefinitionCategory}
	Proposition~\ref{prop:ConditionsInternalCategories} allows us to make sense of the notion of a $\CC$-category for any $\infty$-category $\CC$ with finite limits. That is, we may define a $\CC$-category to be a simplicial object $C\in\Simp\CC$ that satisfies the second condition of Proposition~\ref{prop:ConditionsInternalCategories}.
\end{remark}

Since $\Cat(\BB)$ is a localisation of $\Simp{\BB}$ at a small set of morphisms, one finds:
\begin{proposition}
	\label{prop:categoriesAccessiblelocalisation}
	The inclusion $\Cat(\BB)\into\Simp\BB$ exhibits $\Cat(\BB)$ as an accessible localisation of $\Simp\BB$. In particular, the $\infty$-category $\Cat(\BB)$ is presentable.\qed
\end{proposition}
\begin{remark}
	\label{rem:limitsColimitsofCategories}
	As $\Cat(\BB)$ is an accessible localisation of $\Simp\BB$, the inclusion $\Cat(\BB)\into\Simp\BB$ creates small limits. On the other hand, colimits of small diagrams are usually not created (or even preserved) by the inclusion, as these can be computed by applying the localisation functor $\Simp\BB\to\Cat(\BB)$ to the colimits of the underlying diagrams of simplicial objects in $\BB$. \emph{Filtered} colimits, on the other hand, are created by the inclusion $\Cat(\BB)\into\Simp\BB$. In fact, as such colimits commute with finite limits, it is immediate from Proposition~\ref{prop:ConditionsInternalCategories} that the colimit in $\Simp\BB$ of a small filtered diagram of $\BB$-categories is contained in $\Cat(\BB)$ and is therefore the colimit of this diagram in $\Cat(\BB)$. In particular, as colimits are universal in $\Simp\BB$, this observation implies that \emph{filtered} colimits are universal in $\Cat(\BB)$.
\end{remark}

As $\BB$-categories are \emph{internally} local with respect to the maps $I^2\into\Delta^2$ and $E^1\to 1$, the $\infty$-category $\Cat(\BB)$ is cartesian closed. More precisely, the construction of $\Cat(\BB)$ implies:
\begin{proposition}
	\label{prop:categoriesExponentialIdeal}
	The full subcategory $\Cat(\BB)\into \BB_{\Delta}$ is an \emph{exponential ideal} and therefore in particular a cartesian closed $\infty$-category. In other words, if $\I{C}$ is a $\BB$-category and $D$ is a simplicial object in $\BB$, the internal mapping object $\iFun{D}{\I{C}}$ is a $\BB$-category.\qed
\end{proposition}
\begin{corollary}
	\label{cor:localisationCategoriesFiniteProducts}
	The localisation functor $\Simp\BB\to \Cat(\BB)$ preserves finite products.\qed
\end{corollary}

Recall that the inclusion $\iota\colon \BB\into\BB_{\Delta}$ admits a right adjoint $(-)_0$ and a left adjoint $\colim_{\Delta^{\op}}(-)$. It is moreover immediate from Proposition~\ref{prop:ConditionsInternalCategories} that $\iota$ factors through the inclusion $\Cat(\BB)\into\BB_{\Delta}$, i.e.\ that for any object $A\in \BB$ the associated $\BB$-groupoid defines a $\BB$-category. Therefore, we obtain functors
\begin{equation*}
\begin{tikzcd}[column sep={7em,between origins}]
\Grpd(\BB)\arrow[r, hookrightarrow, ""{name=M}] \arrow[from=r, bend right, shift right, "(-)^\core"'{name=R}, start anchor=north west, end anchor=north east]\arrow[from=r, bend left, shift left, "(-)^\gp"{name=L}, start anchor=south west, end anchor=south east] 
\arrow[from=M, to=R, "\top" description, phantom, near start]\arrow[from=L, to=M, "\top" description, phantom]& \Cat(\BB).
\end{tikzcd}
\end{equation*}
The functor $(-)^{\simeq}$ is referred to as the \emph{core $\BB$-groupoid} functor, and the functor $(-)^\gp$ is referred to as the \emph{groupoidification} functor. If $\I{C}$ is a $\BB$-category, the $\BB$-groupoid $\I{C}^{\core}$ is to be thought of as the maximal $\BB$-groupoid that is contained in $\I{C}$ (i.e.\ the subcategory spanned by all objects and equivalences in $\I{C}$), whereas the $\BB$-groupoid $\I{C}^{\gp}$ should be regarded as the result of formally inverting all morphisms in $\I{C}$.
\begin{proposition}
	The groupoidification functor $(-)^{\gp}\colon \Cat(\BB)\to\Grpd(\BB)$ commutes with small colimits and finite products.
\end{proposition}
\begin{proof}
	As the groupoidification functor is a left adjoint, it commutes with small colimits.
	Moreover, since the final object $\I{1}\in\Cat(\BB)$ is given by the constant simplicial object on $1\in\BB$ and since groupoidification is a localisation functor, there is an equivalence $\I{1}^{\gp}\simeq 1$.  It therefore suffices to consider the case of binary products, which follows from $\Delta^{\op}$ being a sifted $\infty$-category. 
\end{proof}
\begin{proposition}
	The core $\BB$-groupoid functor $(-)^{\core}\colon\Cat(\BB)\to\Grpd(\BB)$ commutes with small filtered colimits and small limits.
\end{proposition}
\begin{proof}
	As the core $\BB$-groupoid functor is a right adjoint, it commutes with small limits. By~\ref{rem:limitsColimitsofCategories}, small filtered colimits between $\BB$-categories are created by the inclusion $\Cat(\BB)\into\Simp\BB$, hence it suffices to show that $(-)_0\colon\Simp\BB\to\BB$ commutes with such  colimits, which is immediate.
\end{proof}

\begin{example}
	\label{ex:categoriesInS}
	Proposition~\ref{prop:ConditionsInternalCategories} shows that an  $\SS$-category is precisely a \emph{complete Segal space} as developed by Rezk~\cite{rezk2001}. By a theorem of Joyal and Tierney~\cite{Joyal2007} the $\infty$-category of complete Segal spaces is a model for the $\infty$-category of $\infty$-categories $\CatS$, i.e.\ the left Kan extension of the inclusion $\Delta\into\CatS$ determines an equivalence $\Cat(\SS)\simeq\CatS$.
\end{example}

\subsection{Functoriality and base change}
\label{sec:functoriality}
In this section we discuss how to change the base $\infty$-topos for internal higher category theory. What makes this possible is the following general lemma:
\begin{lemma}
	\label{lem:functorialityInternalCategories}
	Let $F\colon\CC\to\DD$ be a functor between $\infty$-categories that admit finite limits such that $F$ preserves pullbacks and such that for any map $f\colon c\to c^\prime$ in $\CC$ the commutative square
	\begin{equation*}
	\begin{tikzcd}[column sep=large]
	F(c\times c)\arrow[r, "F(f\times f)"]\arrow[d] & F(c^\prime\times c^\prime)\arrow[d]\\
	F(c)\times F(c)\arrow[r, "F(f)\times F(f)"] & F(c^\prime)\times F(c^\prime)
	\end{tikzcd}
	\end{equation*}
	is cartesian. Then the induced functor $\Simp\CC\to \Simp\DD$ that is given by postcomposition with $F$ sends $\CC$-categories to $\DD$-categories and therefore restricts to a functor
	\begin{equation*}
	F\colon \Cat(\CC)\to\Cat(\DD).
	\end{equation*} 
	In particular, any left exact functor $F$ between finitely complete $\infty$-categories induces a functor on the level of categories.
\end{lemma}
\begin{proof}
	Since $F$ preserves pullbacks the functor $\Simp\CC\to \Simp\DD$ that is given by postcomposition with $F$ preserves the Segal conditions, and by assumption on $F$ the commutative square
	\begin{equation*}
	\begin{tikzcd}
	F(\I{C}_0)\arrow[r]\arrow[d] & F(\I{C}_3)\arrow[d, "{(d_{\{0,2\}},d_{\{1,3\}})}"]\\
	F(\I{C}_0)\times F(\I{C}_0)\arrow[r, "{(s_0,s_0)}"] & F(\I{C}_1)\times F(\I{C}_1)
	\end{tikzcd}
	\end{equation*}
	is cartesian for every $\CC$-category $\I{C}$, hence the claim follows.
\end{proof}
Let $\ZZ$ denote the subcategory of $\CatSSS$ spanned by the $\infty$-categories that admit finite limits, together with those functors that satisfy the conditions of Lemma~\ref{lem:functorialityInternalCategories}. Let $\int \Simp{(-)}\to \ZZ$ be the cocartesian fibration that classifies the functor $\Simp{(-)}\colon \ZZ\to \CatSSS$. Lemma~\ref{lem:functorialityInternalCategories} then implies that the full subcategory of $\int \Simp{(-)}$ that is spanned by the pairs $(\CC, \I{C})$ with $\CC\in \ZZ$ and $\I{C}$ a $\CC$-category is stable under taking cocartesian arrows and therefore defines a cocartesian subfibration of $\int \Simp{(-)}$ over $\ZZ$. Hence one obtains a functor
\begin{equation*}
\Cat \colon \ZZ\to \CatSSS.
\end{equation*}
Since both the forgetful functor $\RTop\to\CatSSS$ and the universe enlargement functor $\BB\mapsto \BBB$ for $\infty$-topoi factor through $\ZZ$ and since moreover the inclusion $\BB\into\BBB$ commutes with small limits, postcomposition with the functor $\Cat$ gives rise to functors $\BB\mapsto \Cat(\BB)$ as well as $\BB\mapsto \Cat(\BBB)$ together with a natural transformation
\begin{equation*}
\begin{tikzcd}[column sep={3cm,between origins}]
\RTop\arrow[r, bend left,"\BB\mapsto \Cat(\BB)"{name=U}, start anchor=east, end anchor=west, shift left=.5em]\arrow[r, bend right, "\BB\mapsto \Cat(\BBB)"{name=L, below}, start anchor=east, end anchor=west, shift right=.5em]\arrow[from=U, to=L, Rightarrow, shorten=3mm] & \CatSSS
\end{tikzcd}
\end{equation*}
that is given by the inclusion $\Cat(\BB)\into\Cat(\BBB)$. Analogously, one obtains two functors $\LTop\rightrightarrows\CatSSS$ together with a natural transformation
\begin{equation*}
\begin{tikzcd}[column sep={3cm,between origins}]
\LTop\arrow[r, bend left,"\BB\mapsto \Cat(\BB)"{name=U}, start anchor=east, end anchor=west, shift left=.5em]\arrow[r, bend right, "\BB\mapsto \Cat(\BBB)"{name=L, below}, start anchor=east, end anchor=west, shift right=.5em]\arrow[from=U, to=L, Rightarrow, shorten=3mm] & \CatSSS.
\end{tikzcd}
\end{equation*}
As the inclusion $\CC\into \Simp{\CC}$ and its right adjoint $(-)_0\colon \Simp{\CC}\to\CC$ are clearly functorial in $\CC\in\ZZ$, we may now conclude:
\begin{proposition}
	\label{prop:universeEnlargementCategories}
	There are two commutative squares
	\begin{equation*}
	\begin{tikzcd}[column sep=large]
	\Grpd(\BB)\arrow[r, hookrightarrow]\arrow[d, hookrightarrow] & \Grpd(\BBB)\arrow[d, hookrightarrow]\\
	\Cat(\BB)\arrow[r, hookrightarrow]\arrow[u, "(-)^{\core}"', bend right, start anchor=north, end anchor=south, shift right=1em] & \Cat(\BBB)\arrow[u, "(-)^{\core}"', bend right, start anchor=north, end anchor=south, shift right=1em]
	\end{tikzcd}
	\end{equation*}
	that are functorial in $\BB$ both with respect to maps in $\RTop$ and maps in $\LTop$.\qed
\end{proposition}
Let $\WW$ be the full subcategory of $\ZZ$ that is spanned by the $\infty$-categories that admit colimits indexed by $\Delta^{\op}$ and the functors that preserve such colimits. Then the map $\colim_{\Delta^{\op}}\colon \Simp\CC\to\CC$ is functorial in $\CC\in\WW$, and as universe enlargement preserves small colimits one finds:
\begin{proposition}
	\label{prop:universeEnlargementCategoriesGroupoidification}
	There is a commutative square
	\begin{equation*}
	\begin{tikzcd}[column sep=large]
	\Grpd(\BB)\arrow[r, hookrightarrow]& \Grpd(\BBB)\\
	\Cat(\BB)\arrow[r, hookrightarrow]\arrow[u, "(-)^{\gp}"'] & \Cat(\BBB)\arrow[u, "(-)^{\gp}"']
	\end{tikzcd}
	\end{equation*}
	that is functorial in $\BB$ with respect to maps in $\LTop$.\qed
\end{proposition}

\begin{convention}
	We refer to the objects in $\Cat(\BBB)$ as \emph{large} $\BB$-categories (or as $\BBB$-categories) and to the objects in $\Cat(\BB)$ as \emph{small} $\BB$-categories. If not specified otherwise, every $\BB$-category is small. Note, however, that by replacing the universe $\bU$ with the larger universe $\bV$ (i.e.\ by working internally to $\BBB$), every statement about $\BB$-categories carries over to one about large $\BB$-categories as well. Also, we will sometimes omit specifying the relative size of a $\BB$-category if it is evident from the context.
\end{convention}

Base change along \emph{\'etale} geometric morphisms is particularly well-behaved, as we will discuss now.
Let $\Over{\Cat(\BB)}{-}\colon \BB^{\op}\to \CatSS$ be the functor that classifies the cartesian fibration
\begin{equation*}
\Fun(\Delta^1,\Cat(\BB))\times_{\Cat(\BB)}\BB\to \BB.
\end{equation*}
One now finds:
\begin{proposition}
	\label{prop:etaleBaseChangeInternalCategories}
	There is an equivalence
	\begin{equation*}
	\Cat(\Over{\BB}{-})\simeq \Over{\Cat(\BB)}{-}
	\end{equation*}
	of functors $\BB^{\op}\to \CatSS$ that fits into a commutative square
	\begin{equation*}
		\begin{tikzcd}
		\Cat(\Over{\BB}{-})\arrow[d, "\simeq"]\arrow[r, hookrightarrow] & \Cat(\widehat{\Over{\BB}{-}})\arrow[d, "\simeq"]\\
		\Over{\Cat(\BB)}{-}\arrow[r, hookrightarrow] & \Over{\Cat(\BBB)}{-}
		\end{tikzcd}
	\end{equation*}
\end{proposition}
\begin{proof}
	On account of the natural equivalence $\Fun(\Delta^1, \Simp\BB)\simeq \Simp{\Fun(\Delta^1, \BB)}$ and by Lemma~\ref{lem:PShCotensored} and Remark~\ref{rem:tensoringCotensoringSheaves} below, one obtains an equivalence
	\begin{equation*}
	\Simp{(\Over{\BB}{-})}\simeq \Over{(\Simp\BB)}{-}
	\end{equation*}
	of functors $\BB^{\op}\to \CatSS$. In order to provide an equivalence $\Cat(\Over{\BB}{-})\simeq \Over{\Cat(\BB)}{-}$, it therefore suffices to show that this equivalence of functors sends $\Over{\BB}{A}$-categories to objects in $\Over{(\Simp\BB)}{A}$ whose underlying simplicial object in $\BB$ is a $\BB$-category. By construction, the component of the above equivalence at $A\in \BB$ sits inside the commutative diagram
	\begin{equation*}
	\begin{tikzcd}[column sep=.1em]
	\Simp{(\Over{\BB}{A})}\arrow[dr, "(\pi_A)_!\circ(-)"']\arrow[rr, "\simeq"] && \Over{(\Simp\BB)}{A}\arrow[dl, "(\pi_A)_!"]\\
	& \Simp\BB,&
	\end{tikzcd}
	\end{equation*}
	hence the claim follows from the fact that the forgetful functor $(\pi_A)_!\colon \Over{\BB}{A}\to\BB$ satisfies the conditions of Lemma~\ref{lem:functorialityInternalCategories}. Lastly, the existence of a commutative square as in the statement of the proposition follows from the construction of the equivalence $\Cat(\Over{\BB}{-})\simeq \Over{\Cat(\BB)}{-}$ and Lemma~\ref{lem:equivalenceSheavesSlices}.
\end{proof}
\subsection{The $(\infty,2)$-categorical structure of $\Cat(\BB)$}
\label{sec:infty2}
For any $\infty$-topos $\BB$, we denote by
\begin{equation*}
\iFun{-}{-}\colon \Cat(\BB)^{\op}\times \Cat(\BB)\to \Cat(\BB)
\end{equation*}
the internal mapping bifunctor in $\Cat(\BB)$. By Proposition~\ref{prop:categoriesExponentialIdeal} this bifunctor is obtained by restricting the internal mapping bifunctor of $\Simp\BB$ to the full subcategory $\Cat(\BB)\into\Simp\BB$. As the product bifunctor on $\Cat(\BB)$ is obtained in the same fashion, the three bifunctors that are defined in the beginning of \S~\ref{sec:simpObjects} restrict to bifunctors on the level of $\BB$-categories and in $\SS$. Explicitly, one obtains a tensoring bifunctor
\begin{equation*}
-\otimes -\colon \CatS\times\Cat(\BB)\to\Cat(\BB)
\end{equation*}
which is given by the composition $(-\times -)\circ (\const\times \id_{\Cat(\BB)})$, a powering bifunctor
\begin{equation*}
(-)^{(-)}\colon \CatS^{\op}\times\Cat(\BB)\to \Cat(\BB)
\end{equation*}
that is given by the composition $\iFun{-}{-}\circ (\const\times\id_{\Cat(\BB)})$, and a functor $\infty$-category bifunctor
\begin{equation*}
	\Fun_{\BB}(-,-)\colon\Cat(\BB)^{\op}\times\Cat(\BB)\to\CatS
\end{equation*}
which is defined as the composition $\Gamma\circ\iFun{-}{-}$. These functors are equipped with equivalences
\begin{equation*}
\map{\Cat(\BB)}(-, (-)^{(-)})\simeq \map{\Cat(\BB)}(-\otimes-, -)\simeq \map{\CatS}(-, \Fun_{\BB}(-,-)).
\end{equation*}
In particular, the second equivalence implies that postcomposing $\Fun_{\BB}(-,-)$ with the core $\infty$-groupoid functor $(-)^{\core}\colon\CatS\to\SS$ recovers the mapping $\infty$-groupoid bifunctor $\map{\Cat(\BB)}(-,-)$. 

The above constructions are well-behaved with respect to universe enlargement:
\begin{proposition}
	\label{prop:internalMappingObjectUniverseEnlargement}
	The internal mapping bifunctor on $\Cat(\BBB)$ restricts to the internal mapping bifunctor on $\Cat(\BB)$.
\end{proposition}
\begin{proof}
	As the product bifunctor on $\Cat(\BBB)$ clearly restricts to the product bifunctor on $\Cat(\BB)$, it suffices to show that for any two \emph{small} $\BB$-categories $\I{C},\I{D}\in\Cat(\BB)\into\Cat(\BBB)$ their internal mapping object $\iFun{\I{C}}{\I{D}}$ in $\Cat(\BBB)$ is small as well. It suffices to show this on the level of simplicial objects, i.e. we need to show that the simplicial object $\iFun{\I{C}}{\I{D}}\in \Simp\BBB$ is contained in $\Simp\BB$. Using Proposition~\ref{prop:simplicialPowering} one finds
	\begin{equation*}
		\iFun{\I{C}}{\I{D}}_n\simeq \iFun{\Delta^n\otimes \I{C}}{\I{D}}_0,
	\end{equation*}
	hence it suffices to show that $\iFun{\I{C}}{\I{D}}_0$ is contained in $\BB$. As in the proof of Proposition~\ref{prop:ConditionsInternalCategories}, using that the functor $\iFun{-}{\I{D}}_0$ sends colimits in $\Simp\BBB$ to limits in $\BBB$, we may assume without loss of generality $\I{C}\simeq\Delta^n\otimes A$. In this case one computes
	\begin{equation*}
		\iFun{\Delta^n\otimes A}{\I{D}}_0\simeq \iFun{A}{\I{D}}_n\simeq \iFun{A}{\I{D}_n}
	\end{equation*}
	in which the last step follows from Remark~\ref{rem:InternalHomConstantSimplicialObjects}, therefore the claim is a consequence of Proposition~\ref{prop:universeEnlargementStructurePreservation}.
\end{proof}
Combining Proposition~\ref{prop:internalMappingObjectUniverseEnlargement} with Proposition~\ref{prop:universeEnlargementCategories}, one now easily deduces:
\begin{corollary}
	\label{cor:universeEnlargementEnrichedStructure}
	The tensoring, powering and mapping $\infty$-category bifunctors on $\Cat(\BBB)$ restrict to the tensoring, powering and mapping $\infty$-category bifunctors on $\Cat(\BB)$.\qed
\end{corollary}

\begin{remark}
The bifunctor $\Fun_{\BB}(-,-)$ gives rise to an $(\infty,2)$-categorical enhancement of $\Cat(\BB)$. More precisely, on account of $\Cat(\BB)$ being cartesian closed, this $\infty$-category is canonically enriched over itself~\cite[\S~7]{gepner2015}. The functor $\Gamma\colon\Cat(\BB)\to\CatS$ can then be used to change enrichment from $\Cat(\BB)$ to $\CatS$, so that $\Cat(\BB)$ becomes a $\CatS$-enriched $\infty$-category, which is one of the known models for $(\infty,2)$-categories~\cite{haugseng2015}.
\end{remark}

\subsection{$\CatS$-valued sheaves on an $\infty$-topos}
\label{sec:categoriesAsSheaves}
Categories in an $\infty$-topos $\BB$ can be alternatively regarded as sheaves of $\infty$-categories on $\BB$. To see this, first recall from the discussion in \S~\ref{sec:universeEnlargement} that the Yoneda embedding induces a commutative square
\begin{equation*}
	\begin{tikzcd}
	\BB\arrow[d, hookrightarrow]\arrow[r, hookrightarrow] & {\BBB}\arrow[d, hookrightarrow]\\
	\PSh_{\SS}(\BB)\arrow[r, hookrightarrow] &\PSh_{\SSS}(\BB)	\end{tikzcd}
\end{equation*}
that is natural in $\BB$ both with respect to maps in $\RTop$ and in $\LTop$. Postcomposition with the functor $\Simp{(-)}$ therefore gives rise to a natural commutative square
\begin{equation*}
	\begin{tikzcd}
	\Simp\BB\arrow[d, hookrightarrow]\arrow[r, hookrightarrow] & \Simp{\BBB}\arrow[d, hookrightarrow]\\
	\PSh_{\Simp\SS}(\BB)\arrow[r, hookrightarrow] &\PSh_{\Simp\SSS}(\BB)	\end{tikzcd}
\end{equation*}
in which the essential image of the two vertical maps is spanned by the collection of $\Simp\SS$-valued and $\Simp\SSS$-valued sheaves, respectively. Using Proposition~\ref{prop:ConditionsInternalCategories} it is immediate that this square further restricts to a natural commutative square
\begin{equation*}
\begin{tikzcd}
\Cat(\BB)\arrow[d, hookrightarrow]\arrow[r, hookrightarrow] & \Cat(\BBB)\arrow[d, hookrightarrow]\\
\PSh_{\CatS}(\BB)\arrow[r, hookrightarrow] &\PSh_{\CatSS}(\BB).
\end{tikzcd}
\end{equation*}
As limits in $\CatS$ and in $\CatSS$ are computed on the level of the underlying simplicial $\infty$-groupoids, the essential image of the two vertical maps is spanned by the collection of $\CatS$-valued and $\CatSS$-valued sheaves, respectively. One therefore obtains:
\begin{proposition}
	\label{prop:internalCategoriesSheaves}
	The inclusions $\Cat(\BB)\into\PSh_{\CatS}(\BB)$ and $\Cat(\BBB)\into\PSh_{\CatSS}(\BB)$ induce a commutative square
	\begin{equation*}
	\begin{tikzcd}
	\Cat(\BB)\arrow[d, "\simeq"]\arrow[r, hookrightarrow] & \Cat(\BBB)\arrow[d, "\simeq"]\\
	\Shv_{\CatS}(\BB)\arrow[r, hookrightarrow] &\Shv_{\CatSS}(\BB)
	\end{tikzcd}
	\end{equation*}
	that is natural in $\BB$ both with respect to maps in $\RTop$ and in $\LTop$.\qed
\end{proposition}
\begin{remark}
	In what follows, we will often implicitly identify a $\BB$-category $\I{C}$ with the associated $\CatS$-valued sheaf on $\BB$. In particular, if $A\in\BB$ is an arbitrary object we will write $\I{C}(A)$ for the $\infty$-category of $A$-sections of $\I{C}$.
\end{remark}
\begin{remark}
	\label{rem:categoriesAsFibrations}
	On account of the natural equivalence $\PSh_{\CatS}(\CC)\simeq \Cart(\CC)$ between the $\infty$-category of $\CatS$-valued presheaves on a small $\infty$-category $\CC$ and the $\infty$-category of cartesian fibrations over $\CC$ established by straightening and unstraightening, the inclusion $\Cat(\BBB)\into \PSh_{\CatSS}(\BB)$ gives rise to an embedding
		\begin{equation*}
			\Cat(\BBB)\into \Cart(\BB)
		\end{equation*}
	that is natural in $\BB$.
	For any (large) $\BB$-category $\I{C}$ we will denote the image of $\I{C}$ under this functor by $\int \I{C}\to \BB$.
\end{remark}

The equivalence in Proposition~\ref{prop:internalCategoriesSheaves} can also be formulated using the mapping $\infty$-category bifunctor $\Fun_{\BB}(-,-)$: If $\iota\colon \BB\into \Cat(\BB)$ denotes the natural inclusion, the computation
\begin{equation*}
	\map{\CatS}(\Delta^{\bullet}, \Fun_{\BB}(\iota(-), -))\simeq\map{\Cat(\BB)}(\iota(-), (-)^{\Delta^{\bullet}})\simeq \map{\BB}(-, (-)_{\bullet})
\end{equation*}
in which the last equivalence follows from Proposition~\ref{prop:simplicialPowering} shows that the transpose of the bifunctor
\begin{equation*}
	\Fun_{\BB}(\iota(-),-)\colon \BB^{\op}\times\Cat(\BB)\to \CatS
\end{equation*}
recovers the natural inclusion $\Cat(\BB)\into \PSh_{\CatS}(\BB)$ and in particular the equivalence $\Cat(\BB)\simeq\Shv_{\CatS}(\BB)$ from Proposition~\ref{prop:internalCategoriesSheaves}. It is therefore reasonable to define:
\begin{definition}
	\label{def:localSectionsFunctor}
	For any $\infty$-topos $\BB$ and any object $A\in \BB$, the \emph{local sections functor over $A$}  is defined as the functor $\Fun_{\BB}(A,-)\colon \Cat(\BB)\to\CatS$.
\end{definition}

\begin{remark}
	In the context of Definition~\ref{def:localSectionsFunctor}, the local sections functor over an object $A\in \BB$ is equivalently given by the composite
	\begin{equation*}
		\Cat(\BB)\xrightarrow{\pi_A^\ast} \Cat(\Over{\BB}{A})\xrightarrow{\Gamma_{\Over{\BB}{A}}} \CatS.
	\end{equation*}
	In fact, the equivalence of functors $-\times(\pi_A)_!(-)\simeq(\pi_A)_!(\pi_A^\ast(-)\times_A -)$ gives rise to the following chain of equivalences
	\begin{align*}
		\map{\Cat(\BB)}(-\otimes (\pi_A)_!(-), -)&\simeq\map{\Cat(\BB)}((\pi_A)_!(-\otimes -), -)\\
		&\simeq\map{\Cat(\Over{\BB}{A})}(-\otimes -, \pi_A^\ast(-))
	\end{align*}
	that induces an equivalence of functors
	\begin{equation*}
		\Fun_{\BB}((\pi_A)_!(-),-)\simeq\Fun_{\Over{\BB}{A}}(-, \pi_A^\ast(-)).
	\end{equation*}
\end{remark}

\begin{remark}
	\label{rem:coreGroupoidificationSheaves}
	By construction, the equivalence $\Cat(\BB)\simeq \Shv_{\CatS}(\BB)$ fits into two commutative squares
	\begin{equation*}
		\begin{tikzcd}[column sep=large]
		\Grpd(\BB)\arrow[r, "\simeq"]\arrow[d, hookrightarrow] & \Shv_{\SS}(\BB)\arrow[d, hookrightarrow]\\
		\Cat(\BB)\arrow[r, "\simeq"]\arrow[u, "(-)^{\core}"', bend right, start anchor=north, end anchor=south, shift right=1em] & \Shv_{\CatS}(\BB)\arrow[u, "(-)^{\core}"', bend right, start anchor=north, end anchor=south, shift right=1em]
		\end{tikzcd}
	\end{equation*}
	that are functorial in $\BB$ with respect to maps both in $\RTop$ and in $\LTop$. Here the two vertical maps on the right-hand side are given by postcomposition with the adjunction $(\iota\dashv (-)^\simeq)\colon \SS\leftrightarrows\CatS$. One moreover has a commutative square
	\begin{equation*}
		\begin{tikzcd}[column sep=large]
		\Grpd(\BB)\arrow[r, "\simeq"]& \Shv_{\SS}(\BB)\\
		\Cat(\BB)\arrow[r, "\simeq"]\arrow[u, "(-)^{\gp}"'] & \Shv_{\CatS}(\BB)\arrow[u, "(-)^{\gp}"']
		\end{tikzcd}
	\end{equation*}
	that is functorial in $\BB$ with respect to maps in $\LTop$, where the vertical map on the right is given by postcomposition with the groupoidification functor $(-)^{\gp}\colon \CatS\to\SS$.
\end{remark}
\begin{remark}
	\label{rem:tensoringCotensoringSheaves}
	The fact that the inclusion $\Cat(\BB)\into\PSh_{\CatS}(\BB)$ is obtained by the functor $\Fun_{\BB}(\iota(-),-)$ implies that there is a commutative square
	\begin{equation*}
		\begin{tikzcd}[column sep=huge]
		\CatS^{\op}\times\Cat(\BB)\arrow[r, "{(-)^{(-)}}"]\arrow[d, hookrightarrow] & \Cat(\BB)\arrow[d, hookrightarrow]\\
		\CatS^{\op}\times\PSh_{\CatS}(\BB)\arrow[r, "{\Fun_{\SS}(-, -)}\circ (-)"] & \PSh_{\CatS}(\BB)
		\end{tikzcd}
	\end{equation*}
	in which the lower horizontal arrow is to be understood as the functor that sends a pair $(\XX, F)$ to the presheaf $\Fun_{\SS}(\XX,-)\circ F$. In other words, the $\CatS$-valued presheaf that underlies the powering $\I{C}^{\XX}$ of a $\BB$-category $\I{C}$ by an $\infty$-category $\XX\in\CatS$ is given by the functor $\Fun_{\SS}(\XX, \I{C}(-))$.
	
	By contrast, the $\CatS$-valued presheaf that underlies the \emph{tensoring} $\XX\otimes \I{C}$ of the $\BB$-category $\I{C}$ by the $\infty$-category $\XX$ is \emph{not} given by the functor $\XX\times \I{C}(-)$. In fact, the latter is in general not a sheaf. However, one can show that the presheaf that is associated with $\XX\otimes \I{C}$ is given by the sheafification of the presheaf $\XX\times \I{C}(-)$, i.e.\ by the image of $\XX\times\I{C}(-)$ under the left adjoint of the inclusion $\Shv_{\CatS}(\BB)\into \PSh_{\CatS}(\BB)$.
\end{remark}
\subsection{Objects and morphisms}
\label{sec:objectsMorphisms}
Let $\I{C}$ be a $\BB$-category, and let $A\in \BB$ be an arbitrary object. An \emph{object $c$ of $\I{C}$ in context $A$} is defined to be a local section $c\colon A\to \I{C}$, which is equivalently determined by a map $c\colon A\to \I{C}_0$ since $A$ is a $\BB$-groupoid. A \emph{morphism in $\I{C}$ with context $A$} is defined as an object $f$ of $\I{C}^{\Delta^1}$, i.e.\  as a map $f\colon \Delta^1\otimes A\to \I{C}$, or equivalently as a map $f\colon A\to \I{C}_1$. Similarly one defines the notion of an $n$-morphism for any $n\geq 1$ as a map $\Delta^n\otimes A\to \I{C}$. Any morphism $f$ has a source and a target which are obtained by precomposing $f\colon \Delta^1\otimes A\to \I{C}$ with $d^1\colon A\to \Delta^1\otimes A$ and with $d^0\colon A\to \Delta^1\otimes A$, respectively. If $c$ and $d$ are the source and target of such a morphism $f$, we also use the familiar notation $f\colon c\to d$. For any object $c$ in $\I{C}$ in context $A$, there is a morphism $\id_c$ that is defined by the composite $cs^0\colon \Delta^1\otimes A\to A\to \I{C}$.

\begin{remark}
	On account of the adjunction $(\pi_A)_!\dashv\pi_A^\ast$, specifying an object $c\colon A\to\I{C}$ in context $A\in\BB$ is tantamount to specifying an object $c\colon 1\to\pi_A^\ast\I{C}$ in context $1\in\Over{\BB}{A}$. A similar observation can be made for morphisms in a $\BB$-category.
\end{remark}

Given two objects $c,d$ in $\I{C}$ in context $A$, we define the \emph{mapping $\BB$-groupoid} $\map{\I{C}}(c,d)\in\Over{\BB}{A}$ as the pullback
\begin{equation*}
	\begin{tikzcd}
	\map{\I{C}}(c,d)\arrow[d]\arrow[r] & \I{C}_1\arrow[d, "{(d_1,d_0)}"]\\
	A\arrow[r, "{(c,d)}"] & \I{C}_0\times\I{C}_0.
	\end{tikzcd}
\end{equation*}
Equivalently, this object can be defined by the pullback
\begin{equation*}
\begin{tikzcd}
\map{\I{C}}(c,d)\arrow[d]\arrow[r] & \I{C}^{\Delta^1}\arrow[d, "{(d_1,d_0)}"]\\
A\arrow[r, "{(c,d)}"] & \I{C}\times\I{C},
\end{tikzcd}
\end{equation*}
see \S~\ref{sec:leftFibrations} below. By construction, sections $A\to \map{\I{C}}(c,d)$ over $A$ correspond to morphisms $f\colon c\to d$ in $\I{C}$ in context $A$. Two maps $f,g\colon c\rightrightarrows d$ are said to be equivalent if they are equivalent as sections $A\rightrightarrows \map{\I{C}}(c,d)$ over $A$, in which case we write $f\simeq g$.

Similarly to the case of two objects, if $c_0,\dots,c_n$ are objects in context $A$ in $\I{C}$, one writes $\map{\I{C}}(c_0,\dots,c_n)$ for the pullback of $(d_n,\dots,d_0)\colon \I{C}_n\to \I{C}_0^{n+1}$ along the map $(c_0,\dots,c_n)\colon A\to \I{C}_0^{n+1}$. Using the Segal conditions, one obtains an equivalence
\begin{equation*}
	\map{\I{C}}(c_0,\dots,c_n)\simeq\map{\I{C}}(c_0,c_1)\times_A\cdots\times_A\map{\I{C}}(c_{n-1},c_n).
\end{equation*}
By combining this identification with the map $\map{\I{C}}(c_0,\dots,c_n)\to\map{\I{C}}(c_0,c_n)$ that is induced by the map $d_{\{0,n\}}\colon \I{C}_n\to\I{C}_1$, one obtains a composition map
\begin{equation*}
	\map{\I{C}}(c_0,c_1)\times_A\cdots\times_A\map{\I{C}}(c_{n-1},c_n)\to\map{\I{C}}(c_0,c_n).
\end{equation*}
Given maps $f_i\colon c_{i-1}\to c_i$ in $\I{C}$ for $i=1,\dots, n$, we write $f_1\cdots f_n$ for their composition. By making use of the simplicial identities, it is straightforward to verify that composition is associative and unital, i.e.\ that the relations $f(gh)\simeq (fg)h$ and $f\id\simeq f\simeq \id f$ as well as their higher analogues hold whenever they make sense, see~\cite[Proposition~5.4]{rezk2001} for a proof.

A morphism $f\colon c\to d$ in $\I{C}$ is an \emph{equivalence} if there are maps $g\colon c\to d$ and $h\colon c\to d$ (all in context $A$) such that $gf\simeq \id_c$ and $fh\simeq \id_d$. Let $\Delta^1\to E^1$ be the map that is induced by the inclusion $d^{\{1,2\}}\colon\Delta^1\into\Delta^3$. One then obtains the following characterisation of equivalences in $\I{C}$:
\begin{proposition}
	\label{prop:walkingEquivalenceClassifiesEquivalences}
	A map $f\colon c\to d$ in context $A$ is an equivalence if and only if the map $\Delta^1\otimes A\to \I{C}$ that is determined by $f$ factors through the map $\Delta^1\otimes A\to E^1\otimes A$.
\end{proposition}
\begin{proof}
	Suppose that there are $g,h\colon d\rightrightarrows c$ together with equivalences $gf\simeq \id_c$ and $fh\simeq\id_d$ that witness $f$ as an equivalence in $\I{C}$. The triple $(h,f,g)$ then determines a map $I^3\otimes A\to \I{C}$ which can be uniquely extended to a map $\Delta^3\otimes A\to \I{C}$ since $\I{C}$ is a $\BB$-category. By construction, the restriction of this map along the inclusions $d^{\{0,2\}}\colon\Delta^1\otimes A\to\Delta^3\otimes A$ and $d^{\{1,3\}}\colon \Delta^1\otimes A\to\Delta^3\otimes A$ are equivalent to $\id_d$ and $\id_c$, respectively. By definition, this means that $\Delta^3\otimes A\to\I{C}$ extends along the map $\Delta^3\otimes A\to E^1\otimes A$.
	
	Conversely, if the map $\Delta^1\otimes A\to \I{C}$ that is determined by $f$ factors through the map $\Delta^1\otimes A\to E^1\otimes A$, it in particular determines a map $\Delta^3\otimes A\to\I{C}$ whose restriction along $d^{\{0,1\}}\colon \Delta^1\otimes A\to\Delta^3\otimes A$ and $d^{\{2,3\}}\colon\Delta^1\otimes A\to\Delta^3\otimes A$ gives rise to two maps $h,g\colon c\rightrightarrows d$ in $\I{C}$. By construction of $E^1$ and the definition of composition, the composites $fh$ and $gf$ factor through $d\colon A\to \I{C}$ and $c\colon A\to \I{C}$, respectively, which means that these composites are equivalent to $\id_c$ and $\id_d$.
\end{proof}
\begin{corollary}
	\label{cor:equivalencesIdentities}
	A map $f\colon A\to\I{C}_1$ defines an equivalence in $\I{C}$ if and only if it factors through the map $s_0\colon \I{C}_0\to\I{C}_1$.
\end{corollary}
\begin{proof}
	Since $\I{C}$ is a $\BB$-category, any map $E^1\otimes A\to\I{C}$ extends uniquely along the projection $E^1\otimes A\to A$, hence the result follows from Proposition~\ref{prop:walkingEquivalenceClassifiesEquivalences}.
\end{proof}
As a consequence of Corollary~\ref{cor:equivalencesIdentities}, given two objects $c,d$ in $\I{C}$ in context $A\in \BB$, we may define the \emph{$\BB$-groupoid of equivalences} $\Eq{\I{C}}(c,d)\in\Over{\BB}{A}$ via the pullback square
\begin{equation*}
	\begin{tikzcd}
		{\Eq{\I{C}}(c,d)}\arrow[r]\arrow[d]& \I{C}_0\arrow[d, "{(d_1,d_0)}"]\\
		A\arrow[r, "{(c,d)}"] & \I{C}_0\times\I{C}_0.
	\end{tikzcd}
\end{equation*}
By construction, sections $A\to \Eq{\I{C}}(c,d)$ over $A$ correspond to equivalences $f\colon c\to d$ in $\I{C}$ in context $A$.
We will say that two objects $c,d\colon A\rightrightarrows \I{C}$ are \emph{equivalent} if there is an equivalence $c\simeq d$, i.e.\ a section $A\to\Eq{\I{C}}(c,d)$ over $A$. This is equivalent to the condition that $c$ and $d$ are equivalent as objects in $\map{\BB}(A,\I{C}_0)$.

\begin{remark}
	Even if we are dealing with a \emph{large} $\BB$-category $\I{C}$, we will usually only need to consider \emph{small} contexts $A\in\BB$ when speaking about objects or maps in $\I{C}$. Essentially, this is possible since $\BBB$ is generated by $\BB$ under large colimits, which implies that every \emph{large} context $A\in\BBB$ admits a cover by small contexts. Another way of saying this is that $\I{C}$ is uniquely specified by its value at small contexts, cf.\ Remark~\ref{rem:transitivityUniverseEnlargement}.
\end{remark}

\subsection{The universe for groupoids}
\label{sec:universe}
Recall that for any large $\infty$-category $\CC$ that admits pullbacks one can define a $\CatSS$-valued presheaf $\Over{\CC}{-}$ as the functor that classifies the codomain fibration $d_0\colon\Fun(\Delta^1,\CC)\to\CC$. If $\CC$ is an $\infty$-topos, then this presheaf is a \emph{sheaf}~\cite[Proposition~6.1.3.10]{htt}. One may therefore define:
\begin{definition}
	The \emph{universe for $\BB$-groupoids} is defined to be the large $\BB$-category $\Univ$ that corresponds to the $\CatSS$-valued sheaf $\Over{\BB}{-}$ on $\BB$.
\end{definition}

\begin{remark}
	\label{rem:BCUniverse}
	For any object $A\in \BB$ there is a canonical equivalence $\pi_A^\ast\Univ\simeq \Univ[\Over{\BB}{A}]$. In fact, the $\CatSS$-valued sheaf associated with $\pi_A^\ast\Univ$ is given by $\Fun_{\BB}((\pi_A)_!(-), \Univ)$, hence the claim follows from the observation that since $(\pi_A)_!$ is a right fibration, the commutative square
	\begin{equation*}
	\begin{tikzcd}
		\Fun(\Delta^1, \Over{\BB}{A})\arrow[r] \arrow[d, "d_1"] & \Fun(\Delta^1, \BB)\arrow[d, "d_1"]\\
		\Over{\BB}{A}\arrow[r, "(\pi_A)_!"] & \BB
	\end{tikzcd}
	\end{equation*}
	is cartesian.
\end{remark}

The universe for small $\BB$-groupoids $\Univ$ plays a role akin to the $\infty$-category $\SS$ of $\infty$-groupoids in higher category theory. It should therefore be regarded as a reflection of the $\infty$-topos $\BB$ within itself. This is supported by  the observation that there is an equivalence
\begin{equation*}
	\Fun_{\BB}(A,\Univ)\simeq\Over{\BB}{A}
\end{equation*}
for every $A\in \BB$, which in particular shows that objects $A\to \Univ$ correspond to objects in the slice $\infty$-topos $\Over{\BB}{A}$. Our goal is to obtain a similar characterisation for the maps in $\Univ$. More precisely, we will show the following result:
\begin{proposition}
	\label{prop:mappingObjectsInternalUniverse}
	For any two objects $g,h$ in $\Univ$ in context $A\in\BB$ that correspond to objects $P,Q\in\Over{\BB}{A}$, there is an equivalence
	\begin{equation*}
		\map{\Univ}(g,h)\simeq \Over{\iFun{P}{Q}}{A}
	\end{equation*}
	in $\Over{\BB}{A}$, where the right-hand side denotes the internal mapping object in $\Over{\BB}{A}$.
\end{proposition}
Before giving a proof of Proposition~\ref{prop:mappingObjectsInternalUniverse}, we need a few preparatory steps. We begin by characterising the powering functor on the $\infty$-category of $\BB$-categories. To that end, if $\I{C}$ is a (large) $\BB$-category and if $p\colon\int \I{C}\to \BB$ is the associated cartesian fibration, postcomposing the powering functor $\I{C}^{(-)}$ with the Grothendieck construction $\int\colon \Cat(\BBB)\into\Cart(\BB)\into\Over{(\CatSS)}{\BB}$ gives rise to a functor $\int \I{C}^{(-)}\in\Fun(\CatSS^{\op},\Over{(\CatSS)}{\BB})$. On account of the equivalence
\begin{equation*}
	\Fun(\CatSS^{\op},\Over{(\CatSS)}{\BB})\simeq \Over{\Fun(\CatSS^{\op},\CatSS)}{\BB},
\end{equation*}
one may equivalently regard $\int \I{C}^{(-)}$ as a functor $\CatSS^{\op}\to\CatSS$ together with a map $\int\I{C}^{(-)}\to\BB$.
\begin{lemma}
	\label{lem:PShCotensored}
	The map $\int \I{C}^{(-)}\to \BB$ fits into a cartesian square
	\begin{equation*}
	\begin{tikzcd}
			\int \I{C}^{(-)}\arrow[d]\arrow[r] & \Fun(-, \int \I{C})\arrow[d, "p_\ast"]\\
			\BB \arrow[r, "\diag"] & \Fun(-, \BB)
		\end{tikzcd}
	\end{equation*}
	in $\Fun(\CatSS^{\op}, \CatSS)$, where $\diag$ denotes the diagonal functor.
\end{lemma}
\begin{proof}
	Let $F\colon\CatSS^{\op}\to\CatSS$ be the pullback of $p_\ast$ along $\diag$. By postcomposition with the Yoneda embedding one obtains a cartesian square
	\begin{equation*}
		\begin{tikzcd}
			\map{\CatSS}(-,F(-))\arrow[d]\arrow[r] & \map{\CatSS}(-\times -, \int \I{C})\arrow[d]\\
			\map{\CatSS}(-,\BB) \arrow[r, "(-)\circ \pr_0"] & \map{\CatSS}(-\times -, \BB)
		\end{tikzcd}
	\end{equation*}
	in which $\pr_0$ denotes the projection onto the first factor. 
	By making use of the adjunction $(\pi_{\BB})_!\dashv \pi_{\BB}^\ast$, this shows that there is an equivalence
	\begin{equation*}
		\map{/\BB}(-,F(-))\simeq\map{/\BB}(-\times_{\BB} \pi_{\BB}^{\ast}(-), \smallint \I{C})
	\end{equation*}
	of presheaves on $\Over{(\CatSS)}{\BB}\times\CatSS$.
	As by \cite[Corollary~3.2.2.12]{htt} the functor $F\colon \CatSS^{\op}\to \Over{(\CatSS)}{\BB}$ factors through the subcategory $\Cart(\BB)$ and so does $\pi_{\BB}^\ast$, we conclude the proof with the chain of equivalences
	\begin{equation*}
		\map{\Cart(\BB)}(-,F(-))\simeq \map{\Cart(\BB)}(-\times_{\BB} \pi_{\BB}^\ast(-), \smallint \I{C}) \simeq \map{\Cart(\BB)}(-, \smallint \I{C}^{(-)})
	\end{equation*}
	that gives rise to the desired equivalence of functors $F\simeq \int \I{C}^{(-)}$.
\end{proof}
Let $(-)^\triangleright\colon \CatS \to \CatS$ denote the functor that sends an $\infty$-category $\CC$ to the pushout
\begin{equation*}
	\begin{tikzcd}[column sep=large]
		\CC\arrow[r, "d^0"]\arrow[d] \arrow[dr, very near end, "\lrcorner", phantom] & \CC\times\Delta^1 \arrow[d]\\
		\Delta^0\arrow[r, "\epsilon"] & \CC^\triangleright,
	\end{tikzcd}
\end{equation*}
and note that the map $\epsilon$ in the above square allows us to consider $(-)^\triangleright$ as a functor $\CatS\to \Under{(\CatS)}{\Delta^0}$. Postcomposition with $\Fun(-,\BB)$ then gives rise to a functor
\begin{equation*}
	\Fun((-)^\triangleright,\BB)\colon \CatSS^{\op}\to \Over{(\CatSS)}{\BB}.
\end{equation*}
\begin{proposition}
	\label{prop:CotensoringInternalUniverse}
	The functor
	\begin{equation*}
		\Fun((-)^\triangleright, \BB)\colon \CatSS^{\op}\to \Over{(\CatSS)}{\BB}
	\end{equation*}
	factors through the inclusion $\Cart(\BB)\hookrightarrow \Over{(\CatSS)}{\BB}$ such that there is a commutative square
	\begin{equation*}
		\begin{tikzcd}[column sep=large]
			\CatSS^{\op}\arrow[d, "\Univ^{(-)}"'] \arrow[r, "{\Fun((-)^\triangleright, \BB)}"] & \Cart(\BB)\arrow[d, "\simeq"]\\
			\Cat(\BBB)\arrow[r, hookrightarrow]&\PSh_{\CatSS}(\BB).
		\end{tikzcd}
	\end{equation*}
\end{proposition}
\begin{proof}
	By Lemma~\ref{lem:PShCotensored}, the functor $\int\Univ^{(-)}$ fits into the pullback square
	\begin{equation*}
		\begin{tikzcd}
			\int \Univ^{(-)} \arrow[d] \arrow[r] & \Fun(-\times\Delta^1, \BB)\arrow[d, "(\id\times d_0)^\ast"]\\
			\BB\arrow[r, "\diag"] & \Fun(-, \BB)
		\end{tikzcd}
	\end{equation*}
	in $\Fun(\CatSS^{\op}, \CatSS)$, which is precisely the square obtained by applying the functor $\Fun(-,\BB)$ to the pushout square that defines $(-)^\triangleright$.
\end{proof}

For any $n\geq 0$ there is a canonical equivalence $(\Delta^n)^\triangleright\simeq\Delta^{n+1}$ in $\CatS$. Hence postcomposing the cosimplicial $\infty$-category $\Delta^{\bullet}$ with the functor $(-)^\triangleright$ gives rise to the cosimplicial object $\Delta^{\bullet+1}$ in $\Under{(\CatS)}{\Delta^0}$. Further postcomposing this functor with $\Fun(-,\BB)$ for any $\infty$-topos $\BB$ then results in the simplicial object
\begin{equation*}
	\Delta^{\op}\to \Over{(\CatSS)}{\BB},\quad \ord{n}\mapsto \Fun(\Delta^{n+1},\BB).
\end{equation*}
By Proposition~\ref{prop:CotensoringInternalUniverse}, this simplicial object takes values in $\Cart(\BB)$.

Let $\RFib(\BB)\subset\Cart(\BB)$ be the full subcategory spanned by the right fibrations into $\BB$, and recall that the Grothendieck construction restricts to an equivalence $\RFib(\BB)\simeq\PSh_{\SSS}(\BB)$ with respect to which the core functor $(-)^{\core}\colon\PSh_{\CatSS}(\BB)\to\PSh_{\SSS}(\BB)$ corresponds to the functor that is given by restricting a cartesian fibration $\PP\to\BB$ to the the subcategory of $\PP$ spanned by the cartesian edges. On account of Proposition~\ref{prop:simplicialPowering}, Proposition~\ref{prop:CotensoringInternalUniverse} now implies:
\begin{corollary}
	\label{cor:RightFibrationInternalUniverse}
	The core functor $\Cart(\BB)\to \RFib(\BB)\simeq\PSh_{\SSS}(\BB)$ sends the simplicial object $\Fun(\Delta^{\bullet +1},\BB)$ in $\Cart(\BB)$ to the universe $\Univ\in\Cat(\BBB)\into\Simp{\PSh_{\SSS}(\BB)}$.\qed
\end{corollary}

\begin{proof}[Proof of Proposition~\ref{prop:mappingObjectsInternalUniverse}]
	We have to show that there is a cartesian square
	\begin{equation*}
		\begin{tikzcd}[column sep=large]
			\iFun{P}{Q} \arrow[d]\arrow[r] & \Univ_1\arrow[d, "{(d_1, d_0)}"] \\
			A\arrow[r, "{(g,h)}"] & \Univ_0\times\Univ_0 
		\end{tikzcd}
	\end{equation*}
	in $\BBB$ or equivalently in $\PSh_{\SSS}(\BB)$. As Remark~\ref{rem:BCUniverse} implies that the pullback functor $\pi_A^{\ast}\colon \BBB\to\widehat{\Over{\BB}{A}}$
	carries the universe of $\BB$ to the universe of $\Over{\BB}{A}$, we may assume without loss of generality that $A\simeq 1$ holds.
	By, Corollary~\ref{cor:RightFibrationInternalUniverse}, the map $(d_1, d_0)\colon\Univ_1\to\Univ_0\times\Univ_0$ corresponds to the map of right fibrations obtained by restricting the functor $\Fun(\Delta^2,\BB)\to\Fun(\Lambda^2_2,\BB)$
	to the subcategories spanned by the cartesian edges.
	Since the core functor $(-)^{\core}\colon \PSh_{\CatS}(\BB)\to\PSh_{\SSS}(\BB)$ commutes with limits, it therefore suffices to show that there is a cartesian square
	\begin{equation*}
		\begin{tikzcd}
			\int \iFun{P}{Q} \arrow[d]\arrow[r] & \Fun(\Delta^2,\BB)\arrow[d]\\
			\BB\arrow[r] & \Fun(\Lambda^2_2, \BB)
		\end{tikzcd}
	\end{equation*}
	of $\infty$-categories, where the lower horizontal map sends an object $A\in\BB$ to the diagram
	\begin{equation*}
		\begin{tikzcd}
			& A\times Q\arrow[d, "\pr_0"]\\
			A\times P\arrow[r, "\pr_0"] & {A}.
		\end{tikzcd}
	\end{equation*}
	By definition, the right fibration $\int \iFun{G}{H}\to\BB$ arises as the pullback
	\begin{equation*}
		\begin{tikzcd}
			\int \iFun{P}{Q}\arrow[d]\arrow[r] & \Fun(\Delta^1,\BB)\arrow[d, "{(d_1,d_0)}"]\\
			\BB\arrow[r, "{(-\times G, H)}"] & \BB\times\BB.
		\end{tikzcd}
	\end{equation*}
	By factoring the functor $-\times G$ into the composite $\BB\to \Fun(\Delta^1,\BB)\to\BB$
	in which the first arrow acts by sending $A\in\BB$ to the projection $\pr_0\colon A\times G\to A$ and the second map is given by $d^1$, we conclude that there is a pullback square
	\begin{equation*}
		\begin{tikzcd}
			\int \iFun{P}{Q}\arrow[d]\arrow[r] & \Fun(\Lambda^2_0,\BB)\arrow[d, "{(d_{\{0,2\}}, d_{\{1\}})}"] \\
			\BB\arrow[r] & \Fun(\Delta^1\sqcup\Delta^0,\BB).
		\end{tikzcd}
	\end{equation*}
	Let $K$ be the simplicial set that is obtained by the pushout
	\begin{equation*}
		\begin{tikzcd}
			\Delta^0\arrow[d, "d^{\{1\}}"]\arrow[r, "d^1"] \arrow[dr, phantom, very near end, "\lrcorner"] & \Delta^1\arrow[d]\\
			\Lambda^2_2\arrow[r] & K
		\end{tikzcd}
	\end{equation*}
	and let $\iota\colon \Delta^1\sqcup \Delta^0\into K$ be the map that is determined by the two inclusions $d^{\{0,2\}}\colon \Delta^1\to \Lambda^2_2$ and $d^0\colon \Delta^0\to\Delta^1$. Then the functor $\BB\to \Fun(\Lambda^2_2,\BB)$ that is determined by the global section $(g,h)\colon 1\to \Univ_0\times\Univ_0$ can be decomposed into the composition
	\begin{equation*}
			\BB\to \Fun(\Delta^1\sqcup\Delta^0,\BB)\xrightarrow{\iota_\ast} \Fun(K, \BB)\to \Fun(\Lambda^2_2,\BB) 
	\end{equation*}
	in which $\iota_\ast$ denotes the functor of right Kan extension along $\iota$. Let $L$ be the simplicial set that is defined by the pushout
	\begin{equation*}
		\begin{tikzcd}
			\Delta^0\arrow[d, "d^{\{1\}}"]\arrow[r, "d^1"] \arrow[dr, phantom, very near end, "\lrcorner"] & \Delta^1\arrow[d]\\
			\Delta^2\arrow[r] & L.
		\end{tikzcd}
	\end{equation*}
	One then obtains a commutative diagram
	\begin{equation*}
		\begin{tikzcd}
			\Fun(\Lambda^2_0,\BB) \arrow[d, "{(d_{\{0,2\}}, d_{\{1\}})}"']\arrow[r, hookrightarrow] & \Fun(L, \BB)\arrow[d]\arrow[r]& \Fun(\Delta^2,\BB)\arrow[d]\\
			\Fun(\Delta^1\sqcup \Delta^0, \BB)\arrow[r,hookrightarrow, "\iota_\ast"] & \Fun(K, \BB)\arrow[r]& \Fun(\Lambda^2_2,\BB).
		\end{tikzcd}
	\end{equation*}
	in $\CatSS$ in which both squares are cartesian, hence the result follows.
\end{proof}

\subsection{Fully faithful and essentially surjective functors}
\label{sec:fullyFaithful}
A \emph{functor} between two $\BB$-categories $\I{C}$ and $\I{D}$ is simply defined to be a map $f\colon \I{C}\to\I{D}$ in $\Cat(\BB)$. Using Proposition~\ref{prop:categoriesExponentialIdeal}, there is a $\BB$-category $\iFun{\I{C}}{\I{D}}$ whose objects in context $A\in \BB$ are given by the functors $\pi_A^\ast\I{C}\to \pi_A^\ast\I{D}$. A map in $\iFun{\I{C}}{\I{D}}$ is referred to as a \emph{morphism of functors} or alternatively as a \emph{natural transformation}, the datum of such a morphism (in context $A$) is given by a map $\Delta^1\otimes\pi_A^\ast\I{C}\to \pi_A^\ast\I{D}$ in $\Cat(\Over{\BB}{A})$.

\begin{definition}
	\label{def:fullyFaithfulEssentiallySurjective}
	A functor $\I{C}\to  \I{D}$ between $\BB$-categories is said to be \emph{fully faithful} if it is internally right orthogonal to the map $(d^1,d^0)\colon\Delta^0\sqcup\Delta^0\to\Delta^1$. Dually, a functor is \emph{essentially surjective} if is (internally) left orthogonal to the class of fully faithful functors.
\end{definition}
\begin{remark}
	A functor $f\colon\CC\to\DD$ between $\infty$-categories is essentially surjective in the sense of Definition~\ref{def:fullyFaithfulEssentiallySurjective} if and only if every object in $\DD$ is equivalent to an object in the image of $f$, i.e.\ if and only if it is essentially surjective in the usual sense of the term. We will show this in~\ref{cor:esoCoverCore} below.
\end{remark}

By Proposition~\ref{prop:factorisationSystemInternallyGenerated}, the two classes of essentially surjective and fully faithful functors form an orthogonal factorisation system in $\Cat(\BB)$. In particular, one obtains:
\begin{proposition}
	\label{prop:fundamentalThmInternalCategories}
	Let $f\colon  \I{C}\to \I{D}$ be a functor between $\BB$-categories. Then $f$ is an equivalence if and only if $f$ is fully faithful and essentially surjective.\qed
\end{proposition}
Since fully faithful functors are by definition \emph{internally} right orthogonal to the map $(d^1,d^0)\colon\Delta^0\sqcup\Delta^0\to\Delta^1$, the class of essentially surjective functors is stable under products with arbitrary $\BB$-categories. Dually, this means that the class of fully faithful functors must be stable under exponentiation. We therefore conclude:
\begin{proposition}
	If $f\colon \I{C}\to \I{D}$ is a fully faithful functor of $\BB$-categories and if $E$ is an arbitrary simplicial object in $\BB$, the induced functor $f_\ast\colon \iFun{E}{\I{C}}\to \iFun{E}{\I{D}}$ is fully faithful as well.\qed
\end{proposition}

Let $f\colon \I{C}\to \I{D}$ be a functor between $\BB$-categories. By definition, $f$ being fully faithful precisely means that the square
\begin{equation*}
	\begin{tikzcd}
	\I{C}^{\Delta^1}\arrow[r, "f^{\Delta^1}"]\arrow[d] & \I{D}^{\Delta^1}\arrow[d]\\
	\I{C}\times\I{C}\arrow[r, "f\times f"] & \I{D}\times\I{D}
	\end{tikzcd}
\end{equation*}
is cartesian. Applying the core functor, this implies that the square
\begin{equation*}
	\begin{tikzcd}
	\I{C}_1\arrow[r, "f_1"]\arrow[d] & \I{D}_1\arrow[d]\\
	\I{C}_0\times \I{C}_0\arrow[r, "f_0\times f_0"] & \I{D}_0\times \I{D}_0
	\end{tikzcd}
\end{equation*}
is cartesian as well. In fact, the latter square being cartesian is even a sufficient criterion for $f$ to be fully faithful. The proof of this statement requires the following combinatorial lemma:
\begin{lemma}
	\label{lem:essentiallySurjectiveStabilitySimplices}
	Lett $A\in \BB$ be an arbitrary object and let $S$ be a saturated class of morphisms in $\Simp\BB$ that contains the maps $I^n\otimes A\into\Delta^n\otimes A$ for all $n\geq 0$. If $S$ contains the map $(d^1,d^0)\colon A\sqcup A\to\Delta^1\otimes A$, then it also contains the map $(d^1,d^0)\colon(\Delta^n\otimes A)\sqcup(\Delta^n\otimes A)\to(\Delta^1\times\Delta^n)\otimes A$ for any integer $n\geq 0$.
\end{lemma}
\begin{proof}
	By replacing $\BB$ with $\Over{\BB}{A}$, we may assume without loss of generality $A\simeq 1$. As all maps in the statement of the lemma are contained in the essential image of $\const\colon\SS\to\BB$, we may further assume $\BB\simeq \SS$. Furthermore, as the inclusion $I^n\sqcup I^n\into\Delta^n\sqcup \Delta^n$ is contained in $S$, it suffices to show that the map $I^n\sqcup I^n\into \Delta^1\times\Delta^n$ is contained in $S$ as well. By Lemma~\ref{lem:internalExternalSegalConditions} the map $\Delta^1\times I^n\into\Delta^1\times\Delta^n$ is contained in $S$, hence we need only show that also the map $I^n\sqcup I^n\into \Delta^1\times I^n$ is an element of $S$. Now $I^n$ being defined as the colimit $\Delta^1\sqcup_{\Delta^0}\cdots\sqcup_{\Delta^0}\Delta^1$, we can assume without loss of generality $n=1$. Using the decomposition $\Delta^1\times\Delta^1\simeq\Delta^2\sqcup_{\Delta^1}\Delta^2$, one easily sees that this map is the pushout in $\Fun(\Delta^1,\Simp{\SS})$ that is obtained by glueing the two maps
	\begin{equation*}
		(d^{\{0,1\}}, d^{\{2\}})\colon \Delta^1\sqcup\Delta^0\into\Delta^2
	\end{equation*}
	and
	\begin{equation*}
		(d^{\{0\}}, d^{\{1,2\}})\colon\Delta^0\sqcup\Delta^1\into\Delta^2
	\end{equation*}
	along the morphism $(d^1,d^0)\colon\Delta^0\sqcup\Delta^0\into\Delta^1$. The proof is therefore finished once we show that the two maps above are contained in $S$. We show this for the first one, the case of the second one is completely analogous. Making use once more of the assumption that the spine inclusion $I^2\into\Delta^2$ is contained in $S$, it suffices to show that the map
	\begin{equation*}
		(d^{\{0,1\}}, d^{\{2\}})\colon \Delta^1\sqcup\Delta^0\into I^2
	\end{equation*}
	is contained in $S$. This follows from the observation that this map is obtained by glueing $\Delta^0\sqcup\Delta^0\into\Delta^1$ and $\id_{\Delta^1}$ along $\id_{\Delta^0}$ in $\Fun(\Delta^1, \Simp\SS)$.
\end{proof}

\begin{proposition}
	\label{prop:equivalentConditionsFullyFaithful}
	A functor $f\colon \I{C}\to\I{D}$ between $\BB$-categories is fully faithful if and only if the square
	\begin{equation*}
	\begin{tikzcd}
	\I{C}_1\arrow[r, "f_1"]\arrow[d] & \I{D}_1\arrow[d]\\
	\I{C}_0\times \I{C}_0\arrow[r, "f_0\times f_0"] & \I{D}_0\times \I{D}_0
	\end{tikzcd}
	\end{equation*}
	is cartesian.
\end{proposition}
\begin{proof}
	We already observed above that $f$ being fully faithful implies that the square is cartesian. Conversely, the square being cartesian is equivalent to $f$ being right orthogonal to the set of maps
	\begin{equation*}
		S=\{(\Delta^0\sqcup\Delta^0)\otimes A\to \Delta^1\otimes A~\vert~A\in \BB\}
	\end{equation*}
	in $\Cat(\BB)$. By Lemma~\ref{lem:essentiallySurjectiveStabilitySimplices}, the saturation of $S$ in $\Cat(\BB)$ contains the maps $(\Delta^n\sqcup\Delta^n)\otimes A\to (\Delta^1\times\Delta^n)\otimes A$ for $A\in \BB$ and $n\geq 0$, which translates into the statement that the induced square
	\begin{equation*}
		\begin{tikzcd}
		(\I{C}^{\Delta^1})_n\arrow[r, "f^{\Delta^1}_n"]\arrow[d] &( \I{D}^{\Delta^1})_n\arrow[d]\\
		\I{C}_n\times \I{C}_n\arrow[r, "f_n\times f_n"] & \I{D}_n\times \I{D}_n
		\end{tikzcd}
	\end{equation*}
	is a pullback square for all $n\geq 0$. As limits in $\Cat(\BB)$ can be computed on the underlying simplicial objects, this shows that $f$ is fully faithful.
\end{proof}

\begin{proposition}
	\label{prop:characterisationInternalFullyFaithfulMappingGroupoids}
	Let $f \colon  \I{C}\to  \I{D}$ be a functor between large $\BB$-categories. Then the following are equivalent:
	\begin{enumerate}
		\item The functor $ f $ is fully faithful;
		\item for any $A\in\BB$ and any two objects $c_0,c_1\colon A\to \I{C}$ in context $A$, the morphism
		\begin{equation*}
			\map{\I{C}}(c_0,c_1)\to\map{\I{D}}(f(c_0), f(c_1))
		\end{equation*}
		that is induced by $ f $ is an equivalence in $\Over{\BBB}{A}$;
		\item for every $A\in \BB$ the functor $f(A)\colon \I{C}(A)\to\I{D}(A)$ of $\infty$-categories is fully faithful;
		\item the map of cartesian fibrations over $\BB$ that is determined by $ f $ is fully faithful.
	\end{enumerate}
\end{proposition}
\begin{proof}
	The equivalence of the last two conditions follows from Lemma~\ref{lem:PShCotensored} and Remark~\ref{rem:tensoringCotensoringSheaves}. As moreover limits in $\Shv_{\CatSS}(\BB)$ are computed objectwise, it is clear that the first and the third condition are equivalent.
	
	Suppose now that $f$ is fully faithful. For any object $A\in\BB$ and any two objects $c_0,c_1\colon A\to \I{C}$  in context $A$, the map
	\begin{equation*}
	\map{\I{C}}(c_0,c_1)\to \map{\I{D}}(f(c_0), f(c_1))
	\end{equation*}
	is defined by the commutative diagram
	\begin{equation*}
	\begin{tikzcd}[column sep={5em,between origins},row sep={2em,between origins}]
	& \map{\I{C}}(c_0,c_1)\arrow[dl]\arrow[rr]\arrow[dd] & & \map{\I{D}}(f(c_0), f(c_1))\arrow[dl]\arrow[dd] \\
	\I{C}_1 \arrow[rr, crossing over, "f_1", near end]\arrow[dd]& & \I{D}_1 & \\
	& A \arrow[rr, "\id", near start]\arrow[dl] & & A\arrow[dl] \\
	\I{C}_0\times \I{C}_0 \arrow[rr, "f_0\times f_0"]& & \I{D}_0\times \I{D}_0. \arrow[from=uu,crossing over]
	\end{tikzcd}
	\end{equation*}
	in which the two vertical squares on the left and on the right are cartesian. As $f$ is fully faithful, the square in the front is cartesian, hence the square in the back must be cartesian as well, which implies that the second condition holds. Conversely, suppose that $f$ induces an equivalence on mapping groupoids. 
	As the object $C_0\times C_0$ in $\BBB$ is obtained as the colimit of the diagram $\Over{\BB}{C_0\times C_0}\to\BB\into\BBB$, we obtain a cover
	\begin{equation*}
		\bigsqcup_{A\to \I{C}_0\times\I{C}_0} A\onto \I{C}_0\times\I{C}_0
	\end{equation*}
	in $\BBB$. By assumption, pasting the front square in the above diagram with the pullback square
	\begin{equation*}
		\begin{tikzcd}
			\bigsqcup_{A\to \I{C}_0\times\I{C}_0} \map{\I{C}}(c_0,c_1)\arrow[d]\arrow[r, twoheadrightarrow] & \I{C}_1\arrow[d]\\
			\bigsqcup_{A\to \I{C}_0\times\I{C}_0} A\arrow[r, twoheadrightarrow] & \I{C}_0\times\I{C}_0
		\end{tikzcd}
	\end{equation*}
	results in a pullback square, hence the claim follows from the fact that the \'etale base change along a cover in an $\infty$-topos constitutes a conservative algebraic morphism.
\end{proof}

\begin{lemma}
	\label{lem:localisationIsSurjective}
	The map $\Delta^1\to\Delta^0$ is essentially surjective as a functor of $\BB$-categories.
\end{lemma}
\begin{proof}
	Algebraic morphisms preserve essential surjectivity since dually geometric morphisms preserve full faithfulness. We may therefore assume without loss of generality $\BB\simeq \SS$.
	Let $S$ be the saturated class of maps in $\Simp\SS$ that is generated by $\Delta^0\sqcup\Delta^0\into\Delta^1$, the spine inclusions $I^n\into\Delta^n$ for $n\geq 0$ as well as the map $E^1\to 1$, where $E^1$ denotes the walking equivalence. It suffices to show that $\Delta^1\to\Delta^0$ is contained in $S$. Let $K\into\Delta^3$ be the unique map of simplicial $\infty$-groupoids that fits into the commutative diagram
	\begin{equation*}
		\begin{tikzcd}
		\Delta^1\arrow[d, "d^{\{1,2\}}"']\arrow[r, "d^{\{0,1\}}"]\arrow[dr, phantom, "\lrcorner", very near end] & \Delta^2\arrow[d]\arrow[ddr, "d^{\{1,2,3\}}", bend left]& \\
		\Delta^2\arrow[r] \arrow[drr, "d^{\{0,1,2\}}"',bend right]& K \arrow[dr, hookrightarrow]& \\
		& & \Delta^3.
		\end{tikzcd}
	\end{equation*}
	The inclusion $K\into\Delta^3$ is contained in $S$: in fact, as the inclusion $I^3\into \Delta^3$ is an element of $S$ and factors through $K\into\Delta^3$, it suffices to observe that the map $I^3\into K$ can be obtained by glueing two copies of the inclusion $I^2\into\Delta^2$ along the $\id_{\Delta^1}$ in $\Fun(\Delta^1,\Simp\SS)$. We now obtain a commutative diagram
	\begin{equation*}
		\begin{tikzcd}
		\Delta^1\sqcup\Delta^1\arrow[d] \arrow[dr, phantom, "\lrcorner", very near end]\arrow[r, hookrightarrow]& K\arrow[r, hookrightarrow]\arrow[d]\arrow[dr, phantom, "\lrcorner", very near end] & \Delta^3\arrow[d]\\
		\Delta^0\sqcup\Delta^0\arrow[r] & L\arrow[r, hookrightarrow] & E^1\\
		\end{tikzcd}
	\end{equation*}
	in which the upper left horizontal map is induced by postcomposing the inclusion $d^{\{0,2\}}\colon\Delta^1\into\Delta^2$ with the two maps $\Delta^2\rightrightarrows K$ that are defined by the pushout square. As $K\into\Delta^3$ is contained in $S$, we conclude that the induced map $L\into E^1$ must be in $S$ as well. As a consequence, the terminal map $L\to \Delta^0$ is an element of $S$ too. Let $\Delta^1\to K$ be the composite map in the pushout square that defines $K$. Postcomposing with the map $K\to L$ from the previous diagram gives rise to a map $\Delta^1\to L$. We finish the proof by showing that this map is contained in $S$. Let $H$ be defined by the pushout square
	\begin{equation*}
		\begin{tikzcd}
		\Delta^1\arrow[d,hookrightarrow, "d^{\{0,2\}}"'] \arrow[r]\arrow[dr, phantom, "\lrcorner", very near end] & \Delta^0\arrow[d]\\
		\Delta^2\arrow[r] & H.
		\end{tikzcd}
	\end{equation*}
	Then the map $\Delta^1\to L$ is recovered as the composite map in the cocartesian square
	\begin{equation*}
		\begin{tikzcd}
		\Delta^1\arrow[d, "\alpha"]\arrow[r, "\beta"]\arrow[dr, phantom, "\lrcorner", very near end] & H\arrow[d]\\
		H\arrow[r] & L,
		\end{tikzcd}
	\end{equation*}
	in which the two maps $\alpha$ and $\beta$ are given by postcomposing $d^{\{1,2\}}\colon\Delta^1\into\Delta^2$ and $d^{\{0,1\}}\colon\Delta^1\into\Delta^2$, respectively, with the map $\Delta^2\to H$. As a consequence, we only need to verify that $\alpha,\beta\in S$. We will show this for $\alpha$, the case of the map $\beta$ is analogous. Consider the commutative diagram
	\begin{equation*}
		\begin{tikzcd}
		\Delta^0 \arrow[r, hookrightarrow, "d^0"]\arrow[d,hookrightarrow, "d^0"]\arrow[dr, phantom, "\lrcorner", very near end]& \Delta^1 \arrow[r]\arrow[d, hookrightarrow]\arrow[dr, phantom, "\lrcorner", very near end] & \Delta^0\arrow[d, hookrightarrow, "d^0"]\\
		\Delta^1\arrow[r, hookrightarrow]\arrow[dr, hookrightarrow, "d^{\{1,2\}}"'] & \Lambda^2_0\arrow[r]\arrow[d, hookrightarrow]\arrow[dr, phantom, "\lrcorner", very near end] & \Delta^1\arrow[d, "\alpha"]\\
		 & \Delta^2\arrow[r] & H
		\end{tikzcd}
	\end{equation*}
	in which the composite of the two vertical maps in the middle is given by the inclusion $d^{\{0,2\}}\colon\Delta^1\into\Delta^2$. As maps in $S$ are stable under pushouts, we only need to show that the inclusion $\Lambda^2_0\into\Delta^2$ is contained in $S$. Consider the commutative square
	\begin{equation*}
		\begin{tikzcd}[column sep=huge]
		\Delta^0\sqcup\Delta^1\arrow[r, hookrightarrow]\arrow[d, hookrightarrow] & \Lambda^2_0\arrow[d, hookrightarrow] \\
		I^2\arrow[r, hookrightarrow] & \Delta^2
		\end{tikzcd}
	\end{equation*}
	that is uniquely determined by the condition that the composite map is induced by $d^{\{0\}}\colon\Delta^0\into\Delta^2$ and $d^0\colon\Delta^1\into\Delta^2$. As the lower horizontal map is contained in $S$ by assumption on $S$, it suffices to show that the two maps from $\Delta^0\sqcup\Delta^1$ are in $S$ as well. This follows immediately from the observation that both of these maps can be obtained as a pushout of the map $\Delta^0\sqcup\Delta^0\into\Delta^1$.
\end{proof}

\begin{proposition}
	\label{prop:stabilityEsoFFCoreGroupoidification}
	 The groupoidification functor preserves essential surjectivity when viewed as a functor $(-)^{\gp}\colon\Cat(\BB)\to\Cat(\BB)$ . Dually, the core $\BB$-groupoid functor $(-)^{\core}\colon\Cat(\BB)\to\Cat(\BB)$ preserves full faithfulness.
\end{proposition}
\begin{proof}
	Let $f\colon \I{C}\to\I{D}$ be an essentially surjective functor, and consider the commutative diagram
	\begin{equation*}
		\begin{tikzcd}
		\I{C}\arrow[r, "f"]\arrow[d] & \I{D}\arrow[d]\\
		\I{C}^{\gp}\arrow[r, "f^{\gp}"] & \I{D}^{\gp}
		\end{tikzcd}
	\end{equation*}
	in which the two vertical maps are obtained from the adjunction unit. In order to show that $f^{\gp}$ is essentially surjective, it suffices to show that the adjunction unit $\I{C}\to\I{C}^{\gp}$ is essentially surjective for every $\BB$-category $\I{C}$. As this map is contained in the saturated class internally generated by $\Delta^1\to\Delta^0$, it suffices to show that the map $\Delta^1\to\Delta^0$ is essentially surjective, which is the content of Lemma~\ref{lem:localisationIsSurjective}.
	
	Suppose now that $f\colon\I{C}\to\I{D}$ is fully faithful. To show that $f^{\core}$ is fully faithful as well, we need to show that it is contained in the right complement of the class of essentially surjective maps. But by the adjunction $(-)^{\gp}\dashv (-)^{\core}$ (viewing both as endofunctors on $\Cat(\BB)$) a map $g$ is left orthogonal to $f^{\core}$ if and only if $g^{\gp}$ is left orthogonal to $f$, hence the claim follows from the first part of the proof.
\end{proof}
\begin{proposition}
	\label{prop:FFMonomorphism}
	Every fully faithful functor of $\BB$-categories is a monomorphism in $\Cat(\BB)$. 
\end{proposition}
\begin{proof}
	By Lemma~\ref{lem:localisationIsSurjective}, the map $\Delta^1\to\Delta^0$ is essentially surjective, hence the map $\Delta^0\sqcup\Delta^0\to\Delta^0$ is essentially surjective as well. Since monomorphisms are precisely those maps that are internally right orthogonal to this map, the claim follows.
\end{proof}

For fully faithful functors between \emph{groupoids} in an $\infty$-topos, the converse of Proposition~\ref{prop:FFMonomorphism} is true as well:
\begin{corollary}
	\label{cor:FFMonomorphismsGroupoids}
	For a map $f\colon \I{G}\to \I{H}$ between $\BB$-groupoids, the following are equivalent:
	\begin{enumerate}
		\item $f$ is fully faithful;
		\item $f$ is a monomorphism.
	\end{enumerate}
\end{corollary}
\begin{proof}
	By Proposition~\ref{prop:FFMonomorphism}, if $f$ is fully faithful, then $f$ is a monomorphism in $\Cat(\BB)$.
	
	Conversely, $f$ being a monomorphism precisely means that $f$ is internally right orthogonal to the map $\Delta^0\sqcup\Delta^0\to\Delta^0$. As the latter is the image of the map $\Delta^0\sqcup\Delta^0\to\Delta^1$ under the groupoidification functor $(-)^{\gp}\colon\Cat(\BB)\to\Cat(\BB)$, the claim follows by adjunction from the assumption that $f$ is internally right orthogonal to $\Delta^0\sqcup\Delta^0\to\Delta^0$.
\end{proof}

Dually, one finds for essentially surjective functors between groupoids:
\begin{corollary}
	\label{cor:esoCoverGroupoids}
	For a map $f\colon \I{G}\to \I{H}$ between $\BB$-groupoids, the following are equivalent:
	\begin{enumerate}
		\item $f$ is essentially surjective;
		\item $f_0$ is a cover in $\BB$.
	\end{enumerate}
\end{corollary}
\begin{proof}
Let $\iota\colon\BB\into\Cat(\BB)$ be the inclusion. If $f$ is essentially surjective and if $g$ is a monomorphism in $\BB$, then $\iota(g)$ is a monomorphism between groupoids and therefore fully faithful by Corollary~\ref{cor:FFMonomorphismsGroupoids}. Since $f_0$ is left orthogonal to $g$ if and only if $f$ is left orthogonal to $\iota(g)$, this shows that $f_0$ is a cover.

Conversely, assuming $f_0$ is a cover and that $g$ is a fully faithful map between $\BB$-categories, then $g_0$ is a monomorphism by Proposition~\ref{prop:FFMonomorphism}, hence $f_0$ is left orthogonal to $g_0$. Since $f\simeq \iota(f_0)$, the adjunction between $(-)_0$ and $\iota$ implies that this is equivalent to $f$ being left orthogonal to $g$, i.e.\ to $f$ being essentially surjective.
\end{proof}

\subsection{Full subcategories}
\label{sec:subcategories}
Let $\I{D}$ be a $\BB$-category. A map $f\colon \I{C}\to\I{D}$ is a monomorphism in $\Cat(\BB)$ if and only if $f$ is a $(-1)$-truncated object in $\Over{\Cat(\BB)}{\I{D}}$. By~\cite[Proposition~6.2.1.3]{htt} the full subcategory of $(-1)$-truncated objects in $\Over{\Cat(\BB)}{\I{D}}$ forms a small partially ordered set that we denote by $\Sub(\I{D})$.
By Proposition~\ref{prop:FFMonomorphism}, every fully faithful functor is a monomorphism. We may therefore define:
\begin{definition}
	Let $\I{D}$ be a $\BB$-category. A \emph{full} subcategory of $\I{D}$ is a fully faithful functor $\I{C}\into\I{D}$. The collection of full subcategories of $\I{D}$ spans a partially ordered subset of $\Sub(\I{D})$ that we denote by $\Sub^{\mathrm{full}}(\I{D})$.
\end{definition}

As in ordinary category theory, a full subcategory of a $\BB$-category should be uniquely specified by the collection of objects that are contained in it. Hereafter our goal is to turn this heuristic into a precise statement.
To that end, note that the functor $(-)_0\colon \Simp\BB\to \BB$ admits a right adjoint $\Cech(-)$ that sends an object $A\in\BB$ to its \emph{\v Cech nerve $\Cech(A)$}. Now if $\I{D}$ is an arbitrary $\BB$-category, the functor $(-)_0\colon \Over{(\Simp\BB)}{\I{D}}\to \Over{\BB}{\I{D}_0}$ admits a right adjoint $\Gen{-}[\I{D}]$ that is given by the composition
\begin{equation*}
	\Over{\BB}{\I{D}_0}\xrightarrow{\Cech} \Over{(\Simp\BB)}{\Cech(\I{D}_0)}\xrightarrow{\eta^\ast} \Over{(\Simp\BB)}{\I{D}}
\end{equation*}
in which $\eta\colon \I{D}\to \Cech(\I{D}_0)$ denotes the adjunction unit. As $\Cech$ is fully faithful, so is the functor $\Gen{-}[\I{D}]$.
\begin{lemma}
	\label{lem:generatedSubcategory}
	For every $\BB$-category $\I{D}$ and any monomorphism $P\into\I{D}_0$, the simplicial object $\Gen{P}[\I{D}]$ is a $\BB$-category, and the functor $\Gen{P}[\I{D}]\to \I{D}$ is fully faithful.
\end{lemma}
\begin{proof}
	By construction, the map $\Gen{P}[\I{D}]\to\I{D}$ fits into a cartesian square
	\begin{equation*}
		\begin{tikzcd}
		\Gen{P}[\I{D}]\arrow[r]\arrow[d] & \I{D}\arrow[d, "\eta"]\\
		\Cech(P)\arrow[r] & \Cech(\I{D}_0).
		\end{tikzcd}
	\end{equation*}
	To show that $\Gen{P}[\I{D}]$ is a $\BB$-category, it therefore suffices to show that the map $\Cech(P)\to\Cech(\I{D}_0)$ is internally right orthogonal to the two maps $E^1\to 1$ and $I^2\into\Delta^2$. This is in turn equivalent to the map $P\into\I{D}_0$ being internally right orthogonal (in $\BB$) to the two maps $(E^1)_0\to 1$ and $(I^2)_0\into (\Delta^2)_0$. As the first one is a cover in $\BB$ and the second one is an equivalence, this is immediate. By the same argumentation, the functor $\Gen{P}[\I{D}]\to \I{D}$ is fully faithful precisely if the map $P\into\I{D}_0$ is internally right orthogonal to $(\Delta^0\sqcup\Delta^0)_0\to(\Delta^1)_0$, which follows from the observation that this map is an equivalence in $\BB$.
\end{proof}
As a consequence of Lemma~\ref{lem:generatedSubcategory}, the functor $\Gen{-}[\I{D}]$ restricts to an embedding
\begin{equation*}
\Gen{-}[\I{D}]\colon \Sub(\I{D}_0)\into \Sub^{\mathrm{full}}(\I{D})
\end{equation*}
of partially ordered sets.
\begin{proposition}
	\label{prop:subcategoriesGeneration}
	For any $\BB$-category $\I{D}$, the map $\Gen{-}[\I{D}]\colon \Sub(\I{D}_0)\into \Sub^{\mathrm{full}}(\I{D})$ is an equivalence.
\end{proposition}
\begin{proof}
	It suffices to show that the map is essentially surjective. Let therefore $\I{C}\into\I{D}$ be a full subcategory of $\I{D}$. Using Corollary~\ref{cor:FFMonomorphismsGroupoids}, the induced map $\I{C}_0\to\I{D}_0$ is a monomorphism in $\BB$. We therefore obtain a factorisation
	\begin{equation*}
		\I{C}\into \Gen{\I{C}_0}[\I{D}]\into\I{D}
	\end{equation*}
	in which the first map is fully faithful since the second map and the composite map are fully faithful. As moreover the map $\I{C}\into\Gen{\I{C}_0}[\I{D}]$ induces an equivalence on level $0$, Proposition~\ref{prop:equivalentConditionsFullyFaithful} implies that it must be an equivalence on level $1$ as well. Together with the Segal condition, this implies that this map is an equivalence, which completes the proof.
\end{proof}

\begin{proposition}
	\label{prop:mappingPropertyFullSubcategory}
	Let $f\colon\I{C}\to\I{D}$ be a functor between large $\BB$-categories and let $\I{E}\into\I{D}$ be a full subcategory. Then the following are equivalent:
	\begin{enumerate}
		\item $f$ factors through the inclusion $\I{E}\into\I{D}$;
		\item $f^{\core}$ factors through $\I{E}^{\core}\into\I{D}^{\core}$;
		\item for every object $c$ in $\I{C}$ in context $A\in\BB$ its image $f(c)$ is contained in $\I{E}$.
	\end{enumerate}
\end{proposition}
\begin{proof}
	Clearly~(1) implies~(2). By making use of the adjunction $(-)_0\dashv\Gen{-}[\I{D}]$ and Proposition~\ref{prop:subcategoriesGeneration}, one finds that conversely~(2) implies~(1). A fortiori~(2) implies~(3). Conversely, suppose that for any $c\colon A\to\I{C}$ the composite map $A\to\I{C}_0\to\I{D}_0$ factors through $\I{E}_0\into\I{D}_0$. As the map
	\begin{equation*}
		\bigsqcup_{A\to \I{C}_0} A\onto \I{C}_0
	\end{equation*}
	defines a cover in $\BBB$, the lifting problem
	\begin{equation*}
		\begin{tikzcd}
			\bigsqcup_{A\to \I{C}_0} A\arrow[r]\arrow[d, twoheadrightarrow] & \I{E}_0\arrow[d, hookrightarrow]\\
			\I{C}_0\arrow[r, "f"] \arrow[ur, dashed]& \I{D}_0
		\end{tikzcd}
	\end{equation*}
	admits a unique solution, which proves that~(2) holds.
\end{proof}

\begin{corollary}
	\label{cor:esoCoverCore}
	A map $f\colon\I{C}\to\I{D}$ between $\BB$-categories is essentially surjective if and only if $f_0\colon \I{C}_0\to\I{D}_0$ is a cover in $\BB$.
\end{corollary}
\begin{proof}
	Suppose first that $f_0$ is a cover, and let
	\begin{equation*}
			\I{C}\xrightarrow{p} \I{E}\xrightarrow{i} \I{D}
	\end{equation*}
	be the factorisation of $f$ into an essentially surjective and a fully faithful functor. We need to show that $i$ is an equivalence. Since $i$ is fully faithful,~Proposition~\ref{prop:subcategoriesGeneration} implies that this is the case if and only if $i_0$ is an equivalence. But since $f_0$ is a cover, the map $i_0$ is one as well and must therefore be an equivalence as it is already a monomorphism by Proposition~\ref{prop:FFMonomorphism}.
	
	Conversely, suppose that $f$ is essentially surjective, and let
	\begin{equation*}
			\I{C}_0\xrightarrow{p}  P\xrightarrow{i}  \I{D}_0
	\end{equation*}
	be the factorisation of $f_0$ into a cover and a monomorphism in $\BB$. We need to show that $i$ is an equivalence. By Proposition~\ref{prop:subcategoriesGeneration}, the object $P$ determines a full subcategory $\Gen{P}[\I{D}]$ of $\I{D}$ and since $f_0$ factors through $P$, Proposition~\ref{prop:mappingPropertyFullSubcategory} implies that $f$ factors through a map $\I{C}\to\Gen{P}[\I{D}]$. It suffices to show that this functor is essentially surjective. Let $\I{C}\onto \I{E}\into\Gen{P}[\I{D}]$ be the factorisation of this functor into an essentially surjective and a fully faithful functor. Then $p\colon \I{C}_0\to P$ factors through a monomorphism $\I{E}_0\into P$, but since $p$ is a cover this map must be a cover as well and therefore an equivalence. Proposition~\ref{prop:subcategoriesGeneration} then implies that the map $\I{E}\into\Gen{P}[\I{D}]$ is an equivalence and therefore that the functor $\I{C}\to\Gen{P}[\I{D}]$ is essentially surjective, as desired.
\end{proof}
\begin{definition}
	\label{def:essentialImage}
	Let $f\colon \I{C}\to\I{D}$ be a map in $\Cat(\BB)$ and let $\I{C}\onto\I{E}\into\I{D}$ be the factorisation of $f$ into an essentially surjective and a fully faithful functor. Then the full subcategory $\I{E}\into\I{D}$ is referred to as the \emph{essential image} of $f$.
\end{definition}
\begin{definition}
	\label{def:generatedFullSubcategory}
	Let $\I{D}$ be a $\BB$-category and let $(d_i\colon A_i\to \I{D})_{i\in I}$ be a small family of objects in $\I{D}$. The essential image of the induced map $\bigsqcup_i A_i\to\I{D}$ is referred to as the full subcategory of $\I{D}$ that is \emph{generated} by the family $(d_i)_{i\in I}$.
\end{definition}
In the context of Definition~\ref{def:generatedFullSubcategory}, let $G\into \I{D}_0$ be the image of the map $(d_i)_{i\in I}\colon \bigsqcup_i A_i\to \I{D}_0$. Then Corollary~\ref{cor:esoCoverCore} implies that the full subcategory of $\I{D}$ generated by the family $(d_i)_{i\in I}$ is given by $\Gen{G}[\I{D}]$.
\begin{remark}
	\label{rem:subcategoryGeneratedObjectsEquivalence}
	Let $(s_i\colon A_i\to C)_{i\in I}$ be a small family of maps in $\BB$ and let $r\colon B\to C$ be another map such that for all $i\in I$ there is a cover $p_i\colon B\onto A_i$ such that $r\simeq s_ip_i$. One then obtains a commutative diagram
	\begin{equation*}
		\begin{tikzcd}
			B\arrow[dr, "r"', bend right]&\bigsqcup_i B\arrow[d, "(r)_{i\in I}"] \arrow[l, twoheadrightarrow, "{{(\id)_{i\in I}}}"'] \arrow[r, twoheadrightarrow, "\bigsqcup_i p_i"]& \bigsqcup_i A_i\arrow[dl, "{(s_i)_{i\in I}}", bend left]\\
			& C& 
		\end{tikzcd}
	\end{equation*}
	which shows that the image of $(s_i)_{i\in I}$ in $C$ coincides with the image of the map $r$. This shows that in the situation of Definition~\ref{def:generatedFullSubcategory}, we may assume without loss of generality that the objects are pairwise locally non-equivalent in the sense that for any pair $(i,j)\in I\times I$ it is \emph{not} possible to find covers $p\colon B\onto A_i$ and $q\colon B\onto A_j$ such that $d_ip\simeq d_j q$.
\end{remark}
\begin{remark}
	Let $\I{D}$ be a $\BB$-category and let $(d_i\colon A_i\to \I{D})_{i\in I}$ be a family of objects in $\I{D}$ that are pairwise locally non-equivalent in the sense of Remark~\ref{rem:subcategoryGeneratedObjectsEquivalence}. Then $I$ is $\bV$-small since for any $A\in \BB$ the set of equivalence classes of maps $A\to \I{D}_0$ is small and the set of equivalence classes of objects in $\BB$ is $\bV$-small. By viewing $\I{D}$ as a large $\BB$-category, the subcategory spanned by the family $(d_i)_{i\in I}$ is therefore well-defined. Note that this is still a small $\BB$-category since any full subcategory of a small $\BB$-category must be small as well.
\end{remark}

We conclude this section with a discussion of the poset of full subcategories of the universe $\Univ$ for $\BB$-groupoids. To that end, let us recall the notion of a \emph{local class} in an $\infty$-topos~\cite[\S~6.1.3]{htt}:
\begin{definition}
	\label{defn:localClass}
	Let $S$ be a collection of maps in $\BB$ that is stable under pullbacks. Then the full subcategory of $\Fun(\Delta^1,\BB)$ spanned by the maps in $S$ forms a cartesian subfibration of the codomain fibration $d_0\colon \Fun(\Delta^1, \BB)\to \BB$ that is classified by a $\CatSS$-valued presheaf $\Over{S}{-}$ on $\BB$. The class $S$ is said to be \emph{local} if $\Over{S}{-}$ is a sheaf and \emph{bounded} if $\Over{S}{-}$ takes values in $\CatS$.
\end{definition}
In the situation of Definition~\ref{defn:localClass},~\cite[Lemma~6.1.3.7]{htt} implies that the presheaf $\Over{S}{-}$ is a sheaf if and only if $(\Over{S}{-})^\core$ is an $\SSS$-valued sheaf, and since the latter takes values in $\SS$ if and only if $\Over{S}{-}$ takes values in $\CatS$, one obtains:
\begin{proposition}
	\label{prop:characterisationLocalClasses}
	Let $S$ be a collection of maps in $\BB$ that is stable under pullbacks. Then the following are equivalent:
	\begin{enumerate}
		\item $S$ is a (bounded) local class.
		\item $\Over{S}{-}$ is a ($\CatS$-valued) sheaf.
		\item $(\Over{S}{-})^\core$ is an ($\SS$-valued) sheaf.\qed
	\end{enumerate}
\end{proposition}
The set of local classes in $\BB$ can be identified with a subset of the partially ordered set $\Sub(\Fun(\Delta^1,\BB))$ and therefore inherits a partial order. One now finds:
\begin{proposition}
	\label{prop:classificationInternalSubcategories}
	There is an equivalence between the partially ordered set of local classes in $\BB$ and $\Sub^{\mathrm{full}}(\Univ)$ with respect to which bounded local classes in $\BB$ correspond to \emph{small} full subcategories of $\Univ$.
\end{proposition}
\begin{proof}
	If $S$ is a local class, Proposition~\ref{prop:characterisationLocalClasses} shows that $\Over{S}{-}$ is a $\CatSS$-valued sheaf and therefore corresponds to a large $\BB$-category $\Univ[S]$. By Proposition~\ref{prop:characterisationInternalFullyFaithfulMappingGroupoids}, this is a full subcategory of the universe of $\BB$. Conversely, if $\I{C}\hookrightarrow  \Univ$ exhibits $\I{C}$ as a full subcategory of $ \Univ$, Proposition~\ref{prop:characterisationLocalClasses} implies that the set of objects contained in the essential image of the associated inclusion $\int\I{C}\hookrightarrow \Fun(\Delta^1, \BB)$ of cartesian fibrations over $\BB$ defines a local class. Clearly these operations are inverse to each other and order-preserving. Applying Proposition~\ref{prop:characterisationLocalClasses} once more, one moreover sees that this equivalence restricts to an equivalence between the poset of bounded local classes and the poset of small full subcategories of $ \Univ$.
\end{proof}
\begin{definition}
	\label{def:subuniverse}
	A full subcategory $\I{C}\into\Univ$ of the universe is said to be a \emph{subuniverse} in $\BB$.
\end{definition}
In the situation of Proposition~\ref{prop:classificationInternalSubcategories}, the case where $S$ is a bounded local class deserves a more careful discussion. In this case, the sheaf $\Over{S}{-}$ is represented by the  $\BB$-category $\Univ[S]$, hence $(\Over{S}{-})^{\core}$ is representable by $\Univ[S]^{\core}$ which by Yoneda's lemma implies that the full subcategory $S\into\Fun(\Delta^1,\BB)$ admits a final object $\phi_S\colon \UnivHat[S]^\core\to\Univ[S]^\core$ that is referred to as the \emph{universal morphism} in $S$.
Hereafter, our goal is to reverse this discussion: Suppose that $p\colon P\to A$ is an arbitrary morphism in $\BB$, and denote by $\Gen{p}$ the class of morphisms in $\BB$ that arise as a pullback of $p$. Since $\Gen{p}$ is stable under pullbacks, the full subcategory of $\Fun(\Delta^1,\BB)$ spanned by the maps in $\Gen{p}$ defines a cartesian fibration over $\BB$ and is therefore classified by a $\CatSS$-valued presheaf $\Over{\Gen{p}}{-}$ on $\BB$.
We would like to investigate the conditions that ensure $\Gen{p}$ to be a bounded local class in $\BB$, with $p$ as a universal morphism.
\begin{definition}
	\label{defn:univalence}
	A map $p\colon P\to A$ in $\BB$ is \emph{univalent} if $\Gen{p}$ is a bounded local class.
\end{definition}

\begin{remark}
	\label{rem:universalMorphism}
	Let $S$ be a bounded local class of morphisms in $\BB$ and let $\phi_S\colon\UnivHat[S]^\core\to\Univ[S]^\core$ denote the associated universal morphism in $S$. Then a map in $\BB$ arises as a pullback of $\phi_S$ if and only if it is contained in $S$, hence the map $\phi_S$ is univalent.
\end{remark}

By Proposition~\ref{prop:characterisationLocalClasses} the notion of univalence admits the following equivalent characterisation:
\begin{proposition}
	\label{prop:UnivalenceRepresentabilitycartesianFibration}
	For a map $p\colon P\to A$ in $\BB$, the following conditions are equivalent:
	\begin{enumerate}
		\item $p$ is univalent.
		\item $\Over{\Gen{p}}{-}$ is representable by a $\BB$-category
		\item $(\Over{\Gen{p}}{-})^\core$ is representable by a $\BB$-groupoid.\qed
	\end{enumerate}
\end{proposition}
\begin{lemma}
	\label{lem:coversToposDiagrams}
	Let $\BB$ be an $\infty$-topos and let $\CC$ be a small $\infty$-category. A map $f\colon Y\to X$ in the $\infty$-topos $\Fun(\CC,\BB)$ is a cover if and only if $f(c)\colon Y(c)\to X(c)$ is a cover for every $c\in \CC$.
\end{lemma}
\begin{proof}
	For every $c\in\CC$, evaluation at $c$ defines an algebraic morphism $\Fun(\CC,\BB)\to \BB$, and since equivalences in $\Fun(\CC,\BB)$ are determined objectwise, the induced algebraic morphism $\Fun(\CC,\BB)\to\prod_{c\in\CC}\BB$ is conservative. Now $f$ is a cover if and only if the inclusion $\Image(f)\into X$ is an equivalence. Since algebraic morphisms preserve the image factorisation of a map and since conservative functors reflect equivalences, the claim follows.
\end{proof}
Suppose that $p\colon P\to A$ is a map in $\BB$ and let $g\colon A\to \Univ_0$ be the associated object in $\Univ$. By definition of $\Gen{p}$ and the fact that a map in $\PSh_{\SSS}(\BB)$ is a cover if and only if it is objectwise given by a cover in $\SSS$ (cf.\ Lemma~\ref{lem:coversToposDiagrams}), the image factorisation of $g$ in $\PSh_{\SSS}(\BB)$ is given by $A\onto (\Over{\Gen{p}}{-})^\core\into\Univ_0$. Therefore $p$ is univalent if and only if the cover $A\onto (\Over{\Gen{p}}{-})^\core$ is a monomorphism, which is the case if and only if $g$ itself is a monomorphism. We therefore conclude:
\begin{proposition}[{\cite[Proposition~3.8]{gepner2017}}]
	\label{prop:univalenceMonomorphism}
	Let $p\colon P\to A$ be a map in $\BB$ and let $g\colon A\to \Univ_0$ be the associated object in $\Univ$. Then $p$ is univalent if and only if $g$ is a monomorphism.\qed
\end{proposition}

\begin{corollary}
	\label{cor:univalenceObjectEquivalences}
	Let $p\colon P\to A$ be a map in $\BB$ and let $g\colon A\to \Univ_0$ be the associated object in $\Univ$. Let $\pr_i\colon A\times A\to A$ be the projection onto the $i$th factor for $i\in \{0,1\}$. Then $p$ is univalent if and only if the canonical map $\phi\colon A\to \Eq{\Univ}(\pr_0^\ast g,\pr_1^\ast g)$ in $\Over{\BB}{A\times A}$ is an equivalence.
\end{corollary}
\begin{proof}
	By Proposition~\ref{prop:univalenceMonomorphism}, the morphism $p$ is univalent precisely if $g\colon A\to \Univ_0$ is a monomorphism in $\BBB$, which is equivalent to the commutative square
	\begin{equation*}
	\begin{tikzcd}
	A\arrow[d, "{(\id,\id)}"] \arrow[r, "g"] & \Univ_0\arrow[d, "{(\id,\id)}"]\\
	A\times A\arrow[r, "g\times g"] & \Univ_0\times \Univ_0
	\end{tikzcd}
	\end{equation*}
	being cartesian. On account of the cartesian square
	\begin{equation*}
	\begin{tikzcd}
	\Eq{\Univ}(\pr_0^\ast g,\pr_1^\ast g) \arrow[d]\arrow[r] & \Univ_0\arrow[d, "{(\id,\id)}"]\\
	A\times A\arrow[r, "g\times g"] & \Univ_0\times \Univ_0,
	\end{tikzcd}
	\end{equation*}
	we see that this is the case if and only if the map $\phi$ is an equivalence.
\end{proof}
The object of morphisms in the $\BB$-category $\Univ[\Gen{p}]$ that is associated with a univalent map $p\colon G\to A$ in $\BB$ admits an explicit description as well:
\begin{proposition}
	\label{prop:objectOfMorphismsUnivalent}
	Let $p\colon P\to A$ be a univalent morphism in $\BB$ and let $\Univ[\Gen{p}]$ be the associated $\BB$-category. Then $(\Univ[\Gen{p}])_1$ is equivalent to the internal mapping object $\iFun{\pr_0^\ast P}{\pr_1^\ast P}$ in $\Over{\BB}{A\times A}$.
\end{proposition}
\begin{proof}
	By construction there is a fully faithful functor $\Univ[\Gen{p}]\hookrightarrow  \Univ$ in $\Cat(\BBB)$, which means that the square
	\begin{equation*}
	\begin{tikzcd}
	(\Univ[\Gen{p}])_1\arrow[d]\arrow[r] & \Univ_1\arrow[d]\\
	(\Univ[\Gen{p}])_0\times(\Univ[\Gen{p}])_0\arrow[r] & \Univ_0\times\Univ_0
	\end{tikzcd}
	\end{equation*}
	is cartesian. On the other hand, Proposition~\ref{prop:mappingObjectsInternalUniverse} identifies the pullback of the above diagram with $\iFun{\pr_0^\ast P}{\pr_1^\ast P}$, which finishes the proof.
\end{proof}

\begin{remark}
	The theory of univalent maps in an $\infty$-topos has been previously worked out by Gepner and Kock in~\cite{gepner2017} and by Rasekh in~\cite{rasekh2018}, using slightly different methods.
\end{remark}

\section{Left fibrations and Yoneda's lemma}
\label{chap:yoneda}
The main goal of this chapter is to formulate and prove Yoneda's lemma for $\BB$-categories. The proof will rely on the interplay between $\Univ$-valued functors on a $\BB$-category $\I{C}$ and \emph{left fibrations} $p\colon \I{P}\to\I{C}$, a result that is commonly referred to as the \emph{Grothendieck construction}. The collection of left fibrations forms the right class of a factorisation system in $\Cat(\BB)$ whose left complement is comprised of \emph{initial functors}. We discuss this factorisation system in \S~\ref{sec:leftFibrations}--\ref{sec:covariantEquivalences}, and in \S~\ref{sec:GrothendieckConstruction} and \S~\ref{sec:universalLeftFibration} we establish the Grothendieck construction and study some of its consequences. Finally, we prove Yoneda's lemma in \S~\ref{sec:YonedaLemma}.

\begin{remark}
	Our strategy for the proof of Yoneda's lemma is inspired by Cisinski's proof of Yoneda's lemma for $\infty$-categories in~\cite{cisinski2019a}.
\end{remark}
\subsection{Left fibrations}
\label{sec:leftFibrations}
In this section we discuss left fibrations between $\BB$-categories (and more generally between simplicial objects) in an $\infty$-topos $\BB$ and discuss some of their basic properties.
\begin{definition}
	A map $P\to C$ between simplicial objects in $\BB$ is a \emph{left fibration} if it is internally right orthogonal to the map $d^1\colon\Delta^0\into\Delta^1$. Dually, $p$ is a \emph{right fibration} if it is internally right orthogonal to the map $d^0\colon\Delta^0\into\Delta^1$. We denote by $\LFib$ and $\RFib$ (or $\LFib_{\BB}$ and $\RFib_{\BB}$ when we want to emphasise the dependency on the base $\infty$-topos) the full subcategories of $\Fun(\Delta^1,\Simp\BB)$ spanned by the left and right fibrations, respectively.
\end{definition}

In what follows, we will mostly restrict the discussion to left fibrations. By dualising, however, all statements carry over unchanged to right fibrations. In more precise terms, this dualisation is obtained by taking \emph{opposite simplicial objects}: Recall that the autoequivalence $(-)^{\op}\colon\CatS\to\CatS$ restricts to an autoequivalence on $\Delta\hookrightarrow \CatS$. By precomposition, one obtains an autoequivalence $(-)^{\op}\colon \BB_{\Delta}\to \BB_{\Delta}$. For any simplicial object $C$ in $\BB$, we refer to the simplicial object $C^{\op}$ as the \emph{opposite simplicial object} associated with $C$. Note that $(-)^{\op}$ restricts to an autoequivalence of $\Cat(\BB)$, hence the opposite of a $\BB$-category is well-defined. Since the functor $(-)^{\op}$ sends the inclusion $d^1\colon\Delta^0\hookrightarrow \Delta^1$ to the map $d^0\colon \Delta^0\hookrightarrow \Delta^1$, one finds that the autoequivalence $(-)^{\op}$ sends right fibrations to left fibrations and vice versa.

\begin{lemma}
	\label{lem:generatorsInitialMaps}
	The saturated class of maps in $\Simp\BB$ that is generated by the maps $d^1\colon E\into\Delta^1\otimes E$  for any simplicial object $E$ in $\BB$ coincides with the saturation of the set
	\begin{equation*}
		\{d^{\{0\}}\colon A\into\Delta^n\otimes A~\vert~A\in\BB,~n\geq 0\}.
	\end{equation*}
\end{lemma}
\begin{proof}
	Let $S$ be the saturation of the set of maps $d^1\colon E\into\Delta^1\otimes E$ for $E\in\Simp\BB$. Then for any $A\in\BB$ and any $n\geq 0$ the map $d^0\colon(\Delta^0\times\Delta^n)\otimes A\into(\Delta^1\otimes\Delta^n)\otimes A$ is contained in $S$ as well. Let $\alpha\colon\Delta^{n+1}\into\Delta^1\times\Delta^n$ be defined by $\alpha(0)=(0,0)$ and $\alpha(k)=(1,k-1)$ for $1\leq k\leq n$, and let	$\beta\colon \Delta^1\times\Delta^n\to\Delta^{n+1}$ be defined by $\beta(0,k)=0$ and $\beta(1,k)=k+1$ for any $0\leq k\leq n$. One then obtains a commutative diagram
	\begin{equation*}
		\begin{tikzcd}
			\Delta^0\arrow[r] \arrow[d, "d^{\{0\}}"]&\Delta^n\arrow[r]\arrow[d, "d^1\times\id"] & \Delta^0\arrow[d, "d^{\{0\}}"]\\
			\Delta^{n+1}\arrow[r, "\alpha"]&\Delta^1\times\Delta^n\arrow[r, "\beta"] &\Delta^{n+1},
		\end{tikzcd}
	\end{equation*}
	and since $\beta\alpha\simeq\id_{\Delta^{n+1}}$ the map $d^{\{0\}}\colon\Delta^0\into\Delta^{n+1}$ is a retract of $d^1\times\id\colon \Delta^n\to \Delta^1\times\Delta^n$. By tensoring with $A$, this shows that the map $d^{\{0\}}\colon A\into\Delta^{n}\otimes A$ is contained in $S$ for all $n\geq 1$.
	
	Conversely, let $S$ be the saturation of the set of maps $d^{\{0\}}\colon A\into\Delta^n\otimes A$ for $n\geq 0$ and $A\in\BB$, and let $E$ be a simplicial object in $\BB$. We need to show that the map $d^1\colon E\into\Delta^1\otimes E$ is contained in $S$. By the same argument as in Lemma~\ref{lem:internalExternalSegalConditions}, we may assume without loss of generality $E\simeq \Delta^n\otimes A$ for some $n\geq 1$ and some $A\in\BB$. Now with respect to the usual decomposition $\Delta^1\times\Delta^n\simeq\Delta^{n+1}\sqcup_{\Delta^n}\cdots\sqcup_{\Delta^n}\Delta^{n+1}$ of the product $\Delta^1\times\Delta^n$ into $n+1$ copies of $\Delta^{n+1}$, the map $d^{\{0\}}\colon\Delta^n\into\Delta^1\times\Delta^n$ is given by the iterated pushout
	\begin{equation*}
		d^{\{0\}}\sqcup_{d^{\{0\}}}\cdots\sqcup_{d^{\{0,\dots,n-1\}}}d^{\{0,\dots,n\}}\colon \Delta^0\sqcup_{\Delta^0}\cdots\sqcup_{\Delta^{n-1}}\Delta^n\to \Delta^{n+1}\sqcup_{\Delta^{n}}\cdots\sqcup_{\Delta^n}\Delta^{n+1}
	\end{equation*}
	in $\Fun(\Delta^1,\Simp\SS)$. It is therefore enough to show that for every integer $n\geq 1$ and every $0\leq i\leq n$ the map $d^{\{0,\dots, i\}}\colon\Delta^i\otimes A\into\Delta^n\otimes A$ is contained in $S$, which follows immediately from the assumption by using item~(2) of Proposition~\ref{prop:propertiesFactorisationSystems}.
\end{proof}
\begin{proposition}
	\label{prop:fibrationsExternalCharacterization}
	A map $P\to C$ between simplicial objects in $\BB$ is a left fibration if and only if for every $n\geq 1$ the commutative diagram
	\begin{equation*}
		\begin{tikzcd}
			P_n \arrow[r] \arrow[d, "d_{\{0\}}"] & C_n\arrow[d, "d_{\{0\}}"]\\
			P_0\arrow[r] & C_0
		\end{tikzcd}
	\end{equation*}
	is cartesian.
\end{proposition}
\begin{proof}
	This follows immediately from Lemma~\ref{lem:generatorsInitialMaps}.
\end{proof}

\begin{lemma}
	\label{lem:leftFibrationsCategorical}
	Let $S$ be the set of maps in $\Simp\BB$ that is internally generated by $d^1\colon \Delta^0\into\Delta^1$. Then $S$ contains the two maps $E^1\to 1$ and $I^2\into\Delta^2$. Dually, the set $S^\prime$ that is internally generated by $d^0\colon \Delta^0\into\Delta^1$ contains the two maps $E^1\to 1$ and $I^2\into\Delta^2$.
\end{lemma}
\begin{proof}
	We show the statement for the set $S$, the proof for the dual case is analogous. Since $d^1$ is a section of the unique map $\Delta^1\to\Delta^0$, the latter is contained in $S$, hence $\Delta^3\to E^1$ is contained in $S$ as well since saturated classes of maps are stable under colimits in $\Fun(\Delta^1,\Simp\BB)$ and pushouts in $\Simp\BB$. Since by Lemma~\ref{lem:generatorsInitialMaps} the map $d^{\{0\}}\colon \Delta^0\into\Delta^3$ defines an element of $S$, we find that the composition $\Delta^0\to E^1$ is contained in $S$. As this is a section of the map $E^1\to 1$, we conclude that the latter map is contained in $S$ as well.
	
	The inclusion $\Delta^1\into I^2$ of the first copy of $\Delta^1$ in $I^2$ is a pushout of $d^1\colon \Delta^0\into\Delta^1$ and therefore an element of $S$. By precomposing with $d^1$, we thus obtain a map $\Delta^0\into I^2$ in $S$ such that its postcomposition with the inclusion $I^2\into\Delta^2$ recovers $d^{\{0\}}\colon\Delta^0\into\Delta^2$. Since Lemma~\ref{lem:generatorsInitialMaps} shows that this map is an element of $S$ as well, we conclude that the inclusion $I^2\into\Delta^2$ must be contained in $S$ too, which finishes the proof.
\end{proof}

\begin{proposition}
	\label{prop:leftFibrationCategorical}
	Let $p\colon P\to \I{C}$ be a left fibration in $\Simp\BB$ such that $\I{C}$ is a $\BB$-category. Then $P$ is a $\BB$-category as well.
\end{proposition}
\begin{proof}
	By Lemma~\ref{lem:leftFibrationsCategorical}, the map $p$ is internally right orthogonal to the two maps $E^1\to 1$ and $I^2\into\Delta^2$. Since $\I{C}$ is internally local with respect to these maps, we conclude that $P$ is internally local with respect to the two maps as well and therefore a $\BB$-category, as claimed.
\end{proof}

\begin{remark}
	\label{rem:initialMapsCatSimp}
	By Proposition~\ref{prop:categoriesExponentialIdeal}, a functor $p\colon \I{P}\to\I{C}$ in $\Cat(\BB)$ is a left fibration precisely if it is internally right orthogonal to the map $d^1$ \emph{in $\Cat(\BB)$}. Therefore, Proposition~\ref{prop:leftFibrationCategorical} implies that the pullback of the cartesian fibration $\LFib\to \Simp\BB$ along the inclusion $\Cat(\BB)\into\Simp\BB$ is given by the full subcategory of $\Fun(\Delta^1,\Cat(\BB))$ that is spanned by the left fibrations between $\BB$-categories. We will denote the resulting cartesian fibration over $\Cat(\BB)$ by $\LFib$ as well. Note, moreover, that also the localisation functor $\Fun(\Delta^1,\Cat(\BB))\to\LFib$ arises as the restriction of the localisation functor $\Fun(\Delta^1,\Simp\BB)\to\LFib$. 
\end{remark}

\begin{remark}
	\label{rem:leftFibsAreLeftOverTotSpace}
	The following observation is due to Sebastian Wolf. By using~\cite[Lemma 1.4.14]{lurie2009b} it is easy to see that a map $\I{P} \rightarrow \I{C}$ in $\Cat(\BBB)$ is a left fibration if and only if the associated functor
	$\smallint \I{P} \rightarrow \smallint \I{C}$ of cocartesian fibrations over $\BB^\op$
	is a left fibration of $\infty$-categories.
	Therefore the unstraightening functor $\int\colon \Cat(\BBB)\into \Cocart(\BB^\op)$ induces an inclusion $\LFib_{\BBB}(-) \into \LFib_{\SS}(\smallint -)$ of presheaves on $\Cat(\BBB)$ that identifies $\LFib_{\BBB}(\I{C})$ with the full subcategory of $\LFib_{\SSS}(\smallint \I{C})$ spanned by those left fibrations $\PP\to \int \I{C}$ for which the induced cocartesian fibration $\PP\to \BB^\op$ is classified by a $\CatSS$-valued sheaf on $\BB$.
\end{remark}

\begin{proposition}
	\label{prop:fibrationsFunctorCategories}
	For any simplicial object $K$ in $\BB$, the endofunctor $\iFun{K}{-}\colon\Simp\BB\to\Simp\BB$ preserves left fibrations.
\end{proposition}
\begin{proof}
	Since left fibrations are by definition \emph{internally} right orthogonal to $d^1\colon\Delta^0\into\Delta^1$, they are also internally right orthogonal to $d^1\colon\Delta^0\otimes K\into\Delta^1\otimes K$, hence the result follows.
\end{proof}

Left fibrations are \emph{fibred in $\BB$-groupoids}. To see this, note that since $d^0\colon \Delta^0\into\Delta^1$ is a section of the unique map $\Delta^1\to\Delta^0$, one finds:
\begin{lemma}
	\label{lem:localisationsInitial}
	Let $S$ be a saturated class of maps in $\Simp\BB$ that contains the maps $d^0\colon K\into\Delta^1\otimes K$ for all $K\in\Simp\BB$. Then $S$ contains the class of maps in $\Simp\BB$ that is internally generated by $\Delta^1\to\Delta^0$.\qed
\end{lemma}
\begin{definition}
	\label{def:conservative}
	A map between simplicial objects in $\BB$ is said to be \emph{conservative} if it is internally right orthogonal to the map $\Delta^1\to\Delta^0$.
\end{definition}
In light of Definition~\ref{def:conservative}, Lemma~\ref{lem:localisationsInitial} shows:
\begin{proposition}
	\label{prop:fibrationsConservative}
	Both left and right fibrations between simplicial objects in $\BB$ are conservative.\qed
\end{proposition}
By Proposition~\ref{prop:characterisationGroupoids}, if $C$ is a simplicial object in $\BB$, the unique map $C\to 1$ in $\Simp\BB$ is conservative if and only if $C$ is a constant simplicial object, i.e.\ contained in the inclusion $\iota\colon \BB\into\Simp\BB$. Consequently, item~(2) of Proposition~\ref{prop:propertiesFactorisationSystems} implies that if $A\in\BB$ is an arbitrary object, a map $C\to A$ in $\Simp\BB$ is conservative if and only if $C$ is contained in $\BB$ as well. In particular, if $f\colon C\to D$ is a conservative map in $\Simp\BB$, then the fibre of $f$ over any map $A\to D$ with $A\in\BB$ is itself contained in $\BB$. 
By using Proposition~\ref{prop:fibrationsConservative} and Corollary~\ref{cor:conservativeFibres}, one now concludes:
\begin{corollary}
	\label{cor:leftFibrationsFibredGroupoids}
	The fibre of a left or right fibration in $\Cat(\BB)$ over any object in the codomain in context $A\in\BB$ is a $\Over{\BB}{A}$-groupoid.\qed
\end{corollary}
\begin{remark}
	\label{rem:groupoidificationFibrantReplacement}
	By Corollary~\ref{cor:leftFibrationsFibredGroupoids} and Proposition~\ref{prop:fibrationsExternalCharacterization}, a map $C\to A$ in $\Simp\BB$ in which $A$ is contained in $\BB$ is a left or right fibration precisely if $C$ is contained in $\BB$ as well. Therefore, both localisation functors $\Fun(\Delta^1, \Simp\BB)\to\RFib$ and $\Fun(\Delta^1,\Simp\BB)\to\LFib$ recover the functor $\colim_{\Delta^{\op}}\colon \Simp\BB\to \BB$ upon taking the fibre over the final object $1\in\Simp\BB$. By restriction, the localisation functors $\Fun(\Delta^1, \Cat(\BB))\to\RFib$ and $\Fun(\Delta^1,\Cat(\BB))\to\LFib$ thus both induce the groupoidification functor on the fibres over $1\in\Cat(\BB)$.
\end{remark}
Conservative maps between \emph{$\BB$-categories} are completely determined by the property that they are fibred in groupoids. To show this, we first need the following lemma:
\begin{lemma}
	\label{lem:generatorsLocalisation}
	Let $S$ be a saturated class of maps in $\Cat(\BB)$ that contains the projections $\Delta^1\otimes A\to A$ for all $A\in\BB$. Then $S$ contains the projection $\Delta^1\otimes\I{C}\to\I{C}$ for any $\BB$-category $\I{C}$.
\end{lemma}
\begin{proof}
	As every $\BB$-category $\I{C}$ is a colimit of $\BB$-categories of the form $\Delta^n\otimes A$ for some $n\geq 0 $ and some $A\in\BB$, we may assume without loss of generality $\I{C}\simeq\Delta^n$. Since $\Delta^n\simeq I^n$ in $\Cat(\BB)$, we may furthermore assume $n=1$. In light of the decomposition $\Delta^1\times\Delta^1\simeq\Delta^2\sqcup_{\Delta^1}\Delta^2$, the projection $\Delta^1\times\Delta^1\to\Delta^1$ is equivalent to the composition
	\begin{equation*}
		\Delta^2\sqcup_{\Delta^1}\Delta^2\xrightarrow{s^1\sqcup_{\Delta^1}\id} \Delta^2\xrightarrow{s^0} \Delta^1.
	\end{equation*}
	It will therefore be enough to show that the two maps $s^0, s^1\colon\Delta^2\rightrightarrows\Delta^1$ are contained in $S$, which follows immediately from the observation that these two maps are given by $s^0\sqcup_{\Delta^0}\id$ and $\id\sqcup_{\Delta^0} s^0$ in light of the decomposition $\Delta^2\simeq\Delta^1\sqcup_{\Delta^0}\Delta^1$.
\end{proof}
Lemma~\ref{lem:generatorsLocalisation} immediately implies:
\begin{proposition}
	\label{prop:conservativeCharacterization}
	A functor $f\colon\I{C}\to\I{D}$ between $\BB$-categories is conservative if and only if the square
	\begin{equation*}
		\begin{tikzcd}
			\I{C}_0\arrow[r, "f_0"]\arrow[d, "s_0"] & \I{D}_0\arrow[d, "s_0"]\\
			\I{C}_1\arrow[r, "f_1"] & \I{D}_1
		\end{tikzcd}
	\end{equation*}
	is cartesian.\qed
\end{proposition}
\begin{corollary}
	\label{cor:conservativeCore}
	A functor $f\colon\I{C}\to\I{D}$ between $\BB$-categories is conservative if and only if the commutative square
	\begin{equation*}
			\begin{tikzcd}
				\I{C}^{\core}\arrow[r, "f^{\core}"]\arrow[d] & \I{D}^{\core}\arrow[d]\\
				\I{C}\arrow[r, "f"] & \I{D}
			\end{tikzcd}
	\end{equation*}
	is cartesian.
\end{corollary}
\begin{proof}
	On account of the Segal conditions, a cartesian square between $\BB$-categories is cartesian if and only if it is cartesian on level $0$ and level $1$, hence the claim follows from the observation that the square in the statement of the corollary is trivially cartesian on level $0$ and recovers the square from Proposition~\ref{prop:conservativeCharacterization} on level $1$.
\end{proof}
\begin{corollary}
	\label{cor:conservativeFibres}
	A functor $\I{C}\to\I{D}$ between large $\BB$-categories is conservative if and only if for any object $d\colon A\to\I{D}$ in context $A\in\BB$ the fibre $\I{C}\vert_d=\I{C}\times_{A}\I{D}$ is a $\Over{\BBB}{A}$-groupoid.
\end{corollary}
\begin{proof}
	If $f$ is conservative, then for any object $d\colon A\to\I{D}$ the map $\I{C}\vert_d\to A$ is conservative as well, hence $\I{C}\vert_d$ is internally local with respect to $\Delta^1\to\Delta^0$ and therefore a $\Over{\BBB}{A}$-groupoid. Conversely, if $\I{C}\vert_d$ is a $\Over{\BBB}{A}$-groupoid for any object $d\colon A\to\I{D}$, the fact that $\I{C}_0\simeq \colim_{A\to \I{C}_0}A$ and descent in $\Simp\BBB$ imply that the fibre of $f$ over the map $\I{C}_0\to\I{C}$ is contained in $\BBB$. On account of Corollary~\ref{cor:conservativeCore}, the claim now follows.
\end{proof}
We conclude this section by showing that equivalences between left or right fibrations can be detected fibrewise:
\begin{proposition}
	\label{prop:equivalenceLeftFibrationsFibrewise}
	A map $f\colon P\to Q$ between left fibrations over a simplicial object $C$ in $\BBB$ is an equivalence if and only if for every object $A\in\BB$ and every map $c\colon A\to C$ the induced map $c^\ast P\to c^\ast Q$ is an equivalence in $\Over{\BBB}{A}$. In particular, a map between left fibrations of large $\BB$-categories is an equivalence if and only if it induces an equivalence on the fibres over every object in the base $\BB$-category.
\end{proposition}
\begin{proof}
	By item~(2) of Proposition~\ref{prop:propertiesFactorisationSystems}, the map $f$ is a left fibration itself. Therefore $f$ is an equivalence whenever the underlying map $f_0\colon P_0\to Q_0$ is one. The claim now follows from descent together with the fact that $C_0$ is canonically obtained as the colimit $\colim_{A\to C_0} A$.
\end{proof}

\subsection{Slice $\BB$-categories}
In this section we will discuss one particularly important example of left fibrations between $\BB$-categories - that of \emph{slice $\BB$-categories}.
\begin{definition}
	\label{def:commaCategory}
	Let $f\colon \I{D}\to \I{C}$ and $g\colon\I{E}\to\I{C}$ be two functors between $\BB$-categories. The \emph{comma $\BB$-category} $\Comma{\I{D}}{\I{C}}{\I{E}}$ is defined as the pullback
	\begin{equation*}
		\begin{tikzcd}
			\Comma{\I{D}}{\I{C}}{\I{E}}\arrow[r]\arrow[d]& \I{C}^{\Delta^1}\arrow[d, "{(d^1,d^0)}"] \\
			\I{D}\times\I{E}\arrow[r, "f\times g"] & \I{C}\times\I{C}.
		\end{tikzcd}
	\end{equation*}
	If $g$ is given by an object $c\colon A\to \I{C}$, we write $\Over{\I{D}}{c}=\Comma{\I{D}}{\I{C}}{A}$, and if in addition $f$ is the identity on $\I{C}$ we refer to this $\BB$-category as the \emph{slice $\BB$-category} over $c$. Dually if $f$ is given by an object $c\colon A\to\I{C}$ we write $\Under{\I{D}}{c}=\Comma{A}{\I{C}}{\I{D}}$ and refer to this $\BB$-category as the slice $\BB$-category under $c$ when furthermore $g$ is the identity on $\I{C}$.
\end{definition}

In the situation of Definition~\ref{def:commaCategory}, the slice $\BB$-category $\Under{\I{C}}{c}$ comes along with a canonical map to $A\times\I{C}$ which we will denote by $(\pi_c)_{!}\colon \Under{\I{C}}{c}\to  A\times\I{C}$.
Furthermore, note that the identity $\id_c\colon A\to \I{C}^{\Delta^1}$ induces a map $A\to \Under{\I{C}}{c}$ over $(\id_A,c)\colon A\to A\times\I{C}$ that we will denote by $\id_c$ as well.
\begin{remark}
	\label{rem:sliceCategoryBaseChange}
	Let $\I{C}$ be a $\BB$-category and let $c\colon A\to \I{C}$ be an object in $\I{C}$, which can be equivalently regarded as an object $c\colon 1\to\pi_A^\ast\I{C}$. Using Proposition~\ref{prop:etaleBaseChangeInternalCategories} and Lemma~\ref{lem:BCFunctorCategory} below, the map $(\pi_c)_!\colon\Under{\I{C}}{c}\to A\times\I{C}$ is equivalent to the image of the projection $(\pi_c)_!\colon\Under{(\pi_A^\ast\I{C})}{c}\to\pi_A^\ast\I{C}$ along the forgetful functor $(\pi_A)_!\colon\Cat(\Over{\BB}{A})\to\Cat(\BB)$.
\end{remark}
\begin{lemma}
	\label{lem:BCFunctorCategory}
	For any object $A\in\BB$ there is a canonical equivalence 
	\begin{equation*}
	\pi_A^\ast\iFun{-}{-}\simeq\iFun{\pi_A^\ast(-)}{\pi_A^\ast(-)}
	\end{equation*}
	of bifunctors $\Cat(\BB)^{\op}\times\Cat(\BB)\to\Cat(\Over{\BB}{A})$.
\end{lemma}
\begin{proof}
	By definition of functor categories in $\Cat(\BB)$ and in $\Cat(\Over{\BB}{A})$, the datum of an equivalence of bifunctors $\pi_A^\ast\iFun{-}{-}\simeq\iFun{\pi_A^\ast(-)}{\pi_A^\ast(-)}$ is equivalent to the datum of an equivalence
	\begin{equation*}
		(\pi_A)_!(-\times_A\pi_A^\ast(-))\simeq (\pi_A)_!(-)\times -
	\end{equation*}
	of bifunctors $\Cat(\Over{\BB}{A})\times\Cat(\BB)\to\Cat(\BB)$. Since $\pi_A^\ast$ commutes with products, the lax square
	\begin{equation*}
	\begin{tikzcd}[column sep={7em,between origins}]
	\BB\times\BB\arrow[d, "-\times -"']\arrow[r, "\pi_A^\ast\times\pi_A^\ast"] & \Over{\BB}{A}\times\Over{\BB}{A}\arrow[d, "-\times_A -"]\\
	\BB\arrow[r, "\pi_A^\ast"']\arrow[ur, Rightarrow, shorten=5mm] & \Over{\BB}{A}
	\end{tikzcd}
	\end{equation*}
	commutes. By making use of the mate construction, this square gives rise to a map
	\begin{equation*}
		\phi\colon (\pi_A)_!(-\times_A-)\to (\pi_A)_!(-)\times \pi_A(-),
	\end{equation*}
	and by combining this map with the adjunction counit $\epsilon\colon (\pi_A)_!\pi_A^\ast\to\id_{\Cat(\BB)}$, we can define a map
	\begin{equation*}
	(\pi_A)_!(-\times_A\pi_A^\ast(-))\xrightarrow{\phi(\id\times\pi_A^\ast)} (\pi_A)_!(-)\times (\pi_A)_!\pi_A^\ast(-)\xrightarrow{\id\times\epsilon} (\pi_A)_!(-)\times -.
	\end{equation*}
	By construction, when evaluated at a pair $(\I{C}\to A, \I{D})$, this map is given by the morphism $\psi$ in the commutative diagram
	\begin{equation*}
	\begin{tikzcd}[column sep={3.5em,between origins}, row sep={1em}]
		& \I{C}\times\I{D}\arrow[rr, "\pr_1"]\arrow[dd, "\pr_0"] && \I{D}\arrow[dd]\\
		\I{C}\times_A(\I{D}\times A)\arrow[rr, "\pr_1", crossing over, near end]\arrow[dd, "\pr_0"]\arrow[ur, "\psi"] && \I{D}\times A\arrow[ur, "\pr_0"']& \\
		& \I{C}\arrow[rr] && 1\\
		\I{C}\arrow[rr]\arrow[ur, "\id"] && A. \arrow[ur, "\pi_A"]\arrow[from=uu, "\pr_1", crossing over, near start]& 
	\end{tikzcd}
	\end{equation*}
	The square on the left side being cartesian now implies that $\psi$ is an equivalence, hence the desired result follows.
\end{proof}

Hereafter, our goal is to prove that the projection $(\pi_c)_!\colon\Under{\I{C}}{c}\to A\times\I{C}$ is a left fibration for any $\BB$-category $\I{C}$ and any object $c$ in $\I{C}$ in context $A\in\BB$. We will achieve this by identifying $(\pi_c)_!$ as the pullback of a left fibration that we will construct hereafter.

Let $-\star -\colon \Delta\times\Delta\to \Delta$ be the ordinal sum bifunctor. We may then define:
\begin{definition}
	\label{def:twistedArrow}
	Let $\epsilon\colon \Delta\to\Delta$ denote the functor $\ord{n}\mapsto \ord{n}^{\op}\star \ord{n}$. For any $\BB$-category $\I{C}$, we define the \emph{twisted arrow $\BB$-category} $\Tw(\I{C})$ to be the simplicial object given by the composition
	\begin{equation*}
			\Delta^{\op}\xrightarrow{\epsilon^{\op}} \Delta^{\op}\xrightarrow{\I{C}} \BB.
	\end{equation*}
	This defines a functor $\Tw\colon \Cat(\BB)\to\Simp\BB$.
\end{definition}
Note that the functor $\epsilon$ in Definition~\ref{def:twistedArrow} comes along with two canonical natural transformations
\begin{equation*}
	(-)^{\op}\to \epsilon \leftarrow \id_{\Delta}
\end{equation*}
which induces a map of simplicial objects
\begin{equation*}
	\Tw(\I{C})\to \I{C}^{\op}\times\I{C}
\end{equation*}
that is natural in $\I{C}$.

\begin{proposition}
	\label{prop:twistedArrowLeftFibration}
	For any $\BB$-category, the simplicial object $\Tw(\I{C})$ is a $\BB$-category as well, and the map $\Tw(\I{C})\to \I{C}^{\op}\times\I{C}$ is a left fibration.
\end{proposition}
\begin{proof}
	We will begin by showing that for any $n\geq 1$ the square
	\begin{equation*}
		\begin{tikzcd}
			\Tw(\I{C})_n\arrow[r]\arrow[d, "d_{\{0\}}"] & \I{C}^{\op}_n\times \I{C}^{\op}_n\arrow[d, "d_{\{0\}}"]\\
			\Tw(\I{C})_0\arrow[r] & \I{C}^{\op}_0\times\I{C}_0
		\end{tikzcd}
	\end{equation*}
	is a pullback diagram. Unwinding the definitions, this is equivalent to the diagram
	\begin{equation*}
		\begin{tikzcd}
			&& \I{C}_{2n+1}\arrow[dll, "d_{\{0,\dots,n\}}"']\arrow[d, "d_{\{n,n+1\}}"]\arrow[drr, "d_{\{n+1,\dots,2n+1\}}"] && \\
			\I{C}_n \arrow[dr, "d_{\{n\}}"']&& \I{C}_1\arrow[dl, "d_{\{0\}}"]\arrow[dr, "d_{\{1\}}"'] && \I{C}_n\arrow[dl, "d_{\{0\}}"] \\
			& \I{C}_0 && \I{C}_0 &
		\end{tikzcd}
	\end{equation*}
	being a limit diagram, which follows easily from the Segal conditions. By Proposition~\ref{prop:fibrationsExternalCharacterization}, the map $\Tw(\I{C})\to \I{C}^{\op}\times\I{C}$ is therefore a left fibration. Since the codomain of this map defines a $\BB$-category, Proposition~\ref{prop:leftFibrationCategorical} now implies that $\Tw(\I{C})$ is a $\BB$-category as well.
\end{proof}
Observe that the ordinal sum functor $\star\colon \Delta\times\Delta\to \Delta$ fits into the commutative square
\begin{equation*}
	\begin{tikzcd}
		\Delta\times\Delta\arrow[r, "\star"] \arrow[d, hookrightarrow] & \Delta\arrow[d, hookrightarrow]\\
		\CatS\times\CatS\arrow[r, "\diamond"] & \CatS
	\end{tikzcd}
\end{equation*}
in which $\diamond$ denotes the bifunctor that sends a pair $(\CC,\DD)$ of $\infty$-categories to the pushout
\begin{equation*}
	\begin{tikzcd}
		(\CC\times\DD)\sqcup(\CC\times\DD)\arrow[d]\arrow[r, "{(d^1,d^0)}"]\arrow[d, "{\pr_0\sqcup \pr_1}"] \arrow[dr, phantom, very near end, "\lrcorner"]& \CC\times\DD\times\Delta^1\arrow[d]\\
		\CC\sqcup\DD\arrow[r] & \CC\diamond\DD.
	\end{tikzcd}
\end{equation*}
In fact, the inclusions $\ord{m}\hookrightarrow \ord{m}\star \ord{n}\hookleftarrow \ord{n}$ in $\Delta$ induce a map $\Delta^m\sqcup \Delta^n\to \Delta^{m\star n}$ of $1$-categories that is natural in $m$ and $n$, and we may also define a map $\Delta^n\times\Delta^m\times\Delta^1\to \Delta^{n\star m}$ of $1$-categories naturally in $m$ and $n$ by sending a triple $(i,j,k)$ to $i$ if $k=0$ and to $m+j$ otherwise. This construction gives rise to a natural map $\Delta^m\diamond \Delta^n\to \Delta^{m\star n}$ that is an equivalence by~\cite[Proposition~4.2.1.2]{htt}. Combining this observation with Proposition~\ref{prop:simplicialPowering}, we therefore conclude that for any $\BB$-category the underlying simplicial object of $\Tw(\I{C})$ is obtained by applying the core functor to the simplicial object $\I{C}^{(\Delta^{\bullet})^{\op}\diamond\Delta^{\bullet}}$ in $\Cat(\BB)$. 
\begin{lemma}
\label{lem:cocommaPushout}
For any integer $n\geq 0$, the canonical square
\begin{equation*}
	\begin{tikzcd}
		(\Delta^n)^{\op}\sqcup \Delta^n\arrow[d] \arrow[r] & (\Delta^n)^{\op}\diamond \Delta^n\arrow[d]\\
		\Delta^0\sqcup \Delta^n \arrow[r] & \Delta^0\diamond\Delta^n
	\end{tikzcd}
\end{equation*}
is a pushout in $\CatS$.
\end{lemma}
\begin{proof}
By definition of the bifunctor $\diamond$ and the pasting lemma for pushout squares, the commutative square in the statement of the lemma is a pushout if and only if the square
\begin{equation*}
	\begin{tikzcd}
		((\Delta^n)^{\op}\times\Delta^n)\sqcup((\Delta^n)^{\op}\times\Delta^n)\arrow[d]\arrow[r, "{(d^1,d^0)}"]\arrow[d]& (\Delta^n)^{\op}\times\Delta^n\times\Delta^1\arrow[d]\\
		\Delta^n\sqcup\Delta^n\arrow[r, "{(d^1,d^0)}"] &\Delta^n\times\Delta^1
	\end{tikzcd}
\end{equation*}
is cocartesian. By making use of the fact that the functor $-\times\Delta^n$ commutes with colimits, we may assume $\Delta^n\simeq \Delta^0$. Moreover, in light of the decomposition $\Delta^n\simeq\Delta^1\sqcup_{\Delta^0}\cdots\sqcup_{\Delta^0}\Delta^1$ in $\CatS$, we may asssume $n=1$. We now have to show that the commutative square
\begin{equation*}
\begin{tikzcd}
(\Delta^1)^{\op}\sqcup(\Delta^1)^{\op}\arrow[r, "{(d^1,d^0)}"] \arrow[d ]& (\Delta^1)^{\op}\times\Delta^1\arrow[d]\\
\Delta^0\sqcup\Delta^0\arrow[r, "{(d^1,d^0)}"] & \Delta^1
\end{tikzcd}
\end{equation*}
is a pushout, which is easily shown by making use of the equivalence $(\Delta^1)^{\op}\times\Delta^1\simeq\Delta^2\sqcup_{\Delta^1}\Delta^2$ and the fact that the diagram
\begin{equation*}
\begin{tikzcd}
\Delta^1\arrow[r, hookrightarrow, "d^0"] \arrow[d]& \Delta^2\arrow[d]\\
\Delta^0\arrow[r, hookrightarrow, "d^0"] & \Delta^1
\end{tikzcd}
\end{equation*}
is cocartesian in $\CatS$.
\end{proof}
By making use of Lemma~\ref{lem:cocommaPushout}, one now obtains a cartesian square
\begin{equation*}
	\begin{tikzcd}
		(\I{C}^{\Delta^{0}\diamond\Delta^{\bullet}})_0 \arrow[d]\arrow[r] & \Tw(\I{C})\arrow[d]\\
		\I{C}_0\times \I{C}\arrow[r] &\I{C}^{\op}\times\I{C}.
	\end{tikzcd}
\end{equation*}
On the other hand, the defining pushout for $\Delta^0\diamond \Delta^{\bullet}$ gives rise to a cartesian square
\begin{equation*}
	\begin{tikzcd}
		(\I{C}^{(\Delta^{0})^{\op}\diamond\Delta^{\bullet}})_0 \arrow[d]\arrow[r] & \I{C}^{\Delta^1}\arrow[d]\\
		\I{C}_0\times \I{C}\arrow[r] &\I{C}\times\I{C},
	\end{tikzcd}
\end{equation*}
which in particular shows:
\begin{proposition}
	\label{prop:fibresTwistedArrowFibration}
	Let $\I{C}$ be a $\BB$-category. For any object $c\colon A\to \I{C}$, the canonical map $(\pi_c)_!\colon\Under{\I{C}}{c}\to A\times\I{C}$ fits into a cartesian square
	\begin{equation*}
		\begin{tikzcd}
			\Under{\I{C}}{c}\arrow[d, ""]\arrow[r] & \Tw(\I{C})\arrow[d]\\
			A\times \I{C}\arrow[r, "c\times\id"] &\I{C}^{\op}\times\I{C}
		\end{tikzcd}
	\end{equation*}
	in $\Cat(\BB)$. In particular, $(\pi_c)_!$ is a left fibration.\qed
\end{proposition}

\subsection{Initial functors}
We will now focus on the left complement of the class of left fibrations. The results in this section are heavily inspired by Cisinski's book~\cite{cisinski2019a}.
\begin{definition}
	A map $J\to I$ between simplicial objects in $\BB$ is said to be \emph{initial} if it is left orthogonal to every left fibration in $\Simp\BB$. Dually, $J\to I$ is \emph{final} if it is left orthogonal to every right fibration in $\Simp\BB$.
\end{definition}
\begin{remark}
	A map  $J\to I$ between simplicial objects in $\BB$ is initial if and only if its opposite $J^{\op}\to I^{\op}$ is final. Therefore all properties of initial maps carry over to final maps upon taking opposite simplicial objects. We will therefore restrict our attention to the case of initial maps.
\end{remark}
\begin{example}
	\label{ex:localisationInitial}
	By Lemma~\ref{lem:localisationsInitial}, any map between simplicial objects in $\BB$ that is in the internal saturation of the map $\Delta^1\to\Delta^0$ defines both an initial and a final map.
\end{example}
\begin{remark}
	A functor $\I{J}\to\I{I}$ between $\BB$-categories is initial (final) in $\Simp\BB$ precisely if it is internally left orthogonal to left (right) fibrations in $\Cat(\BB)$. This is easily seen as a consequence of Proposition~\ref{prop:leftFibrationCategorical}.
\end{remark}

\begin{definition}
	\label{def:initialFinalObject}
	Let $\I{C}$ be a $\BB$-category. An object $c\colon A\to \I{C}$ is said to be \emph{initial} if the transpose map $1\to \pi_A^\ast\I{C}$ defines an initial functor in $\Cat(\Over{\BB}{A})$. Dually, $c$ is \emph{final} if the transpose map $1\to \pi_A^\ast \I{C}$ defines a final functor in $\Cat(\Over{\BB}{A})$.
\end{definition}
\begin{remark}
	In Corollary~\ref{cor:universalPropertyInitialObject} we will see that initial and final objects satisfy the expected universal property, which in particular implies that an object in an ordinary $\infty$-category is initial or final in the sense of Definition~\ref{def:initialFinalObject} precisely if it is initial or final in the usual sense.
\end{remark}

\begin{remark}
	\label{rem:initialObjectequivalentCondition}
	In the context of Definition~\ref{def:initialFinalObject}, the object $c$ is initial if and only if $(c,\id)\colon A\to \I{C}\times A$ is an initial map in $\Cat(\BB)$. To see this, observe that $(c,\id)$ is precisely the image of $c\colon 1\to \pi_A^\ast\I{C}$ under the forgetful functor $(\pi_A)_!$, hence one direction follows from the observation that $(\pi_A)_!$ preserves initial maps. The other direction follows from the fact that every map in $\Cat(\Over{\BB}{A})$ arises as a pullback of a map that is in the image of the base change functor $\pi_A^\ast\colon\Cat(\BB)\to\Cat(\Over{\BB}{A})$, which implies that a map $f$ in $\Cat(\Over{\BB}{A})$ is initial if and only if $(\pi_A)_!(f)$ is initial.
\end{remark}
\begin{remark}
	\label{rem:warningInitialObject}
	Note that an object $c\colon A\to\I{C}$ in a $\BB$-category $\I{C}$ being initial is different from the condition that $c$ is initial when viewed as a functor in $\Cat(\BB)$. In fact, Remark~\ref{rem:initialObjectequivalentCondition} shows that the first condition is equivalent to $(c,\id)\colon A\to A\times\I{C}$ being initial as a functor in $\Cat(\BB)$, hence either of the two conditions implying the other would imply that the projection $A\times\I{C}\to\I{C}$ is initial as well, which is not true in general. In fact, since the projection $A\times\I{C}\to\I{C}$ is a left fibration, this map is initial if and only if it is an equivalence.
\end{remark}
Any functor between $\BB$-categories admits a unique factorisation into an initial map followed by a left fibration. In what follows, our goal is to describe this factorisation explicitly for the case where the domain is an object in $\BB$. To that end, recall that to any $\BB$-category $\I{C}$ and any object $c\colon A\to \I{C}$ one can associate the slice $\BB$-category $\Under{\I{C}}{c}\to A\times\I{C}$ such that the identity map on $c$ defines a section $\id_c\colon A\to \Under{\I{C}}{c}$ over $(\id_A, c)\colon A\to A\times\I{C}$.
\begin{proposition}
	\label{prop:initialityCanonicalSection}
	For any $\BB$-category and any object $c$ in $\I{C}$ in context $A\in \BB$, the section $\id_c\colon A\to \Under{\I{C}}{c}$ is initial as a map in $\Cat(\BB)$.
\end{proposition}
\begin{remark}
	In the situation of Proposition~\ref{prop:initialityCanonicalSection}, let us denote by $c\colon 1\to \pi_A^\ast\I{C}$ the global section in $\Cat(\Over{\BB}{A})$ that arises as the transpose of $c\colon A\to \I{C}$. Then the section $\id_c\colon A\to \Under{\I{C}}{c}$ arises as the image of the global section $\id_c\colon 1\to \Under{({\pi_A^\ast\I{C}})}{c}$  under the forgetful functor $(\pi_A)_!\colon\Cat(\Over{\BB}{A})\to\Cat(\BB)$. Since this functor creates initial maps (cf.\ Remark~\ref{rem:initialObjectequivalentCondition}), we conclude that the section $\id_c\colon A\to \Under{{\I{C}}}{c}$ is initial if and only if the associated global section $\id_c\colon 1\to \Under{({\pi_A^\ast\I{C}})}{c}$ defines an initial object in the $\Over{\BB}{A}$-category $\Under{({\pi_A^\ast\I{C}})}{c}$.
\end{remark}
The proof of Proposition~\ref{prop:initialityCanonicalSection} requires a few preparations. We begin by describing how the comma $\BB$-category construction from Definition~\ref{def:commaCategory} can be turned into a functor, and we will construct a left adjoint to this functor. To that end, let $K$ denote the $1$-category that is depicted in the following diagram:
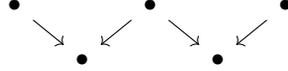
\begin{equation*}
	\begin{tikzcd}[column sep=small, row sep=small]
	\bullet \arrow[dr]&& \bullet\arrow[dl]\arrow[dr] &&\bullet\arrow[dl]\\
	& \bullet && \bullet &
	\end{tikzcd}
\end{equation*}
Let $i\colon K\into K^{\triangleleft}$ be the canonical inclusion, and let $r\colon \Lambda^2_0\to K^{\triangleleft}$ be the span that is given by the two dashed arrows in the following depiction of $K^{\triangleleft}$: 
\begin{equation*}
	\begin{tikzcd}[column sep=small, row sep=small]
	&& \bullet\arrow[dll, dashed]\arrow[d]\arrow[drr, dashed] && \\
	\bullet \arrow[dr]&& \bullet\arrow[dl]\arrow[dr] &&\bullet\arrow[dl]\\
	& \bullet && \bullet &
	\end{tikzcd}
\end{equation*}
Let furthermore $\phi\colon\Fun(\Lambda^2_2,\Cat(\BB))\to\Fun(K, \Cat(\BB))$ be the evident functor that sends a cospan
\begin{equation*}
	\begin{tikzcd}[column sep=small, row sep=small]
		& \I{E}\arrow[d, "g"]\\
		\I{D}\arrow[r, "f"] & \I{C}
	\end{tikzcd}
\end{equation*}
to the diagram
\begin{equation*}
	\begin{tikzcd}[column sep=small, row sep=small]
		& & \I{E}\arrow[d, "g"]\\
		& \I{C}^{\Delta^1}\arrow[r, "d_0"] \arrow[d, "d_1"]& \I{C}\\
		\I{D}\arrow[r, "f"] & \I{C}. &
	\end{tikzcd}
\end{equation*}
The composition
\begin{equation*}
		\Fun(\Lambda^2_2, \Cat(\BB))\xrightarrow{\phi} \Fun(K, \Cat(\BB))\xrightarrow{i_\ast} \Fun(K^{\triangleleft}, \Cat(\BB))\xrightarrow{r^\ast} \Fun(\Lambda^2_0, \Cat(\BB)) 
\end{equation*}
in which $i_\ast$ denotes the functor of right Kan extension along $i$ then sends a cospan
\begin{equation*}
	\begin{tikzcd}[column sep=small, row sep=small]
		& \I{E}\arrow[d, "g"]\\
		\I{D}\arrow[r, "f"] & \I{C}
	\end{tikzcd}
\end{equation*}
to the span $\I{D}\leftarrow\Comma{D}{C}{E}\rightarrow\I{E}$ and therefore defines the desired comma $\BB$-category functor $\Comma{-}{}{-}$.
\begin{definition}
	\label{def:cocomma}
	Let $f\colon \I{C}\to\I{D}$ and $g\colon\I{C}\to\I{E}$ be functors in $\Cat(\BB)$. The \emph{cocomma $\BB$-category} $\Cocomma{\I{D}}{\I{C}}{\I{E}}$ is the $\BB$-category that is defined by the pushout square
	\begin{equation*}
		\begin{tikzcd}
			\I{C}\sqcup\I{C}\arrow[r, "{(d^1,d^0)}"]\arrow[d, "f\sqcup g"] & \Delta^1\otimes\I{C}\arrow[d]\\
			\I{D}\sqcup\I{E}\arrow[r] &\Cocomma{\I{D}}{\I{C}}{\I{E}}.
		\end{tikzcd}
	\end{equation*}
	If $A\in \BB$ is an arbitrary object and $\I{C}\to A$ is a map (i.e.\ if $\I{C}$ is a $\Over{\BB}{A}$-category), we write $\I{C}^{\triangleright}=\Cocomma{\I{C}}{\I{C}}{A}$ and refer to this $\BB$-category as the \emph{right cone of $\I{C}\to A$}. Dually, we write $\I{C}^{\triangleleft}=\Cocomma{A}{\I{C}}{\I{C}}$ and refer to this $\BB$-category as the \emph{left cone of $\I{C}\to A$}.
\end{definition}
Analogous to the comma $\BB$-category construction, the construction of the cocomma $\BB$-category from Definition~\ref{def:cocomma} can be turned into a functor $\Cocomma{-}{}{-}$ as follows: Let $L=K^{\op}$, let $j\colon L\into L^{\triangleright}$ denote the canonical inclusion, and let $s=r^{\op}\colon \Lambda^2_2\to L^{\triangleright}$.
If $\psi\colon \Fun(\Lambda^2_0,\Cat(\BB))\to \Fun(L,\Cat(\BB))$ denotes the functor that sends a span 
\begin{equation*}
	\begin{tikzcd}[column sep=small, row sep=small]
		\I{C}\arrow[r, "g"]\arrow[d, "f"] & \I{E}\\
		\I{D}
	\end{tikzcd}
\end{equation*}
to the diagram
\begin{equation*}
	\begin{tikzcd}[column sep=small, row sep=small]
		& \I{C}\arrow[r,"g"] \arrow[d, "d^0"]& \I{E}\\
		\I{C}\arrow[d, "f"] \arrow[r, "d^1"]& \Delta^1\otimes \I{C} &\\
		\I{D}, & &  
	\end{tikzcd}
\end{equation*}
the desired functor $\Cocomma{-}{}{-}$ can be defined by the composition
\begin{equation*}
		\Fun(\Lambda^2_0, \Cat(\BB))\xrightarrow{\psi} \Fun(L, \Cat(\BB))\xrightarrow{j_!} \Fun(L^{\triangleright}, \Cat(\BB))\xrightarrow{s^\ast} \Fun(\Lambda^2_2, \Cat(\BB)) 
\end{equation*}
in which $j_!$ denotes the functor of left Kan extension along $j$.
\begin{proposition}
	\label{prop:adjunctionCommaCocomma}
	The functor $\Cocomma{-}{}{-}$ is left adjoint to the functor $\Comma{-}{}{-}$.
\end{proposition}
The proof of Proposition~\ref{prop:adjunctionCommaCocomma} makes use of the following description of the mapping $\infty$-groupoids in functor $\infty$-categories, due to Glasman:
\begin{lemma}[{\cite[Proposition~2.3]{glasman2016}}]
	\label{lem:mappingGroupoidFunctorCategory}
	For any two locally small $\infty$-categories $\CC$ and $\DD$, the mapping $\infty$-groupoid functor $\map{\Fun(\CC,\DD)}$ is equivalent to the composition
	\begin{equation*}
			\Fun(\CC,\DD)^{\op}\times\Fun(\CC,\DD)\to \Fun(\CC^{\op}\times \CC, \DD^{\op}\times\DD)\xrightarrow{p^{\ast}(\map{\DD})_\ast} \Fun(\Tw(\CC), \SS)\xrightarrow{\lim} \SS,
	\end{equation*}
	in which $p\colon \Tw(\CC)\to\CC^{\op}\times\CC$ denotes the canonical projection.\qed
\end{lemma}
\begin{remark}
	\label{rem:mappingGroupoidFunctorCategory}
	Quite generally, if $\II$ and $\JJ$ are arbitrary $\infty$-categories 
	and if $d\colon\II\to \Fun(\JJ,\SS)$ is any functor, the composition $\lim_{\JJ} d\colon \II\to\Fun(\JJ, \SS)\to \SS$
	is the limit of the diagram $\JJ\to\Fun(\II,\SS)$
	that corresponds to $d$ in light of the equivalence $\Fun(\JJ,\Fun(\II,\SS))\simeq\Fun(\II,\Fun(\JJ,\SS))$.
	In the situation of Lemma~\ref{lem:mappingGroupoidFunctorCategory}, this observation implies that the mapping $\infty$-groupoid functor $\map{\Fun(\CC,\DD)}$ is equivalently obtained as the limit of the diagram
	\begin{equation*}
	\Tw(\CC)\to  \Fun(\Fun(\CC,\DD)^{\op}\times\Fun(\CC,\DD), \SS)
	\end{equation*}
	that is determined by transposing the composition
	\begin{equation*}
			\Fun(\CC,\DD)^{\op}\times\Fun(\CC,\DD)\to \Fun(\CC^{\op}\times \CC, \DD^{\op}\times\DD)\xrightarrow{p^{\ast}(\map{\DD})_\ast} \Fun(\Tw(\CC), \SS).
	\end{equation*}
\end{remark}
\begin{proof}[{Proof of Proposition~\ref{prop:adjunctionCommaCocomma}}]
	Note that there is an evident identification $K^\triangleleft\simeq \Delta^1\times\Lambda^2_0$ with respect to which the functor $r$ corresponds to $d^1\times \id_{\Lambda^2_0}$. Hence, by making use of the two adjunctions $d^1\dashv s^0\colon \Delta^1\leftrightarrows \Delta^0$ and $s^0\dashv d^0\colon \Delta^0\leftrightarrows\Delta^1$, the functor $r^\ast$ admits a left adjoint $r_!$ that is given by precomposition with $s^0\times\id$, and the functor $s^\ast$ admits a right adjoint $s_\ast$ that is given by precomposition with $s^0\times\id$. We therefore obtain equivalences
	\begin{equation*}
		\map{\Fun(\Lambda^2_0,\Cat(\BB))}(-, \Comma{-}{}{-})\simeq\map{\Fun(K,\Cat(\BB))}(i^\ast r_!(-), \phi(-))
	\end{equation*}
	and
	\begin{equation*}
		\map{\Fun(\Lambda^2_2,\Cat(\BB))}(\Cocomma{-}{}{-},-)\simeq\map{\Fun(L,\Cat(\BB))}(\psi(-), j^\ast s_\ast(-)).
	\end{equation*}
	We complete the proof by constructing an equivalence
	\begin{equation*}
		\map{\Fun(K,\Cat(\BB))}(i^\ast r_!(-), \phi(-))\simeq\map{\Fun(L,\Cat(\BB))}(\psi(-), j^\ast s_\ast(-)).
	\end{equation*}
	To that end, observe that the twisted arrow $1$-categories $\Tw(K)$ and $\Tw(L)$ can both be identified with the poset that can be depicted by the diagram
	
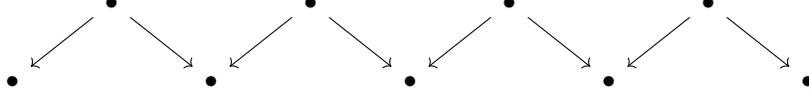
\begin{equation*}
		\begin{tikzcd}
			& \bullet\arrow[dl]\arrow[dr] && \bullet\arrow[dl]\arrow[dr] && \bullet\arrow[dl]\arrow[dr] && \bullet \arrow[dl]\arrow[dr]&\\
			\bullet && \bullet && \bullet && \bullet && \bullet
		\end{tikzcd}
	\end{equation*}
	in which the vertices in the lower row correspond to the degenerate edges in $K$ and $L$ and the vertices in the upper row correspond to the non-degenerate edges in $K$ and $L$, and in which the edges correspond to the assignment of each non-degenerate edge in $K$ and $L$ to its domain and its codomain, respectively.
	Let us label the vertices and edges of the two $1$-categories $\Lambda^2_2$ and $\Lambda^2_0$ as depicted in the following two diagrams
	\begin{equation*}
		\begin{tikzcd}
			c\arrow[d, "p_0"]\arrow[r, "p_1"] & b && & y\arrow[d, "q_1"]\\
			a & && x\arrow[r, "q_0"] & z.
		\end{tikzcd}
	\end{equation*}
	 By using Lemma~\ref{lem:mappingGroupoidFunctorCategory} and Remark~\ref{rem:mappingGroupoidFunctorCategory}, the bifunctor $\map{\Fun(K,\Cat(\BB))}(i^\ast r_!(-), \phi(-))$ is obtained as the limit of the diagram
	\begin{equation*}
	\begin{tikzcd}[column sep={8em,between origins}]
		\map{\Cat(\BB)}(a^\ast(-),x^\ast(-)) \arrow[dr, "(q_0)_\ast p_0^\ast"]&& \map{\Cat(\BB)}(c^\ast(-),z^\ast(-)^{\Delta^1}) \arrow[dl, "(d_1)_\ast"']\arrow[dr, "(d_0)_\ast"] && \map{\Cat(\BB)}(b^\ast(-),y^\ast(-))\arrow[dl, "(q_1)_\ast p_1^\ast"']\\
		& \map{\Cat(\BB)}(c^\ast(-),z^\ast(-)) && \map{\Cat(\BB)}(c^\ast(-),z^\ast(-)) &
	\end{tikzcd}
	\end{equation*}
	in which for $i\in\{0,1\}$ the maps $p_i$ and $q_i$ denote the morphisms of functors that are induced by the inclusion of the two edges $p_i\colon\Delta^1\into\Lambda^2_0$ and $q_i\colon\Delta^1\into\Lambda^2_2$. Dually, the bifunctor $\map{\Fun(L,\Cat(\BB))}(\psi(-), j^\ast s_\ast(-))$ is obtained as the limit of the diagram
	\begin{equation*}
	\begin{tikzcd}[column sep={8em,between origins}]
		\map{\Cat(\BB)}(a^\ast(-),x^\ast(-)) \arrow[dr, "(q_0)_\ast p_0^\ast"]&& \map{\Cat(\BB)}(\Delta^1\otimes c^\ast(-),z^\ast(-)) \arrow[dl, "(d^1)^\ast"']\arrow[dr, "(d^0)^\ast"] && \map{\Cat(\BB)}(b^\ast(-),y^\ast(-))\arrow[dl, "(q_1)_\ast p_1^\ast"']\\
		& \map{\Cat(\BB)}(c^\ast(-),z^\ast(-)) && \map{\Cat(\BB)}(c^\ast(-),z^\ast(-)), &
	\end{tikzcd}
	\end{equation*}
	hence the adjunction $\Delta^1\otimes -\dashv (-)^{\Delta^1}$ gives rise to the desired equivalence.
\end{proof}
\begin{remark}
	\label{rem:commaCocommaRelativeAdjunction}
	In the situation of Proposition~\ref{prop:adjunctionCommaCocomma}, let $\alpha\colon\Delta^0\sqcup\Delta^0\into\Lambda^2_0$ be the inclusion of the two objects that are not initial, and let $\beta\colon\Delta^0\sqcup\Delta^0\into\Lambda^2_2$ be the inclusion of the two objects that are not final. Then one obtains two restriction functors $\alpha^\ast\colon\Fun(\Lambda^2_0,\Cat(\BB))\to\Cat(\BB)\times\Cat(\BB)$ and $\beta^\ast\colon\Fun(\Lambda^2_2,\Cat(\BB))\to\Cat(\BB)\times\Cat(\BB)$ that commute with the comma and cocomma constructions, in that there are two commutative diagrams
	\begin{equation*}
		\begin{tikzcd}[column sep={5.5em,between origins}]
			\Fun(\Lambda^2_2, \Cat(\BB))\arrow[rr, "\Comma{-}{}{-}"]\arrow[dr, "\beta^\ast"'] && \Fun(\Lambda^2_0,\Cat(\BB))\arrow[dl, "\alpha^\ast"] &&& \Fun(\Lambda^2_2, \Cat(\BB))\arrow[from=rr, "\Cocomma{-}{}{-}"']\arrow[dr,"\beta^\ast"'] && \Fun(\Lambda^2_0,\Cat(\BB))\arrow[dl, "\alpha^\ast"]\\
			& \Cat(\BB)\times\Cat(\BB) & &&& & \Cat(\BB)\times\Cat(\BB).
		\end{tikzcd}
	\end{equation*}
	Moreover, the explicit construction of the equivalence
	\begin{equation*}
		\begin{tikzcd}
			\map{\Fun(\Lambda^2_2,\Cat(\BB))}(\Cocomma{-}{}{-},-)\simeq\map{\Fun(\Lambda^2_0,\Cat(\BB))}(-,\Comma{-}{}{-})
		\end{tikzcd}
	\end{equation*}
	in the proof of Proposition~\ref{prop:adjunctionCommaCocomma} shows that if $\eta$ and $\epsilon$ denote unit and counit of the adjunction $\Cocomma{-}{}{-}\dashv \Comma{-}{}{-}$, the induced morphisms $\alpha^\ast\eta$ and $\beta^\ast \epsilon$ recover the identity transformations on $\alpha^\ast$ and $\beta^\ast$, respectively.
\end{remark}
\begin{remark}
	\label{rem:commaCocommaExplicit}
	In the situation of Proposition~\ref{prop:adjunctionCommaCocomma}, if $\I{C}$ is a $\BB$-category, note that the cospan $\Cocomma{\I{C}}{\I{C}}{\I{C}}$ in $\Cat(\BB)$ is given by  $\I{C}\into\Delta^1\otimes\I{C}\hookleftarrow \I{C}$ in which the two inclusions are the cosimplicial maps $d^1$ and $d^0$. The explicit construction of the equivalence
	\begin{equation*}
	\begin{tikzcd}
		\map{\Fun(\Lambda^2_2,\Cat(\BB))}(\Cocomma{-}{}{-},-)\simeq\map{\Fun(\Lambda^2_0,\Cat(\BB))}(-,\Comma{-}{}{-})
	\end{tikzcd}
	\end{equation*}
	in the proof of Proposition~\ref{prop:adjunctionCommaCocomma} then shows that the transpose of the map $(\id, s_0,\id)\colon(\I{C}=\I{C}=\I{C})\to \Comma{\I{C}}{\I{C}}{\I{C}}$ is given by the map $(\id,s^0,\id)\colon \Cocomma{\I{C}}{\I{C}}{\I{C}}\to(\I{C}=\I{C}=\I{C})$.
\end{remark}
\begin{remark}
	\label{rem:canonicalSectionCocomma}
	Let $\I{C}$ be a $\BB$-category and let $c\colon A\to \I{C}$ be an object in $\I{C}$. Then the canonical section $\id_c\colon A\to \Under{\I{C}}{c}$ can be viewed as a map $(A=A=A)\to \Comma{A}{\I{C}}{\I{C}}$ in $\Fun(\Lambda^2_0,\Cat(\BB))$ that fits into a commutative square
	\begin{equation*}
		\begin{tikzcd}[column sep=large]
			{(A=A=A)}\arrow[r]\arrow[d, "{(c,c,c)}"] & \Comma{A}{\I{C}}{\I{C}}\arrow[d, "{\Comma{c}{\id}{\id}}"]\\
			{(C=C=C)}\arrow[r, "{(\id, s_0,\id)}"] & \Comma{\I{C}}{\I{C}}{\I{C}}.
		\end{tikzcd}
	\end{equation*}
	Using Proposition~\ref{prop:adjunctionCommaCocomma} and Remark~\ref{rem:commaCocommaExplicit}, this diagram corresponds to a commutative square
	\begin{equation*}
		\begin{tikzcd}[column sep=large]
			\Cocomma{A}{A}{A}\arrow[r]\arrow[d, "{\Cocomma{c}{c}{c}}"] & (A\to\I{C}=\I{C} )\arrow[d, "{(c,\id,\id)}"]\\
			\Cocomma{\I{C}}{\I{C}}{\I{C}}\arrow[r, "{(\id, s^0,\id)}"] & (\I{C}=\I{C}=\I{C}),
		\end{tikzcd}
	\end{equation*}
	hence the map $\Cocomma{A}{A}{A}\to (A\to \I{C}=\I{C})$ that corresponds to $\id_c\colon (A=A=A)\to\Comma{A}{\I{C}}{\I{C}}$ is given by the triple $(\id, \id_c, c)$.
\end{remark}
\begin{lemma}
	\label{lem:criterioninitiality}
	For any $\BB$-category $\I{C}$, an object $c\colon 1\to\I{C}$ is initial if the canonical projection $(\I{\pi}_c)_!\colon \Under{\I{C}}{c}\to \I{C}$ has a section $s\colon\I{C}\to \Under{\I{C}}{c}$ such that $s(c)$ is equivalent to the canonical section $\id_c$.
\end{lemma}
\begin{proof}
	The section $s$ defines a commutative diagram
	\begin{equation*}
		\begin{tikzcd}[column sep=large]
			(1=1=1)\arrow[d, "{(c,c,\id)}"'] \arrow[dr, bend left, "{(\id, \id_c,\id)}"]& \\
			(1\leftarrow\I{C}=\I{C})\arrow[r, "{(\id,s,\id)}"] & \Comma{A}{\I{C}}{\I{C}}
		\end{tikzcd}
	\end{equation*}
	in $\Fun(\Lambda^2_2,\Cat(\BB))$. Using the adjunction between the comma and the cocomma construction as well as Remark~\ref{rem:canonicalSectionCocomma}, this diagram corresponds to a commutative diagram
	\begin{equation*}
		\begin{tikzcd}[column sep=large]
			\Cocomma{1}{1}{1}\arrow[d]\arrow[dr, bend left, "{(\id,\id_c,c)}"] &\\
			\Cocomma{1}{\I{C}}{\I{C}}\arrow[r, "{(\id,r, \id)}"] & (1\to\I{C}=\I{C})
		\end{tikzcd}
	\end{equation*}
	in $\Fun(\Lambda^2_0,\Cat(\BB))$, which is explicitly given by the commutative diagram
	\begin{equation*}
		\begin{tikzcd}
			1\arrow[d, hookrightarrow, "d^1"]\arrow[r, "\id"] & 1 \arrow[r, "\id"] & 1\arrow[d, "c"]\\
			\Delta^1\arrow[r]\arrow[rr, bend left, "\id_c", near start]& \I{C}^{\triangleleft}\arrow[r, "r"] \arrow[from=u, crossing over]& \I{C} \\
			1\arrow[u, "d^0"'] \arrow[r, "c"]& \I{C}\arrow[u] \arrow[r, "\id"] & \I{C}\arrow[u, "\id"'] 
		\end{tikzcd}
	\end{equation*}
	in $\Cat(\BB)$. As a consequence, the map $c\colon 1\to\I{C}$ is seen to be a retract of the map $\Delta^1\to\I{C}^{\triangleleft}$, hence it suffices to show that the latter is initial. Since both $d^1$ and the map $1\to\I{C}^{\triangleleft}$ are initial, this follows immediately from item~(2) in Proposition~\ref{prop:propertiesFactorisationSystems}.
\end{proof}
\begin{proof}[{Proof of Proposition~\ref{prop:initialityCanonicalSection}}]
	 The canonical section $\id_c\colon A\to \Under{\I{C}}{c}$ is obtained as the image of the canonical section $\id_c\colon 1\to \Under{(\pi_A^\ast\I{C})}{c}$ in $\Cat(\Over{\BB}{A})$ under the forgetful functor $(\pi_A)_!\colon\Cat(\Over{\BB}{A})\to\Cat(\BB)$ when viewing $c$ as a an object of $\pi_A^\ast\I{C}$ (cf.\ Remark~\ref{rem:sliceCategoryBaseChange}). As the functor $(\pi_A)_!$ preserves initial maps, we may replace $\BB$ with $\Over{\BB}{A}$ and can therefore assume without loss of generality $A\simeq 1$. We would like to apply Lemma~\ref{lem:criterioninitiality} to the pair $(\Under{\I{C}}{c},\id_c)$. By using the adjunction between the comma and the cocomma construction, it suffices to construct a map $r\colon (\Under{\I{C}}{c})^{\triangleleft}\to\Under{\I{C}}{c}$ such that the triple $(\id, r, \id)$ defines a map $\Cocomma{1}{\Under{\I{C}}{c}}{\Under{\I{C}}{c}}\to (1\to \Under{\I{C}}{c}=\Under{\I{C}}{c})$ in $\Fun(\Lambda^2_2,\Cat(\BB))$. To that end, let $d\colon \Delta^1\times\Delta^1\to\Delta^1$ be the projection onto the diagonal that is given by composing the equivalence $\Delta^1\times\Delta^1\simeq\Delta^2\sqcup_{\Delta^1}\Delta^2$ with $s^0\sqcup_{\Delta^1}s^0\colon\Delta^2\sqcup_{\Delta^1}\Delta^2\to \Delta^1$. Then the map $d^\ast\colon \I{C}^{\Delta^1}\to\I{C}^{\Delta^1\times\Delta^1}$ fits into the two commutative squares
	 \begin{equation*}
	 	\begin{tikzcd}
	 		\I{C}^{\Delta^1}\arrow[r, "d^\ast"] \arrow[d, "d_1"] & \I{C}^{\Delta^1\times\Delta^1}\arrow[d, "(\id\times d^1)^\ast"] && \I{C}^{\Delta^1}\arrow[r, "d^\ast"] \arrow[d, "d_1"] & \I{C}^{\Delta^1\times\Delta^1}\arrow[d, "(d^1\times \id)^\ast"]\\
	 		\I{C}\arrow[r, "s_0"] & \I{C}^{\Delta^1}  && \I{C}\arrow[r, "s_0"] & \I{C}^{\Delta^1}.
	 	\end{tikzcd}
	 \end{equation*}
 	Transposing $d^\ast$ along the adjunction $\Delta^1\otimes -\dashv (-)^{\Delta^1}$ thus determines a map $e\colon \Delta^1\otimes\I{C}^{\Delta^1}\to\I{C}^{\Delta^1}$ together with two commutative squares
 	\begin{equation*}
 		\begin{tikzcd}
 			\Delta^1\otimes\I{C}^{\Delta^1}\arrow[r, "e"]\arrow[d, "\id\otimes d_1"] & \I{C}^{\Delta^1}\arrow[d, "d_1"] && \I{C}^{\Delta^1}\arrow[r, "d^1\otimes\id"]\arrow[d, "d_1"] & \Delta^1\otimes \I{C}^{\Delta^1}\arrow[d, "e"]\\
 			\Delta^1\otimes\I{C}\arrow[r, "s^0\otimes\id"] & \I{C} && \I{C}\arrow[r, "s_0"] & \I{C}^{\Delta^1}.
 		\end{tikzcd}
 	\end{equation*}
 	By pasting the right square with the pullback diagram that defines the slice $\BB$-category $\Under{\I{C}}{c}$, we obtain a map $h\colon (\Under{\I{C}}{c})^{\triangleleft}\to\I{C}^{\Delta^1}$ that fits into the commutative diagram
 	\begin{equation*}
 		\begin{tikzcd}[column sep={4em,between origins}, row sep={3em,between origins}]
 			& \Under{\I{C}}{c}\arrow[dd]\arrow[rr, "d^1\otimes\id"]\arrow[dl] && \Delta^1\otimes\Under{\I{C}}{c}\arrow[dd]\arrow[dl]\arrow[rr, "\id\otimes \pi_{\Under{\I{C}}{c}}"] && \Delta^1\arrow[dd]\arrow[dl, "\id\otimes c"{fill=white}]\\
 			\I{C}^{\Delta^1}\arrow[rr, crossing over, "d^1\otimes\id"{fill=white}]\arrow[dd, "d_1"] && \Delta^1\otimes\I{C}^{\Delta^1} \arrow[rr, crossing over, "\id\otimes d_1"{fill=white}]&& \Delta^1\otimes\I{C} &\\
 			 & 1\arrow[rr] \arrow[dl, "c"]&& (\Under{\I{C}}{c})^{\triangleleft}\arrow[dl, "h"]\arrow[rr] && 1\arrow[dl, "c"]\\
 			 \I{C}\arrow[rr, "s_0"] && \I{C}^{\Delta^1}. \arrow[from=uu, crossing over,"e", near start]\arrow[rr, "d_1"]&& \I{C}. \arrow[from=uu, "s^0\otimes\id", near start]&
 		\end{tikzcd}
 	\end{equation*}
 	The horizontal square on the bottom right in this diagram now gives rise to the desired map $r\colon (\Under{\I{C}}{c})^{\triangleleft}\to\Under{\I{C}}{c}$. By inspection of the above commutative diagram, it is clear that $r$ sends the section $1\to (\Under{\I{C}}{c})^{\triangleleft}$ to the canonical section $\id_c\colon 1\to\Under{\I{C}}{c}$. Lastly, the observation that the composition
 	\begin{equation*}
 			\I{C}^{\Delta^1}\xrightarrow{d^0\otimes\id} \Delta^1\otimes\I{C}^{\Delta^1}\xrightarrow{e} \I{C}^{\Delta^1}
 	\end{equation*}
 	recovers the identity on $\I{C}^{\Delta^1}$ implies that $r$ is a retract of the map $\Under{\I{C}}{c}\to (\Under{\I{C}}{c})^{\triangleleft}$, which finishes the proof.
\end{proof}
\begin{corollary}
	\label{cor:factorisationInternalObject}
	Let $\I{C}$ be a $\BB$-category and let $c\colon A\to\I{C}$ be an object in $\I{C}$. The factorisation of $c$ into an initial map and a left fibration is given by the composition $\pr_1(\pi_c)_!\id_c\colon A\to \Under{(\I{C})}{c}\to \I{C}$ where $\pr_1\colon A\times\I{C}\to\I{C}$ is the projection.\qed
\end{corollary}
Lemma~\ref{lem:criterioninitiality} can furthermore be used to derive the following characterisation of initial objects in a $\BB$-category:
\begin{proposition}
	\label{prop:characterisationinitialObject}
	Let $\I{C}$ be a $\BB$-category. For any object $c\colon A\to \I{C}$, the following are equivalent:
	\begin{enumerate}
		\item  $c$ is an initial object;
		\item the projection $(\pi_c)_!\colon\Under{\I{C}}{c}\to A\times \I{C}$ is an equivalence;
		\item for any object $d\colon B\to \I{C}$ the map $\map{\I{C}}(\pr_0^\ast c,\pr_1^\ast d)\to A\times B$ is an equivalence in $\BB$.
	\end{enumerate}
\end{proposition}
\begin{proof}
	By replacing $\BB$ with $\Over{\BB}{A}$, we can assume that $A\simeq 1$. Now if $c$ is initial, Corollary~\ref{cor:factorisationInternalObject} implies that the left fibration $(\pi_c)_!\colon\Under{\I{C}}{c}\to \I{C}$ must be initial as well and therefore an equivalence. Conversely, if this map is an equivalence, Corollary~\ref{cor:factorisationInternalObject} implies that $c$ is initial. Lastly, since the map $(\pi_c)_!\colon\Under{\I{C}}{c}\to \I{C}$ is a left fibration, Proposition~\ref{prop:equivalenceLeftFibrationsFibrewise} implies that this map is an equivalence whenever the induced map $(\Under{\I{C}}{c})\vert_{d}\to B$ is an equivalence for any object $d\colon  B\to\I{C}$. As this map recovers the morphism in~(3), the claim follows.
\end{proof}
\begin{corollary}
	\label{cor:universalPropertyInitialObject}
	Let $\I{C}$ be a $\BB$-category and let $c$ and $d$ be objects in $\I{C}$ in context $A\in\BB$ such that $c$ is initial. Then there is a unique map $c\to d$ in $\I{C}$ in context $A$ that is an equivalence if and only if $d$ is initial as well.
\end{corollary}
\begin{proof}
	We may again assume $A\simeq 1$. By Proposition~\ref{prop:characterisationinitialObject}, the map $\map{\I{C}}(c,d)\to 1$ is an equivalence. Therefore, there is a unique map $f\colon c\to d$ that corresponds to the unique section $1\to \map{\I{C}}(c,d)$. If $d$ is initial, then by the same argumentation there is a unique map $g\colon d\to c$, and by uniqueness this must be an inverse of $f$. Hence $f$ is an equivalence. Conversely, if $f$ is an equivalence, then $c$ and $d$ are equivalent as objects in $\I{C}$, which implies that the two maps $c,d\colon 1\rightrightarrows\I{C}$ in $\Cat(\Over{\BB}{A})$ are equivalent, which shows that $d$ must be initial.
\end{proof}
As a consequence of Corollary~\ref{cor:universalPropertyInitialObject} (and its dual), initial (and final) objects in a $\BB$-category $\I{C}$ are unique. We will usually denote an initial object by $0_{\I{C}}\colon A\to\I{C}$ and a final object by $1_{\I{C}}\colon A\to\I{C}$.

\subsection{Covariant equivalences}
\label{sec:covariantEquivalences}
Let $L\colon \Fun(\Delta^1,\Simp\BB)\to \LFib$ denote a left adjoint to the inclusion and let $\Over{L}{C}\colon \Over{(\Simp\BB)}{C}\to \Over{\LFib}{C}$ be the induced functor on the fibre over a simplicial object $C$ in $\BB$. 
\begin{definition}
\label{def:covariantEquivalence}
Let $C$ be a simplicial object in $\BB$ and let $f\colon P\to Q$ be a map in $\Over{(\Simp\BB)}{C}$. Then $f$ is said to be a \emph{covariant equivalence} if $\Over{L}{C}(f)$ is an equivalence in $\Over{\LFib}{C}$.
\end{definition}
\begin{remark}
\label{rem:covariantEquivalencesExplicitly}
In the context of Definition~\ref{def:covariantEquivalence}, the map $\Over{L}{C}(f)$ is constructed by means of the unique commutative diagram
\begin{equation*}
\begin{tikzcd}[column sep=small]
P\arrow[d]\arrow[rr, "f"]&&Q\arrow[d]\\
\Over{L}{C}(P)\arrow[rr, "\Over{L}{C}(f)"] \arrow[dr] && \Over{L}{C}(Q)\arrow[dl]\\
& C&
\end{tikzcd}
\end{equation*}
in which the two vertical maps are initial and the two diagonal maps are left fibrations. In particular, if $f$ is initial then $f$ is a covariant equivalence over $C$. The converse implication is true whenever the map $Q\to C$ is already a left fibration.
\end{remark}

The main goal of this section is to prove the following characterisation of covariant equivalences over $\I{C}$:
\begin{proposition}
	\label{prop:characterisationCovariantEquivalences}
	Let $\I{C}$ be a $\BB$-category and let
	\begin{equation*}
		\begin{tikzcd}[column sep=small]
			P\arrow[dr, "p"']\arrow[rr, "f"] && Q\arrow[dl, "q"] \\
			& \I{C}
		\end{tikzcd}
	\end{equation*}
	be a commutative triangle in $\Simp\BB$. Then the following are equivalent:
	\begin{enumerate}
	\item $f$ is a covariant equivalence over $\I{C}$;
	\item for any object $c\colon A\to \I{C}$ the induced map $\Over{f}{c}\colon \Over{P}{c}\to \Over{Q}{c}$ is a covariant equivalence over $\Over{\I{C}}{c}$;
	\item for any object $c\colon A\to \I{C}$ the induced map $\colim_{\Delta^{\op}}(\Over{f}{c})\colon \colim_{\Delta^{\op}}(\Over{P}{c})\to \colim_{\Delta^{\op}}(\Over{Q}{c})$ is an equivalence in $\BB$.
	\end{enumerate}
\end{proposition}

The proof of Proposition~\ref{prop:characterisationCovariantEquivalences} is based on the concept of a \emph{proper map}. Observe that for any map $p\colon P\to C$ in $\Simp\BB$ the commutative square
\begin{equation*}
	\begin{tikzcd}
		\Over{\LFib}{C}\arrow[d, "p^\ast"]\arrow[r, hookrightarrow] & \Over{(\Simp\BB)}{C}\arrow[d, "p^\ast"]\\
		\Over{\LFib}{P}\arrow[r, hookrightarrow] & \Over{(\Simp\BB)}{P}
	\end{tikzcd}
\end{equation*}
gives rise to a  left lax square
\begin{equation*}
	\begin{tikzcd}
		\Over{\LFib}{C}\arrow[d, "p^\ast"]\arrow[from=r, "\Over{L}{C}"'] & \Over{(\Simp\BB)}{C}\arrow[d, "p^\ast"]\\
		\Over{\LFib}{P}\arrow[from=r, "\Over{L}{D}"'] & \Over{(\Simp\BB)}{\I{P}}\arrow[ul, Rightarrow, shorten=3mm]
	\end{tikzcd}
\end{equation*}
by means of the mate construction. As $L$ does not preserve pullbacks, this square does not commute in general.
\begin{definition}
	\label{def:properSmooth}
	A map $p\colon D\to C$ in $\Simp\BB$ is said to be \emph{proper} if for any cartesian square
	\begin{equation*}
		\begin{tikzcd}
			Q\arrow[d, "q"]\arrow[r] & P\arrow[d, "p"]\\
			D\arrow[r] & C
		\end{tikzcd}
	\end{equation*}
	in $\Simp\BB$ the left lax square
	\begin{equation*}
		\begin{tikzcd}
			\Over{\LFib}{D}\arrow[d, "q^\ast"]\arrow[from=r, "\Over{L}{D}"'] & \Over{(\Simp\BB)}{D}\arrow[d, "q^\ast"]\\
			\Over{\LFib}{Q}\arrow[from=r, "\Over{L}{Q}"'] & \Over{(\Simp\BB)}{Q}\arrow[ul, Rightarrow, shorten=3mm]
		\end{tikzcd}
	\end{equation*}
	commutes. Dually, a map $p\colon P\to C$ is \emph{smooth} if $P^{\op}\to C^{\op}$ is proper.
\end{definition}
Unwinding the definitions, the left lax square from Definition~\ref{def:properSmooth} is commutative if and only if for any $f\colon E\to D$ the lower square in the commutative diagram
\begin{equation*}
	\begin{tikzcd}
		q^\ast E\arrow[d]\arrow[r]\arrow[dd, bend right=40, "q^\ast f"'] & E\arrow[d]\arrow[dd, bend left=40, "f"]\\
		L(q^\ast E)\arrow[r] \arrow[d] & L(E)\arrow[d]\\
		Q\arrow[r, "q"] & D
	\end{tikzcd}
\end{equation*}
is cartesian. Here the vertical maps are given by the factorisation of $E\to D$ and $q^\ast E\to Q$ into an initial map and a left fibration. This is equivalent to the condition that base change along $q$ preserves initial maps. In fact, if this condition is satisfied, then in particular the map $L(q^\ast E)\to q^\ast L(E)$ is initial. But since this map must also be a left fibration, it is necessarily an equivalence. The converse direction follows from chasing an initial map through the commutative square that is provided in the definition of proper maps.

By definition, proper maps preserve covariant equivalences:
\begin{proposition}
\label{prop:properFunctorCovariantEquivalences}
If $p\colon P\to C$ is a proper map between simplicial objects in $\BB$, then the base change functor $p^\ast\colon \Over{(\Simp\BB)}{C}\to\Over{(\Simp\BB)}{P}$ carries covariant equivalences over $C$ to covariant equivalences over $P$.\qed
\end{proposition}

\begin{proposition}
	For any two simplicial objects $C$ and $D$ in $\BB$ the projection $C\times D\to C$ is proper.
\end{proposition}
\begin{proof}
	This follows immediately from the fact that initial maps are stable under products (since initial maps are internally left orthogonal to left fibrations).
\end{proof}

The central ingredient towards the proof of Proposition~\ref{prop:characterisationCovariantEquivalences} is the following proposition:
\begin{proposition}
	\label{prop:rightFibrationProper}
	Any right fibration between simplicial objects in $\BB$ is proper.
\end{proposition}
\begin{proof}
	Since right fibrations are stable under pullbacks, it suffices to show that the base change along a right fibration preserves initial maps. To that end, let $S$ be the set of maps in $\Simp\BB$ whose base change along right fibrations results in an initial map. We claim that $S$ is saturated. In fact, it is obvious that $S$ is closed under composition and contains all equivalences, and the stability of $S$ under pushouts and small colimits in $\Fun(\Delta^1,\Simp\BB)$ follows from the fact that $\RFib$ defines a sheaf on $\Simp\BB$ (as it is defined as the right orthogonality class of a factorisation system). As a consequence, it suffices to show that the initial map $d^1\colon D\hookrightarrow \Delta^1\otimes D$ is contained in $S$ for any simplicial object $D$ in $\BB$. 
	
	Let therefore
	\begin{equation*}
		\begin{tikzcd}
			Q\arrow[r, "f"] \arrow[d] & P\arrow[d, "p"]\\
			D\arrow[r, "d^1"] & \Delta^1\otimes D
		\end{tikzcd}
	\end{equation*}
	be a cartesian square in $\Simp\BB$ such that $p$ is a right fibration. Let $\tau\colon\Delta^1\times\Delta^1\to\Delta^1$ be the map that sends the final vertex $(1,1)$ to $1$ and any other vertex to $0$. We then obtain a commutative diagram
	\begin{equation*}
		\begin{tikzcd}
			&\Delta^1\arrow[d, "\id\times d^1"']\arrow[r, "s_0"] &\Delta^0\arrow[d, "d^1"] \\
			\Delta^1\arrow[r, "d^0\times\id"]&\Delta^1\times\Delta^1\arrow[r, "\tau"] & \Delta^1\\
			&\Delta^1\arrow[u, "\id\times d^0"]\arrow[ur, "\id"] & 
		\end{tikzcd}
	\end{equation*}
	of $\infty$-categories such that the composition of the two horizontal arrows is the identity. Now let $(\Delta^1\otimes Q)\sqcup_{Q}P$ be the pushout of $f$ along the map $d^0\colon Q\to \Delta^1\otimes Q$, and observe that the induced map $(\Delta^1\otimes Q)\sqcup_{Q}P\to \Delta^1\otimes P$ is final (by making use of item~(2) of Proposition~\ref{prop:propertiesFactorisationSystems}). Therefore the lifting problem
	\begin{equation*}
		\begin{tikzcd}[column sep=large, row sep=large]
			(\Delta^1\otimes Q)\sqcup_{Q}P\arrow[r, "{(f\otimes s^0, \id)}"]\arrow[d] & P\arrow[d, "p"]\\
			\Delta^1\otimes P\arrow[r, "(\tau\otimes\id)\circ(\id\otimes p)"]\arrow[ur, dotted, "h"] & \Delta^1\otimes D
		\end{tikzcd}
	\end{equation*}
	admits a unique solution $h$. Let $r\colon P\to Q$ be defined as the unique functor that makes the diagram
	\begin{equation*}
		\begin{tikzcd}[column sep={2cm,between origins}, row sep=small]
			& P\arrow[dl, "r"]\arrow[rr, "d^1"]\arrow[dd, "p", near start] && \Delta^1\otimes P\arrow[dl, "h"]\arrow[dd, "\id\otimes p"] \\
			Q\arrow[dd, "q"]\arrow[rr, "f", near end, crossing over] && P\\
			& \Delta^1\otimes D\arrow[dl, "s^0"']\arrow[rr, "d^1\otimes\id", near start]&& (\Delta^1\times\Delta^1)\otimes D\arrow[dl, "\tau\otimes\id"]\\
			D\arrow[rr, "d^1"] && \Delta^1\otimes D\arrow[from=uu, near start, crossing over, "p"]
		\end{tikzcd}
	\end{equation*}
	commute. By construction, this map satisfies $rf\simeq\id$ and moreover fits into the commutative diagram
	\begin{equation*}
		\begin{tikzcd}
			P\arrow[d, "d^1"]\arrow[dr, "fr"] &\\
			\Delta^1\otimes P\arrow[r, "h"] & P\\
			P.\arrow[u,"d^0"']\arrow[ur, "\id"'] &
		\end{tikzcd}
	\end{equation*}
	Hence the commutative diagram
	\begin{equation*}
		\begin{tikzcd}
			Q\arrow[d, "f"]\arrow[r] & \Delta^1\otimes Q\sqcup_{Q}P\arrow[r, "{(s^0, r)}"]\arrow[d] & Q\arrow[d, "f"] \\
			P\arrow[r, "d^1"] &\Delta^1\otimes P\arrow[r, "s^0"] & P
		\end{tikzcd}
	\end{equation*}
	exhibits $f$ as a retract of the initial map $\Delta^1\otimes Q\sqcup_{Q}P\to \Delta^1\otimes P$ (in which the domain is the pushout of $f$ along the inclusion $d^1\colon Q\to \Delta^1\otimes Q$) and therefore as an initial map itself.
\end{proof}

\begin{proof}[{Proof of Proposition~\ref{prop:characterisationCovariantEquivalences}}]
	Suppose first that $f$ is a covariant equivalence over $\I{C}$. Since the projection $(\pi_c)_!\colon \Over{\I{C}}{c}\to \I{C}$ is a right fibration and since by Proposition~\ref{prop:rightFibrationProper} any right fibration is proper, Proposition~\ref{prop:properFunctorCovariantEquivalences} implies that the map $\Over{f}{c}$ must be a covariant equivalence over $\Over{\I{C}}{c}$.
	
	Suppose now that $\Over{f}{c}$ is a covariant equivalence over $\Over{\I{C}}{c}$, i.e.\ that $\Over{L}{(\Over{\I{C}}{c})}(f)$ is an equivalence. Quite generally, note that for any simplicial object $D$ in $\BB$ the base change functor $\pi_{D}^\ast\colon\Simp\BB\to\Over{(\Simp\BB)}{\I{D}}$ admits a left adjoint $(\pi_{D})_!$ that is given by the forgetful functor, which implies that the base change functor $\pi_{D}^\ast\colon\BB\to\Over{\LFib}{D}$ admits a left adjoint $(\pi_{D})_!$ as well that is explicitly given by the composition
	\begin{equation*}
	\Over{\LFib}{D}\into \Over{(\Simp\BB)}{D}\xrightarrow{(\pi_{D})_!} \Simp\BB\xrightarrow{\colim_{\Delta^{\op}}} \BB,
	\end{equation*}
	cf.\ Remark~\ref{rem:groupoidificationFibrantReplacement}.
	One consequently obtains a commutative square
	\begin{equation*}
	\begin{tikzcd}
	\Over{(\Simp\BB)}{D}\arrow[r, "(\pi_{D})_!"]\arrow[d, "\Over{L}{D}"] & \Simp\BB\arrow[d, "\colim_{\Delta^{\op}}"]\\
	\Over{\LFib}{D}\arrow[r, "(\pi_{D})_!"] & \BB.
	\end{tikzcd}
	\end{equation*}
	Applying this observation to $D=\Over{\I{C}}{c}$, one finds that the map $\colim_{\Delta^{\op}}(\Over{f}{c})$ arises as the image of $\Over{L}{(\Over{\I{C}}{c})}(f)$ along the functor $(\pi_{\Over{\I{C}}{c}})_!$ and is therefore an equivalence.
	
	Lastly, assume that $\colim_{\Delta^{\op}}(\Over{f}{c})$ is an equivalence in $\BB$ for every object $c\colon A\to\I{C}$ in context $A\in\BB$. We need to show that the map $\Over{L}{\I{C}}(f)\colon \Over{L}{\I{C}}(P)\to \Over{L}{\I{C}}(Q)$ in $\Over{\LFib}{\I{C}}$ is an equivalence. By Proposition~\ref{prop:equivalenceLeftFibrationsFibrewise}, it suffices to show that the map $\Over{L}{\I{C}}(f)\vert_c\colon \Over{L}{\I{C}}(P)\vert_c\to \Over{L}{\I{C}}(Q)\vert_c$ that is induced on the fibres over $c\colon A\to \I{C}$ is an equivalence in $\BB$ for all objects $c$ in $\I{C}$. By making use of the factorisation of $c$ into the canonical final map $\id_c\colon A\to \Over{\I{C}}{c}$ followed by the right fibration $(\pi_c)_!\colon \Over{\I{C}}{c}\to\I{C}$,
	one obtains a pullback square
	\begin{equation*}
	\begin{tikzcd}
	\Over{L}{\I{C}}(P)\vert_c\arrow[r, "{\Over{L}{\I{C}}(f)\vert_c}"]\arrow[d] & \Over{L}{\I{C}}{(Q)}\vert_c \arrow[d]\\
	\Over{\Over{L}{\I{C}}(P)}{c}\arrow[r, "{\Over{(\Over{L}{\I{C}}(f))}{c}}"] & \Over{\Over{L}{\I{C}}(Q)}{c}
	\end{tikzcd}
	\end{equation*}
	in which the vertical maps are final since they arise as pullbacks of the final map $\id_c$ along left fibrations and since the dual of Proposition~\ref{prop:rightFibrationProper} implies that left fibrations are smooth.
	Note that the fibres $\Over{L}{\I{C}}(P)\vert_c$ and $\Over{L}{\I{C}}(Q)\vert_c$ are contained in $\BB$. Hence the functor $\colim_{\Delta^{\op}}$ sends the two vertical maps in the above square to equivalences in $\BB$ while leaving the upper horizontal map unchanged. We conclude that ${\Over{L}{\I{C}}(\I{f})\vert_c}$ is an equivalence whenever $\colim_{\Delta^{\op}}({\Over{(\Over{L}{\I{C}}(f))}{c}})$ is one, and since the latter recovers the map $\colim_{\Delta^{\op}}(\Over{f}{c})$ (again using properness of the right fibration $(\pi_c)_!\colon \Over{\I{C}}{c}\to\I{C}$), the result follows.
\end{proof}
Proposition~\ref{prop:characterisationCovariantEquivalences} can be used to derive an internal version of Quillen's theorem A:
\begin{corollary}
	\label{cor:QuillenA}
	A functor $f\colon\I{J}\to\I{I}$ between $\BB$-categories is initial if and only if for every object $i$ in $\I{I}$ in context $A\in \BB$ the canonical map $(\Over{\I{J}}{i})^{\gp}\to A$ is an equivalence.
\end{corollary}
\begin{proof}
	On account of Remark~\ref{rem:covariantEquivalencesExplicitly}, the map $f$ is initial if and only if it is a covariant equivalence over $\I{I}$. By Proposition~\ref{prop:characterisationCovariantEquivalences}, this is the case if and only if for every object $i\colon A\to\I{I}$ the map $(\Over{\I{J}}{i})^{\gp}\to(\Over{\I{I}}{i})^{\gp}$ is an equivalence. By construction, there is a commutative diagram
	\begin{equation*}
		\begin{tikzcd}[column sep=small]
		(\Over{\I{J}}{i})^{\gp}\arrow[rr]\arrow[dr] && (\Over{\I{I}}{i})^{\gp}\arrow[dl]\\
		& A.&
		\end{tikzcd}
	\end{equation*}
	Therefore, the proof is finished once we show that the map $(\Over{\I{I}}{i})^{\gp}\to A$ is an equivalence. But this follows from the observation that the map $\Over{\I{I}}{i}\to A$ is final as it is a retraction of the final section $\id_c\colon A\to \Over{\I{I}}{c}$.
\end{proof}

\begin{remark}
	\label{rem:initialFunctorsLocal}
	As a consequence of Corollary~\ref{cor:QuillenA}, the condition of a functor between $\BB$-categories to be initial is \emph{local}. More precisely, if $\bigsqcup_k A_k\onto 1$ is a cover in $\BB$ and if $f\colon\I{J}\to\I{I}$ is a functor between $\BB$-categories, then $f$ is initial if and only if $\pi_{A_k}^\ast f$ is initial for all $k$. In fact, Corollary~\ref{cor:QuillenA} tells us that $f$ being initial is equivalent to the map $(\Over{\I{J}}{i})^\gp\to A$ being an equivalence for every object $i\colon A\to \I{I}$. Using Lemma~\ref{lem:BCFunctorCategory}, we obtain an equivalence $\pi_{A_k}^\ast(\Over{\I{J}}{i})\simeq\Over{(\pi_{A_k}^\ast\I{J})}{\pi_{A_k}^\ast i}$ over $A_k$ for all $k$. Since $\pi_{A_k}^\ast$ moreover commutes with the groupoidification functor (Proposition~\ref{prop:universeEnlargementCategoriesGroupoidification}), we may thus identify the pullback of the map $(\Over{\I{J}}{i})^\gp\to A$ along $\pi_{A_k}$ with the map $(\Over{\pi_{A_k}^\ast\I{J}}{\pi_{A_k}^\ast i})^\gp \to \pi_{A_k}^\ast A$ in $\Over{\BB}{{A_k}}$. The claim now follows from the fact that the algebraic morphism $\BB\to\prod_k \Over{\BB}{A_k}$ is conservative since $\bigsqcup_k A_k\onto 1$ is a cover in $\BB$.
\end{remark}
We end this section with yet another characterisation of covariant equivalences that will be useful later:
\begin{proposition}
	\label{prop:initialityFibrewiseSmooth}
	Let $\I{C}$ be a $\BB$-category and let
	\begin{equation*}
		\begin{tikzcd}[column sep=small]
			P\arrow[dr, "p"']\arrow[rr, "f"] && Q\arrow[dl, "q"] \\
			& \I{C}
		\end{tikzcd}
	\end{equation*}
	be a commutative triangle in $\Simp\BB$ in which both $p$ and $q$ are smooth. Then $f$ is a covariant equivalence over $\I{C}$ if and only if for any object $c$ in $\I{C}$ the induced map $f\vert_c\colon \colim_{\Delta^{\op}}(P\vert_c)\to \colim_{\Delta^{\op}}(Q\vert_c)$ is an equivalence.
\end{proposition}
\begin{proof}
	$f$ is a covariant equivalence over $\I{C}$ if and only if $\Over{L}{\I{C}}(f)$ is an equivalence. Let
	\begin{equation*}
	\begin{tikzcd}
	P\arrow[r, "f"]\arrow[d, "i"] & Q\arrow[d, "j"]\\
	\Over{L}{\I{C}}(P)\arrow[r, "\Over{L}{\I{C}}(f)"] & \Over{L}{\I{C}}(Q)
	\end{tikzcd}
	\end{equation*}
	be the canonical square in which the two vertical maps are obtained from the adjunction unit and are therefore initial.
	Since $\Over{L}{\I{C}}(f)$ is a map in of left fibrations over $\I{C}$, Proposition~\ref{prop:equivalenceLeftFibrationsFibrewise} implies that this map is an equivalence if and only if the induced map $\Over{L}{\I{C}}(f)\vert_c$ on the fibres over $c$ is one for every object $c\colon A\to \I{C}$. It therefore suffices to show that in the induced commutative diagram
	\begin{equation*}
	\begin{tikzcd}
	P\vert_c\arrow[r, "f\vert_c"]\arrow[d, "i\vert_c"] & Q\vert_c\arrow[d, "j\vert_c"]\\
	\Over{L}{\I{C}}(P)\vert_c\arrow[r, "\Over{L}{\I{C}}(f)\vert_c"] & \Over{L}{\I{C}}(Q)\vert_c
	\end{tikzcd}
	\end{equation*}
	of the fibres over $c$ the maps $i\vert_c$ and $j\vert_c$ are initial. We will show this for $i\vert_c$, the case of $j\vert_c$ is analogous.
	
	Let $p^{\prime}\colon \Over{L}{\I{C}}(P)\to\I{C}$ be the structure map, and consider the commutative diagram
	\begin{equation*}
		\begin{tikzcd}[column sep=small, row sep=small]
			& P\vert_c \arrow[dd, "\id", near start]\arrow[rr]\arrow[dl, "i\vert_c"', near start]  && \Over{P}{c} \arrow[dl, "\Over{i}{c}"']\arrow[dd, "\id", near start]\arrow[rr] && P\arrow[dl, "i"']\arrow[dd, "\id", near start] \\
			\Over{L}{\I{C}}(P)\vert_c\arrow[dd, "p^{\prime}\vert_c"', near start]\arrow[rr, crossing over] && \Over{\Over{L}{\I{C}}(P)}{c} \arrow[rr, crossing over] && \Over{L}{\I{C}}(P)\\
			& P\vert_c \arrow[rr]\arrow[dl, "p\vert_c"'] && \Over{P}{c} \arrow[dl, "\Over{p}{c}"']\arrow[rr] && P\arrow[dl, "p"'] \\
			A\arrow[rr, "\id_c"] && \Over{\I{C}}{c}\arrow[rr, "(\pi_c)_!"]\arrow[from=uu, crossing over, "\Over{p^{\prime}}{c}"', near start] && \I{C}.\arrow[from=uu, crossing over, "p^{\prime}"', near start]
		\end{tikzcd}
	\end{equation*}
	The projection $(\pi_c)_!$ is a right fibration, hence so are the maps $\Over{P}{c}\to P$ and $\Over{\Over{L}{\I{C}}(P)}{c}\to \Over{L}{\I{C}}(P)$. Since right fibrations are proper by Proposition~\ref{prop:rightFibrationProper}, we conclude that the map $\Over{i}{c}$ must be initial. Moreover, since $p$ and $p^{\prime}$ are smooth the maps $\Over{p}{c}$ and $\Over{p^\prime}{c}$ must be smooth as well, which implies that the maps $P\vert_c\to \Over{P}{c}$ and $\Over{L}{\I{C}}(P)\vert_c\to \Over{\Over{L}{\I{C}}(P)}{c}$ must be final since $\id_c$ is final. We therefore obtain a pullback square
	\begin{equation*}
		\begin{tikzcd}
			P\vert_c\arrow[d, "i\vert_c"]\arrow[r] & \Over{P}{c}\arrow[d, "\Over{i}{c}"]\\
			\Over{L}{\I{C}}(P)\vert_c\arrow[r] & \Over{\Over{L}{\I{C}}(P)}{c}
		\end{tikzcd}
	\end{equation*}
	in which the horizontal maps are final and the vertical map on the right is initial. Since the  functor $\colim_{\Delta^{\op}}$ carries both final and initial maps to equivalences in $\BB$ (cf.\ Remark~\ref{rem:localEquivalences} and Remark~\ref{rem:groupoidificationFibrantReplacement}), the map $\colim_{\Delta^{\op}}(i\vert_c)$ must be an equivalence. But as $\Over{L}{\I{C}}(P)\vert_c$ is already contained in $\BB$, the map $i\vert_c$ is equivalent to the composition of an initial map with an equivalence and therefore initial itself.
\end{proof}

\subsection{The Grothendieck construction}
\label{sec:GrothendieckConstruction}

For any $\BB$-category $\I{C}$, let $\ILFib_{\I{C}}$ be the presheaf on $\BB$ that is given by the assignment $A\mapsto \LFib(\I{C}\times A)$. This defines a functor $\I{C}\mapsto \ILFib_{\I{C}}$ from $\Cat(\BB)$ into $\PSh_{\CatSS}(\BB)$.
\begin{theorem}
	\label{thm:internalGrothendieck}
	For any $\BB$-category, the presheaf $\ILFib_{\I{C}}$ defines a large $\BB$-category, and there is a canonical equivalence
	\begin{equation*}
		\iFun{\I{C}}{\Univ}\simeq \ILFib_{\I{C}}
	\end{equation*}
	that is natural in $\I{C}$.
\end{theorem}
Note that the statement of Theorem~\ref{thm:internalGrothendieck} can be reformulated as follows: By making use of the embedding $\Cat(\BB)\hookrightarrow \PSh_{\CatSS}(\BB)$, one sees that the functor $\iFun{-}{\Univ}$ is equivalent to the bifunctor
\begin{equation*}
	\Fun_{\BBB}(-\times -, \Univ)\colon \BB^{\op}\times \Cat(\BB)^{\op}\to \CatSS.
\end{equation*}
Similarly, the functor $\ILFib_{(-)}$ corresponds to the bifunctor
\begin{equation*}
	\LFib(-\times -)\colon \BB^{\op}\times \Cat(\BB)^{\op}\to \CatSS
\end{equation*}
under this identification. It therefore suffices to show that there is an equivalence 
\begin{equation*}
	\Fun_{\BBB}(-,\Univ)\simeq \LFib
\end{equation*}
of $\CatSS$-valued presheaves on $\Cat(\BB)$. Our strategy is to even show that there is an equivalence
\begin{equation*}
	\Fun_{\BBB}(-,\Univ)\simeq \LFib
\end{equation*}
of $\CatSS$-valued presheaves on $\Simp\BB$. Upon restricting this equivalence along the inclusion $\Cat(\BB)\into\Simp\BB$, the desired result will follow.

As a first step, let us choose a small $\infty$-category $\GG$ such that $\BB$ arises as a left exact and accessible localisation of $\PSh_{\SS}(\GG)$. Then $\Simp\BB$ is then given as a localisation of $\PSh_{\SS}(\Delta\times \GG)$, hence the composition of the localisation functor $L$ with the Yoneda embedding $h$ defines a functor
\begin{equation*}
	Lh\colon \Delta\times\GG\to \Simp\BB.
\end{equation*}
Note that by its very construction this is precisely the restriction of the tensoring bifunctor $-\otimes -$ along $\Delta\times\GG\to\Simp\SS\times\Simp\BB$.
\begin{lemma}
	\label{lem:coYoneda}
	The identity on $\Simp\BB$ is a left Kan extension of the functor $Lh$ along itself.
\end{lemma}
\begin{proof}
	Let $i\colon \Simp\BB\hookrightarrow \PSh_{\SS}(\Delta\times\GG)$ denote the inclusion. In the diagram
	\begin{equation*}
		\begin{tikzcd}
			\Fun(\PSh_{\SS}(\Delta\times\GG),\PSh_{\SS}(\Delta\times\GG))\arrow[r, shift left, "L_\ast"]\arrow[from=r, shift left, "i_\ast"] \arrow[from=d, shift left, "h_!"]\arrow[d, shift left, "h_\ast"] & \Fun(\PSh_{\SS}(\Delta\times\GG), \Simp\BB) \arrow[r, shift left, "i^\ast"]\arrow[from=r, shift left, "L^\ast"]\arrow[from=d, shift left, "h_!"]\arrow[d, shift left, "h^\ast"]& \Fun(\Simp\BB, \Simp\BB)\\
			\Fun(\Delta\times\GG, \PSh_{\SS}(\Delta\times\GG))\arrow[r, shift left, "L_\ast"]\arrow[from=r, shift left, "i_\ast"] & \Fun(\Delta\times\GG, \Simp\BB)
		\end{tikzcd}
	\end{equation*}
	the square of right adjoints is commutative, hence the the square of left adjoints must commute as well. By~\cite[Lemma~5.1.5.3]{htt} the identity on $\PSh_{\SS}(\Delta\times\GG)$ arises as the left Kan extension of $h$ along itself, i.e.\ as the functor $h_!(h)$. Since moreover the composite $i^\ast L_\ast$ sends the identity on $\PSh_{\SS}(\Delta\times\GG)$ to the identity on $\Simp\BB$, a simple diagram chase shows that the identity on $\Cat(\BB)$ is equivalent to the functor $i^\ast h_!(Lh)$. We conclude by observing that since the adjunction $L\dashv i$ induces an adjunction $i^\ast\dashv L^\ast$,  the functor $i^\ast h_!$ is the functor of left Kan extension along $Lh$.
\end{proof}
Our next step towards the proof of Theorem~\ref{thm:internalGrothendieck} will be to establish an equivalence
\begin{equation*}
	\Fun_{\BBB}(\Delta^{\bullet}\otimes -  ,\Univ)\simeq \LFib(\Delta^{\bullet}\otimes -  )
\end{equation*}
of $\CatSS$-valued presheaves on $\Delta\times\GG$. To that end, for any $n\geq 0$ an any $G\in\GG$ let us denote by $\Delta^n\times G$ the presheaf on $\Delta\times\GG$ that is represented by the pair $(\ord{n},G)$. We thus obtain $L(\Delta^n\times G)\simeq \Delta^n\otimes G$ (where we abuse notation by identifying $G$ with its image along the functor $\GG\to \BB$).
\begin{lemma}
	\label{lem:embeddingLFibPSh}
	Assigning to a left fibration $P\to \Delta^n\otimes G$ in $\Simp\BB$ its pullback along the adjunction unit $\Delta^n\times G \to \Delta^n\otimes G$ in $\PSh_{\SS}(\Delta\times\GG)$ defines an embedding
	\begin{equation*}
		\LFib(\Delta^{\bullet}\otimes -  )\hookrightarrow \Over{\PSh_{\SS}(\Delta\times\GG)}{\Delta^{\bullet}\times - }
	\end{equation*}
	of presheaves on $\Delta\times\GG$. 
\end{lemma}
\begin{proof}
	Since the presheaf $\LFib(\Delta^{\bullet}\otimes -  )$ embeds into the presheaf $\Over{(\Simp\BB)}{\Delta^{\bullet}\otimes -  }$, it suffices to prove that one can define an embedding
	\begin{equation*}
		\Over{(\Simp\BB)}{\Delta^{\bullet}\otimes -  }\hookrightarrow \Over{\PSh_{\SS}(\Delta\times\GG)}{\Delta^{\bullet}\times - }
	\end{equation*}
	in the desired way. Since the localisation functor $\PSh_{\SS}(\Delta\times\GG)\to \BB_{\Delta}$ is left exact, there is a functorial map
	\begin{equation*}
		\Over{\PSh_{\SS}(\Delta\times\GG)}{\Delta^{\bullet}\times -}\to \Over{(\BB_{\Delta})}{\Delta^{\bullet}\otimes -  }
	\end{equation*}
	that is given on the level of cartesian fibrations by the pullback of the natural map
	\begin{equation*}
	\Fun(\Delta^1,\PSh_{\SS}(\Delta\times\GG))\to \Fun(\Delta^1,\Simp\BB)\times_{\PSh_{\SS}(\Delta\times\GG)}\Simp\BB
	\end{equation*}
	along the Yoneda embedding $\Delta\times\GG\into\PSh_{\SS}(\Delta\times\GG)$.
	On the fibre over $(\ord{n},G)$, this functor is given by the map that is naturally induced by the localisation functor $\PSh_{\SS}(\Delta\times\GG)\to\Simp\BB$ upon taking slice $\infty$-categories. Hence there are fibrewise right adjoints that are given by composing the natural map
	\begin{equation*}
		\Over{(\Simp\BB)}{\Delta^n\otimes G}\to \Over{\PSh_{\SS}(\Delta\times \GG)}{\Delta^n\otimes G}
	\end{equation*}
	that is induced by the inclusion $\BB_{\Delta}\hookrightarrow\PSh_{\SS}(\Delta\times\GG)$ with the pullback functor along the adjunction unit $\Delta^n\times G\to \Delta^n\otimes G$ in $\PSh_{\SS}(\Delta\times\GG)$. Note that each of these fibrewise right adjoints is fully faithful since the localisation functor commutes with pullbacks. We conclude by observing that these fibrewise right adjoints assemble to a map of $\CatSS$-valued presheaves on $\Delta\times\GG$.
\end{proof}
By~\cite[Proposition~9.8]{haugseng2017} there is a functorial equivalence
\begin{equation*}
	\Over{\PSh_{\SS}(\Delta\times\GG)}{\Delta^{\bullet}\times - }\simeq \PSh_{\SS}(\Over{(\Delta\times\GG)}{\Delta^{\bullet}\times - })
\end{equation*}
in which the right-hand side can furthermore be functorially identified with $\PSh_{\SS}(\Over{\Delta}{\Delta^{\bullet}}\times\Over{\GG}{-})$. Combining this result with Lemma~\ref{lem:embeddingLFibPSh}, we conclude that there is an embedding
\begin{equation*}
	\LFib(\Delta^{\bullet}\otimes -  )\hookrightarrow \PSh_{\SS}(\Over{\Delta}{\Delta^{\bullet}}\times\Over{\GG}{-})
\end{equation*}
that sends a left fibration $P\to \Delta^n\otimes G$ in $\Simp\BB$ to the presheaf that maps a pair $(\tau\colon \ord{k}\to \ord{n},s\colon H\to G)$ to the fibre of the map $P_k(H)\to \const_{\BB}(\Delta^n_k)\times G(H)$ over the point that is determined by $(\tau,s)$. Note that this embedding factors through the inclusion
\begin{equation*}
	\PSh_{\Over{\BB}{-}}(\Over{\Delta}{\Delta^{\bullet}})\hookrightarrow \PSh_{\PSh_{\SS}(\Over{\GG}{-})}(\Over{\Delta}{\Delta^{\bullet}})\simeq \PSh_{\SS}(\Over{\Delta}{\Delta^{\bullet}}\times\Over{\GG}{-}),
\end{equation*}
which implies that we end up with a functorial embedding
\begin{equation*}
	\LFib(\Delta^{\bullet}\otimes -  )\hookrightarrow\PSh_{\Over{\BB}{-}}(\Over{\Delta}{\Delta^{\bullet}}).
\end{equation*}
Our next goal is to characterise the essential image of this embedding:
\begin{lemma}
	\label{lem:characterisationLFibPresheaf}
	For any $G\in\GG$ and any $n\geq 0$, a presheaf $F\in \PSh_{\Over{\BB}{G}}(\Over{\Delta}{\Delta^{n}})$ is contained in the essential image of the inclusion $\LFib(\Delta^n\otimes G)\into \PSh_{\Over{\BB}{G}}(\Over{\Delta}{\Delta^{n}})$ if and only if for any $k\geq 1$ and any map $\tau\colon \ord{k}\to \ord{n}$ in $\Delta$, the inclusion $\delta^{\{0\}}\colon \ord{0}\to \ord{k}$ induces an equivalence $F(\tau)\simeq F(\tau \delta^{\{0\}})$.
\end{lemma}
\begin{proof}
	Let $F$ be a $\Over{\BB}{G}$-valued presheaf on $\Over{\Delta}{\Delta^n}$ and let $P\to \Delta^n\times G$ be the map in $\PSh_{\SS}(\Delta\times\GG)$ that corresponds to $F$ in view of the inclusion $\PSh_{\Over{\BB}{G}}(\Over{\Delta}{\Delta^n})\hookrightarrow \Over{\PSh_{\SS}(\Delta\times\GG)}{\Delta^n\times G}$. Then $P\to \Delta^n\times G$ is in the essential image of the inclusion $\Over{(\BB_{\Delta})}{\Delta^n\otimes G }\hookrightarrow \Over{\PSh_{\SS}(\Delta\times\GG)}{\Delta^n\times G}$. To see this, let $(L\dashv i)$ denotes the adjunction $\BB_{\Delta}\leftrightarrows \PSh_{\SS}(\Delta\times\GG)$. We need to show that the commutative square
	\begin{equation*}
		\begin{tikzcd}
			P\arrow[r]\arrow[d] & iL(P)\arrow[d]\\
			\Delta^n\times G \arrow[r] & i(\Delta^n\otimes G )
		\end{tikzcd}
	\end{equation*}
	that is induced by the adjunction unit $\id\to iL$ is a pullback square in $\PSh_{\SS}(\Delta\times\GG)$. It suffices to show this for each $k\in \Delta$ individually. In this case, the map $P_k\to \const(\Delta^n_k)\times G$ is given by the coproduct
	\begin{equation*}
		\bigsqcup_{\tau\colon \ord{k}\to \ord{n}} F(\tau)\to \bigsqcup_{\tau\colon \ord{k}\to \ord{n}} G,
	\end{equation*}
	hence it suffices to show that for each map $\tau\colon \ord{k}\to \ord{n}$ in $\Delta$ the square
	\begin{equation*}
		\begin{tikzcd}
			F(\tau)\arrow[d]\arrow[r] & iL(F(\tau))\arrow[d]\\
			G\arrow[r] & iL(G)
		\end{tikzcd}
	\end{equation*}
	is cartesian, which follows from $F(\tau)$ being contained in $\BB$. Hence $F$ is contained in $\LFib(\Delta^n\otimes G)$ precisely if the map $L(P)\to \Delta^n\otimes G$ is a left fibration.

	Using that $L$ is left exact as well as Proposition~\ref{prop:fibrationsExternalCharacterization}, one finds that map $L(P)\to \Delta^n\otimes G$ being a left fibration is equivalent to the square
	\begin{equation*}
		\begin{tikzcd}
			P_k\arrow[d]\arrow[r, "d_{\{0\}}"] & P_0\arrow[d]\\
			\Delta^n_k\times G\arrow[r, "d_{\{0\}}"] & \Delta^n_0\times G
		\end{tikzcd}
	\end{equation*}
	being a pullback diagram for all $k\geq 1$. On account of the commutative diagrams
	\begin{equation*}
		\begin{tikzcd}[column sep=tiny, row sep=tiny]
			& F(\tau)\arrow[dl]\arrow[rr]\arrow[dd] & & F(\tau \delta^{\{0\}})\arrow[dl]\arrow[dd]\\
			P_k\arrow[dd]\arrow[rr, "d_{\{0\}}", crossing over, near end] && P_0&\\
			& G \arrow[rr, "\id", near start]\arrow[dl, "\tau"]&& G\arrow[dl, "\tau \delta^{\{0\}}"]\\
			\Delta^n_k\times G \arrow[rr, "d_{\{0\}}", near end] && \Delta^n_0\times G\arrow[from=uu, crossing over]  &
		\end{tikzcd}
	\end{equation*}
	for all $\tau\colon \ord{k}\to \ord{n}$, in which the squares on the left and on the right are cartesian, this is seen to be equivalent to the map $\delta^{\{0\}}\colon \ord{0}\to \ord{k}$ inducing an equivalence $F(\tau)\simeq F(\tau \delta^{\{0\}})$.
\end{proof}

Recall that there is an equivalence
\begin{equation*}
	\Fun_{\BBB}(\Delta^{\bullet}\otimes -  ,\Univ)\simeq \Fun(\Delta^{\bullet}, \Univ(-))\simeq \Fun(\Delta^{\bullet}, \Over{\BB}{-})
\end{equation*}
of $\CatSS$-valued presheaves on $\Delta\times\GG$. Let 
\begin{equation*}
	\lambda\colon \Delta^n\to \left(\Over{\Delta}{\Delta^n}\right)^{\op}
\end{equation*}
denote the functor that sends $k\leq n$ to the inclusion $d^{\{k,\dots,n\}}\colon\Delta^{n-k}\subset \Delta^n$. This functor admits a right adjoint
\begin{equation*}
	\epsilon\colon \left(\Over{\Delta}{\Delta^n}\right)^{\op}\to \Delta^n
\end{equation*}
that sends $\tau\colon \ord{k}\to \ord{n}$ to $\tau(0)$. One easily checks that $\epsilon$ is natural in $n$. Moreover, since $\lambda\sigma$ is the identity functor on $\Delta^n$, this adjunction exhibits $\Delta^n$ as a localisation of $\left(\Over{\Delta}{\Delta^n}\right)^{\op}$. By precomposition, we therefore obtain a functorial embedding
\begin{equation*}
	\epsilon^\ast\colon \Fun(\Delta^{\bullet},\Over{\BB}{-})\hookrightarrow \PSh_{\Over{\BB}{-}}(\Over{\Delta}{\Delta^{\bullet}})
\end{equation*}
that exibits each $\infty$-category $\Fun(\Delta^n,\Over{\BB}{G})$ as a colocalisation of $\PSh_{\Over{\BB}{G}}(\Over{\Delta}{\Delta^{n}})$, with the right adjoint given by $\lambda^\ast$.
\begin{lemma}
	\label{lem:GrothendieckConstructionSimplices}
	For any pair $(\ord{n},G)\in\Delta\times\GG$, the essential image of the functor 
	\begin{equation*}
		\epsilon^\ast\colon \Fun(\Delta^{n},\Over{\BB}{G})\hookrightarrow \PSh_{\Over{\BB}{G}}(\Over{\Delta}{\Delta^{n}})
	\end{equation*}
	coincides with the essential image of the embedding
	\begin{equation*}
		\LFib(\Delta^n\otimes G)\hookrightarrow \PSh_{\Over{\BB}{G}}(\Over{\Delta}{\Delta^{n}}).
	\end{equation*}
\end{lemma}
\begin{proof}
	We first claim that $\sigma\colon \Delta^n\to\Over{\BB}{G}$ the associated presheaf $\sigma\epsilon\colon (\Over{\Delta}{\Delta^n})^{\op}\to\Over{\BB}{G}$ satisfies the condition of Lemma~\ref{lem:characterisationLFibPresheaf}. In fact, if $\tau\colon \ord{k}\to \ord{n}$ is a map in $\Delta$ with $k\geq 1$, then the map in $\Over{\BB}{G}$ that is induced by the inclusion $\delta^{\{0\}}\colon \ord{0}\to \ord{k}$ is simply the identity $\sigma(\tau(0))\simeq \sigma(\tau(0))$, hence the claim follows.
	
	To finish the proof, it now suffices to show that for any presheaf $F\in \PSh_{\Over{\BB}{G}}(\Over{\Delta}{\Delta^n})$ that satisfies the condition of Lemma~\ref{lem:characterisationLFibPresheaf} the adjunction counit $\epsilon^{\ast}\lambda^{\ast}F\to F$ is an equivalence. Since this map is given by precomposition with the adjunction counit of $\lambda\dashv \epsilon$, the map $\epsilon^{\ast}\lambda^{\ast}F\to F$ is defined on each object $\tau\colon \ord{k}\to \ord{n}$ in $\Over{\Delta}{\Delta^n}$ by applying $F$ to the map $\ord{k}\to \ord{n-\tau(0)}$
	over $\ord{n}$, where the structure map of the codomain into $\ord{n}$ is given by the inclusion $\delta^{\{\tau(0),\dots,n\}}$. Since precomposing this map with $\delta^{\{0\}}\colon \ord{0}\to \ord{k}$ recovers the inclusion $\delta^{\{0\}}\colon\ord{0}\to \ord{n-\tau(0)}$, the two out of three property of equivalences and the condition on $F$ imply that $F$ sends the map $\ord{k}\to\ord{n-\tau(0)}$ to an equivalence in $\Over{\BB}{G}$, as desired.
\end{proof}

\begin{proof}[Proof of Theorem~\ref{thm:internalGrothendieck}]
	By Lemma~\ref{lem:GrothendieckConstructionSimplices}, there is an equivalence
	\begin{equation*}
		\LFib(\Delta^{\bullet}\otimes -  )\simeq \Fun_{\BBB}(\Delta^{\bullet}\otimes -  ,\Univ)
	\end{equation*}
	of functors $\Delta^{\op}\times\BB^{\op}\to\Simp\BB^{\op}\to\CatSS$. As both $\LFib$ and $\Fun_{\BBB}(-,\Univ)$ are sheaves on $\Simp\BB$, Lemma~\ref{lem:coYoneda} implies that the above equivalences can be uniquely extended to an equivalence
	\begin{equation*}
		\LFib\simeq \Fun_{\BBB}(-,\Univ),
	\end{equation*}
	which finishes the proof.
\end{proof}
\begin{remark}
	\label{rem:BCLFib}
	Let $A\in\BB$ be an arbitrary object. By combining the Grothendieck construction, Remark~\ref{rem:BCUniverse} and Lemma~\ref{lem:BCFunctorCategory}, one obtains an equivalence $\pi_A^\ast \ILFib_{\I{C}}\simeq\ILFib_{\pi_A^\ast\I{C}}$ that is natural in $\I{C}$. By unwinding the constructions, this equivalence is explicitly given by the chain of equivalences
	\begin{equation*}
	\Fun_{\Over{\BB}{A}}(-,\ILFib_{\pi_A^\ast\I{C}})\simeq \LFib(-\times_A \pi_A^\ast\I{C})\simeq \LFib((\pi_A)_!(-)\times\I{C})\simeq \Fun_{\Over{\BB}{A}}(-, \pi_A^\ast\ILFib_{\I{C}})
	\end{equation*}
	in which the middle equivalence is induced by the equivalence $(\pi_A)_!(-\times_A\pi_A^\ast\I{C})\simeq (\pi_A)_!(-)\times \I{C}$, together with the evident observation that the forgetful functor $(\pi_A)_!\colon \Cat(\Over{\BB}{A})\to\Cat(\BB)$ gives rise to an equivalence $\LFib((\pi_A)_!(-))\simeq \LFib(-)$ of sheaves on $\Cat(\Over{\BB}{A})$.
\end{remark}

\begin{remark}
	\label{rem:restrictionGrothendieck}
	The proof of Theorem~\ref{thm:internalGrothendieck} shows that the restriction of the equivalence $\Fun_{\BB}(-, \Univ)\simeq \LFib$ along the inclusion $\BB\hookrightarrow\Cat(\BB)$ recovers the equivalence
	\begin{equation*}
		\Fun_{\BB}(-,\Univ)\simeq \Over{\BB}{-}
	\end{equation*}
	of $\CatS$-valued sheaves on $\BB$.
\end{remark}

 Let $p \colon P\to C$ be a map between simplicial objects in  $\BBB$. We will say that $p$ is \emph{small} if for every map $D\to C$ in $\Simp\BBB$ in which $\D$ is small (i.e.\ contained in $\Simp\BB$), the pullback $p^\ast D=P\times_{C} D$ is small as well. 
The collection of small maps defines a cartesian subfibration of the codomain fibration $\Fun(\Delta^1,\Simp\BBB)\to\Simp\BBB$. We therefore obtain a subpresheaf $\LFib^{\bU}_{\BBB}\into\LFib_{\BBB}$ of the sheaf of left fibrations on $\Simp\BBB$ that is spanned by the small left fibrations.
\begin{proposition}
	\label{prop:characterisationSmallFibrations}
	Let $p\colon P\to C$ be a left fibration between simplicial objects in $\BBB$. Then $p$ is small if and only if for all maps $c\colon A\to C$ with $A\in \BB$ the fibre $P\vert_c= P\times_{C} A$ is contained in $\Over{\BB}{A}$.
\end{proposition}
\begin{proof}
	The condition is clearly necessary. For the converse direction, it suffices to show that if $p\colon P\to C$ is a left fibration in $\Simp\BBB$ such that $C$ is small and such that the fibre $P\vert_c$ is small for every map $c\colon A\to C$ with $A\in\BB$, the simplicial object $P$ is small as well. To see this, note that since $p$ is a left fibration and $C$ is small, it suffices to show that $P_0$ is small. But this follows from the fact that $P_0$ arises as the fibre of $p$ over the map $C_0\to C$, see Corollary~\ref{cor:conservativeCore}.
\end{proof}
Let $\Univ[\BBB]$ denote the universe for large $\BB$-groupoids, i.e.\ the very large $\BB$-category that corresponds to the sheaf $\Over{\BBB}{-}$ on $\BBB$.  By the discussion in \S~\ref{sec:functoriality}, the inclusions $\Over{\BB}{A}\into\Over{\BBB}{A}$ for $A\in\BB$ define an embedding of presheaves $\Over{\BB}{-}\into\Over{\BBB}{-}$ on $\BB$. Since moreover restriction along the inclusion $\BB\into\BBB$ defines an equivalence
\begin{equation*}
	\begin{tikzcd}
	\Shv_{\CatSSS}(\BBB)\simeq \Shv_{\CatSSS}(\BB)
	\end{tikzcd}
\end{equation*}
(see the argument in Remark~\ref{rem:transitivityUniverseEnlargement}), we obtain a fully faithful functor $\Univ[\BB]\into\Univ[\BBB]$ in $\Cat(\BBBB)$. Explicitly, an object $A\to \Univ[\BBB]$ in context $A\in\BB$ that corresponds to a map $P\to A$ in $\BBB$ is contained in $\Univ[\BB]$ precisely if $B$ is small. By combining Remark~\ref{rem:restrictionGrothendieck} with Proposition~\ref{prop:characterisationSmallFibrations}, we thus obtain:
\begin{corollary}
	\label{cor:GrothendieckConstructionLarge}
	For any large $\BB$-category $\I{C}$, the subpresheaf $\LFib^{\bU}(-\times\I{C})\into {\LFib}(-\times\I{C})$ defines a large $\BB$-category $\ILFib^{\bU}_{\I{C}}$ in $\BB$, and the restriction of the Grothendieck construction $\ILFib_{\I{C}}\simeq \iFun{\I{C}}{\Univ[\BBB]}$ in $\BBB$ along the fully faithful functor $\iFun{\I{C}}{\Univ[\BB]}\into\iFun{\I{C}}{\Univ[\BBB]}$ gives rise to an equivalence
	\begin{equation*}
		\iFun{\I{C}}{\Univ}\simeq \ILFib^{\bU}_{\I{C}}
	\end{equation*}
	in $\Cat(\BBB)$.\qed
\end{corollary}

\begin{remark}
	\label{rem:GrothendieckConstructionDual}
	The equivalence $(-)^\op\colon\Simp\BB\simeq\Simp\BB$ that we constructed in the beginning of \S~\ref{sec:leftFibrations} induces a commutative square
	\begin{equation*}
	\begin{tikzcd}
	\RFib\arrow[r, "(-)^\op"]\arrow[d] & \LFib\arrow[d]\\
	\Simp\BB\arrow[r, "(-)^\op"] & \Simp\BB
	\end{tikzcd}
	\end{equation*}
	and therefore by Theorem~\ref{thm:internalGrothendieck} an equivalence
	\begin{equation*}
	\RFib\simeq\Fun_{\BB}((-)^\op,\Univ)
	\end{equation*}
	of $\CatSS$-valued sheaves on $\Simp\BB$. Since the diagonal embedding $\BB\into\Simp\BB$ commutes with taking opposite simplicial objects in $\BB$ (owing to $\op\colon\Delta\simeq\Delta$ preserving the final object $\ord{0}$), we thus obtain an equivalence
	\begin{equation*}
	\IRFib_{\I{C}}\simeq\iFun{\I{C}^\op}{\Univ}
	\end{equation*}
	(where the large $\BB$-category $\IRFib_{\I{C}}$ is given by the $\CatSS$-valued sheaf $\RFib(\I{C}\times -$)) that is natural in $\I{C}\in\Cat(\BB)$. Similarly, one also obtains an equivalence $\IRFib_{\I{C}}^{\bU}\simeq\iFun{\I{C}^{\op}}{\Univ}$ for every \emph{large} $\BB$-category $\I{C}$.
\end{remark}

\subsection{The universal left fibration}
\label{sec:universalLeftFibration}
By Corollary~\ref{cor:GrothendieckConstructionLarge}, the identity $\id_{\Univ}\colon\Univ\simeq\Univ$ determines a small left fibration $\phi\colon\UnivHat\to\Univ$ of large $\BB$-categories.
\begin{definition}
	The left fibration $\phi\colon\UnivHat\to\Univ$ is referred to as the the  \emph{universal left fibration}.
\end{definition}
\begin{remark}
	Yoneda's lemma for $\infty$-categories implies that the equivalence
	\begin{equation*}
		\map{\Cat(\BB)}(-,\Univ)\simeq 	(\LFib^{\bU})^{\simeq}
	\end{equation*}
	that is induced by the equivalence in Corollary~\ref{cor:GrothendieckConstructionLarge} on the underlying $\SSS$-valued sheaves is induced by assigning to each functor $f\colon \I{C}\to \Univ$ the left fibration $\int f\to\I{C}$ that is determined by the pullback square
	\begin{equation*}
		\begin{tikzcd}
			\int f\arrow[d]\arrow[r] & \UnivHat\arrow[d, "\phi"]\\
			\I{C}\arrow[r, "f"] & \Univ.
		\end{tikzcd}
	\end{equation*}
	We say that the left fibration $\int  f\to \I{C}$ is \emph{classified} by $f$. Conversely, given a small left fibration $p\colon\I{P}\to \I{C}$ of large $\BB$-categories, the functor $\I{C}\to\Univ$ that classifies $p$ acts by carrying an object $c\colon A\to\I{P}$ to the $\Over{\BB}{A}$-groupoid $\I{P}\vert_c$ that is determined by the cartesian square
	\begin{equation*}
		\begin{tikzcd}
			\I{P}\vert_c\arrow[d]\arrow[r] & \I{P}\arrow[d]\\
			A\arrow[r, "c"] & \I{C}.
		\end{tikzcd}
	\end{equation*}
\end{remark}

The main goal of this section is to prove that the universal left fibration admits the following explicit description:
\begin{proposition}
	\label{prop:explicitDescriptionUniversalLeftFibration}
	The global section $1_{\Univ}\colon 1\to \Univ$ that is determined by the final object $1_{\BB}\in \BB$ defines a final object in the universe $\Univ$, and there is an equivalence $\Under{\Univ}{1_{\Univ}}\simeq \UnivHat$ that fits into the commutative diagram
	\begin{equation*}
		\begin{tikzcd}
			\Under{\Univ}{1_{\Univ}}\arrow[r, "\simeq"] \arrow[dr, "(\pi_{1_{\Univ}})_!"']& \UnivHat\arrow[d, "\phi"] \\
			& \Univ.
		\end{tikzcd}
	\end{equation*}
\end{proposition}
The remainder of this section is devoted to the proof of Proposition~\ref{prop:explicitDescriptionUniversalLeftFibration}. We begin with the following criterion how to recognise initial objects:
\begin{lemma}
	\label{lem:criterioninitialObject}
	Let $\I{C}$ be a $\BB$-category, let $c\colon 1\to \I{C}$ be an object and let $\epsilon\colon\Delta^1\otimes\I{C}\to \I{C}$ be a map such that
	\begin{enumerate}
		\item the composite $\epsilon d^1\colon \I{C}\to \I{C}$ is equivalent to the constant map with value $c$, i.e.\ to the composite $\I{C}\to  1\to \I{C}$ in which the second arrow is given by $c$;
		\item the composite $\epsilon d^0\colon \I{C}\to \I{C}$ is equivalent to the identity;
		\item the map $\epsilon \circ(\id\otimes c)\colon \Delta^1\to \I{C}$ is equivalent to the identity map $\id_c$.
	\end{enumerate}
	Then $c$ is an initial object.
\end{lemma}
\begin{proof}
	Let us abuse notation and denote by $c$ the constant map $\I{C}\to \I{C}$ with value $c$. The map $\epsilon$ can be regarded as an object $1\to \iFun{\I{C}}{\I{C}}^{\Delta^1}$ that is sent to the constant map $c$ by the projection $d_1\colon \iFun{\I{C}}{\I{C}}^{\Delta^1}\to \iFun{\I{C}}{\I{C}}$. Therefore the map $\epsilon$ defines an object in the slice $\BB$-category $\Under{{\iFun{\I{C}}{\I{C}}}}{c}$. Observe, moreover, that the commutative diagram
	\begin{equation*}
		\begin{tikzcd}
			\Under{\iFun{\I{C}}{\I{C}}}{c}\arrow[r]\arrow[d] & \iFun{\I{C}}{\I{C}}^{\Delta^1}\arrow[r, "\simeq"] \arrow[d]&\iFun{\I{C}}{\I{C}^{\Delta^1}}\arrow[d]\\
			1\arrow[r, "c"] & {\iFun{\I{C}}{\I{C}}}\arrow[r, "\id"] & {\iFun{\I{C}}{\I{C}}}
		\end{tikzcd}
	\end{equation*}
	in which both squares are cartesian gives rise to an equivalence
	\begin{equation*}
		\Under{\iFun{\I{C}}{\I{C}}}{c}\simeq \iFun{\I{C}}{\Under{\I{C}}{c}}.
	\end{equation*}
	With respect to this equivalence, the map $\epsilon$ corresponds to a section
	$\I{C}\to \Under{\I{C}}{c}$ of the functor $(\pi_c)_!\colon \Under{\I{C}}{c}\to \I{C}$ that sends the object $c\colon 1\to \I{C}$ to $\id_c\colon 1\to \Under{\I{C}}{c}$. By Lemma~\ref{lem:criterioninitiality}, this implies that $c$ is initial.
\end{proof}

\begin{proposition}
	\label{prop:YonedaCoreLemmaB}
	Let $g\colon\I{D}\to \I{C}$ be a functor between large $\BB$-categories and let $pi\colon \I{D}\to \I{P}\to\I{C}$ be its factorisation into a final map and a right fibration. Suppose that the right fibration $p$ is small, and let $f\colon \I{C}^{\op}\to \Univ$ be the associated functor. Consider the left fibration $\I{\pi}$ that is defined by the cartesian square
	\begin{equation*}
		\begin{tikzcd}
			\I{Z}\arrow[d, "\I{\pi}"]\arrow[r]& {\IPSh_{\UnivHat}(\I{D})}\arrow[d]\\
			{\IPSh_{\Univ}(\I{C})} \arrow[r, "g^\ast"] & {\IPSh_{\Univ}(\I{D})}.
		\end{tikzcd}
	\end{equation*}
	Then there is an initial object $z\colon 1\to \I{Z}$ whose image along $\I{\pi}$ is $f$.
\end{proposition}
\begin{proof}
	Let us denote by $\phi\colon \UnivHat\to\Univ$ the universal left fibration. Since $i\colon \I{D}^{\op}\to \I{P}^{\op}$ is initial and $\UnivHat\to \Univ$ is a left fibration, the pullback square in the statement of the lemma decomposes into a pasting of cartesian squares
	\begin{equation*}
		\begin{tikzcd}
			\I{Z}\arrow[r] \arrow[d, "\I{\pi}"]& \IPSh_{\UnivHat}(\I{P})\arrow[d, "\phi_\ast"]\arrow[r] &{\IPSh_{\UnivHat}(\I{D})} \arrow[d, "\phi_\ast"]\\
			\IPSh_{\Univ}(\I{C})\arrow[r, "p^\ast"]&\IPSh_{\Univ}(\I{P})\arrow[r, "i^\ast"] & \IPSh_{\Univ}(\I{D}).
		\end{tikzcd}
	\end{equation*}
	By applying the Yoneda embedding $\Cat(\BBB)\into\PSh_{\SSS}(\Cat(\BBB))$ to the left square, we obtain a pullback square
	\begin{equation*}
		\begin{tikzcd}[column sep=large]
			\map{\Cat(\BBB)}(-, \I{Z})\arrow[d]\arrow[r] & \map{\Cat(\BBB)}(-\times\I{P}^{\op}, \UnivHat)\arrow[d, "\phi_\ast"]\\
			\map{\Cat(\BBB)}(-\times\I{C}^{\op}, \Univ)\arrow[r, "(\id\times p)^\ast"] & \map{\Cat(\BBB)}(-\times\I{P}^{\op}, \Univ)
		\end{tikzcd}
	\end{equation*}
	of $\SSS$-valued presheaves on $\Cat(\BBB)$. By using Lemma~\ref{lem:mappingGroupoidFunctorCategory} in the case where $\CC=\Delta^1$ and $\DD=\Cat(\BBB)$, we thus obtain an equivalence
	\begin{equation*}
	\map{\Cat(\BBB)}(-,\I{Z})\simeq \map{\LFib^{\bU}}(s_0(-)\times p, \phi).
	\end{equation*}
	As a consequence, the cartesian square
	\begin{equation*}
		\begin{tikzcd}
			\I{P}^\op\arrow[r]\arrow[d, "p"] & \UnivHat\arrow[d]\\
			\I{C}^{\op}\arrow[r, "f"] & \Univ
		\end{tikzcd}
	\end{equation*}
	gives rise to an object $z\colon 1\to \I{Z}$ whose image along $\I{\pi}$ is $f$. 
	
	We still need to show that $z$ is initial.
	To that end, note that by the above equivalence of presheaves on $\Cat(\BBB)$, the identity on $\I{Z}$ corresponds to a commutative square
	\begin{equation*}
	\begin{tikzcd}
	\I{Z}\times\I{P}^\op\arrow[d, "\id\times p"]\arrow[r] & \UnivHat\arrow[d]\\
	\I{Z}\times\I{C}^{\op}\arrow[r, "k"] & \Univ,
	\end{tikzcd}
	\end{equation*}
	and therefore gives rise to a map $\I{Z}\times\I{P}^{\op}\to \int k$ of left fibrations over $\I{Z}\times\I{C}^{\op}$. Using Corollary~\ref{cor:GrothendieckConstructionLarge}, this map corresponds to a morphism $\alpha\colon \pr_1^\ast(f)\to k$ in $\iFun{\I{Z}\times\I{C}^{\op}}{\Univ}$, where $\pr_1\colon \I{Z}\times\I{C}^\op\to\I{C}^\op$ denotes the projection. Let us consider the evident commutative square
	\begin{equation*}
		\begin{tikzcd}
			\pr_1^\ast(f)\arrow[r, "\id"] \arrow[d, "\id"]& \pr_1^\ast(f)\arrow[d, "\alpha"]\\
			\pr_1^\ast(f)\arrow[r, "\alpha"] & k
		\end{tikzcd}
	\end{equation*}
	as a map $\tau\colon \id\to\alpha$ in the $\BB$-category $\iFun{(\Delta^1\otimes\I{Z})\times\I{C}^{\op})}{\Univ}$. By again using Corollary~\ref{cor:GrothendieckConstructionLarge}, this map corresponds to a morphism $(\Delta^1\otimes \I{Z})\times\I{P}^{\op}\to \int \alpha$ of left fibrations over $(\Delta^1\otimes \I{Z})\times\I{C}^{\op}$ and therefore by the above equivalence of presheaves on $\Cat(\BBB)$ to a map $\epsilon\colon\Delta^1\otimes \I{Z}\to \I{Z}$. 
	By functoriality, the map $\I{Z}\to \I{Z}$ that is induced by restricting $\epsilon$ along the inclusion $d^1\colon \I{Z}\to \Delta^1\otimes\I{Z}$ corresponds to the outer square in the commutative diagram
	\begin{equation*}
		\begin{tikzcd}
			\I{Z}\times\I{P}^{\op}\arrow[r, "\pr_1"]\arrow[d, "\id\times p"] & \I{P}^{\op}\arrow[r]\arrow[d, "p"] & \UnivHat\arrow[d]\\
			\I{Z}\times\I{C}^{\op}\arrow[r, "\pr_1"]& \I{C}^{\op}\arrow[r, "f"] & \Univ,
		\end{tikzcd}
	\end{equation*}
	hence this map is equivalent to the constant functor $z\colon\I{Z}\to 1 \to \I{Z}$. Precomposing the map $\Delta^1\otimes\I{Z}\to \I{Z}$ with the inclusion $d^0\colon \I{Z}\to \Delta^1\otimes\I{Z}$, on the other hand, produces the identity on $\I{Z}$. As moreover the restriction of the map $\Delta^1\otimes\I{Z}\to \I{Z}$ along $z\colon 1\to \I{Z}$ recovers the identity on $z$, Lemma~\ref{lem:criterioninitialObject} implies that $z$ is initial.
\end{proof}
Proposition~\ref{prop:YonedaCoreLemmaB} has the following immediate consequence:
\begin{corollary}
	\label{cor:PrecompositionRepresentable}
	Let $g\colon\I{D}\to \I{C}$ be a functor between large $\BB$-categories and let $pi\colon \I{D}\to \I{P}\to\I{C}$ be its factorisation into a final map and a right fibration. Suppose that the right fibration $p$ is small, and let $f\colon \I{C}^{\op}\to \Univ$ be the associated functor. Then there is a cartesian square
	\begin{equation*}
		\begin{tikzcd}
			\Under{{\IPSh_{\Univ}(\I{C})} }{f}\arrow[d]\arrow[r]& {\IPSh_{\UnivHat}(\I{D})}\arrow[d]\\
			{\IPSh_{\Univ}(\I{C})} \arrow[r, "g^\ast"] & {\IPSh_{\Univ}(\I{D})}.
		\end{tikzcd}
	\end{equation*}
\end{corollary}
\begin{proof}
	By Proposition~\ref{prop:YonedaCoreLemmaB}, the $\BB$-category $\I{Z}$ that is defined by the cartesian square
	\begin{equation*}
		\begin{tikzcd}
			\I{Z}\arrow[d]\arrow[r]& {\IPSh_{\UnivHat}(\I{D})}\arrow[d]\\
			{\IPSh_{\Univ}(\I{C})} \arrow[r, "g^\ast"] & {\IPSh_{\Univ}(\I{D})}.
		\end{tikzcd}
	\end{equation*}
	admits an initial object $0_{\I{Z}}\colon 1\to \I{Z}$ whose image in $\IPSh_{\Univ}(\I{C})$ is $f$. By Corollary~\ref{cor:factorisationInternalObject} we obtain a commutative square
	\begin{equation*}
		\begin{tikzcd}
			1\arrow[d, "f"]\arrow[r, "0_{\I{Z}}"] & \I{Z}\arrow[d]\\
			\Under{\IPSh_{\Univ}(\I{C})}{f}\arrow[r, "(\pi_f)_!"] & \IPSh_{\Univ}(\I{C})
		\end{tikzcd}
	\end{equation*}
	in which the two maps starting in the upper left corner are initial and the two maps ending in the lower right corner are left fibrations. When regarded as a lifting problem, the above square thus admits a unique filler $\Under{\IPSh_{\Univ}(\I{C})}{f}\to \I{Z}$ that is both initial and a left fibration and therefore an equivalence.
\end{proof}

\begin{proof}[{Proof of Proposition~\ref{prop:explicitDescriptionUniversalLeftFibration}}]
	Applying Corollary~\ref{cor:PrecompositionRepresentable} to $\I{D}\simeq \I{C}\simeq 1$, the object $1_{\Univ}\colon 1\to \Univ$ gives rise to a cartesian square
	\begin{equation*}
		\begin{tikzcd}
			\Under{\Univ}{1_{\Univ}}\arrow[d]\arrow[r]& \UnivHat\arrow[d]\\
			\Univ \arrow[r, "\id"] & \Univ,
		\end{tikzcd}
	\end{equation*}
	which implies that the upper horizontal map must be an equivalence. Moreover, if $g\colon A\to \Univ$ is an object that classifies a $\Over{\BB}{A}$-groupoid $P\to A$, the equivalence $\map{\Univ}(g, \pi_A^\ast1_{\Univ})\simeq \Over{\iFun{P}{A}}{A}$ from Proposition~\ref{prop:mappingObjectsInternalUniverse} shows together with the dual version of Proposition~\ref{prop:characterisationinitialObject} that $1_{\Univ}$ is a final object of $\Univ$ (as $\Over{\iFun{P}{A}}{A}$ is final in $\Over{\BB}{A}$).
\end{proof}
\begin{remark}
	Proposition~\ref{prop:explicitDescriptionUniversalLeftFibration} shows that the $\CatSS$-valued sheaf that corresponds to $\UnivHat$ is given by the assignment
	\begin{equation*}
		A\mapsto {\BB}_{A\sslash A}
	\end{equation*}
	in which the right-hand side denotes the $\infty$-category of pointed objects in $\Over{\BB}{A}$.
\end{remark}

\begin{corollary}
	\label{cor:GrothendieckConstructionExternalInternal}
	Let $\I{C}$ be a (large) $\BB$-category and let $p\colon \I{P}\to\I{C}$ be a small left fibration that is classified by a functor $f\colon\I{C}\to\Univ$. Then $\Gamma(p)$ is classified by the functor $\Gamma\circ\Gamma(f)\colon\Gamma(\I{C})\to\BB\to\SS$. 
\end{corollary}
\begin{proof}
	By Proposition~\ref{prop:explicitDescriptionUniversalLeftFibration}, there is a commutative diagram
	\begin{equation*}
		\begin{tikzcd}
		\Gamma(\I{P})\arrow[r]\arrow[d, "\Gamma(p)"] & \Under{\BB}{1}\arrow[d] \arrow[r] & \Under{\SS}{1}\arrow[d]\\
		\Gamma(\I{C})\arrow[r, "\Gamma(f)"] & \BB\arrow[r, "\Gamma"] & \SS
		\end{tikzcd}
	\end{equation*}
	in which both squares are cartesian and in which the left square arises from applying the global sections functor $\Gamma$ to the cartesian square in $\Cat(\BBB)$ that exhibits $p$ as the Grothendieck construction of $f$.
\end{proof}

\subsection{Yoneda's lemma}
\label{sec:YonedaLemma} Theorem~\ref{thm:internalGrothendieck} may be used to construct a functorial version of the mapping $\BB$-groupoid construction for $\BB$-categories.

\begin{definition}
	A \emph{locally small} $\BB$-category is a large $\BB$-category $\I{C}$ for which the left fibration $\Tw(\I{C})\to \I{C}^{\op}\times\I{C}$ is small.
\end{definition}
As a consequence of Proposition~\ref{prop:twistedArrowLeftFibration}, the fibre of the functor $\Tw(\I{C})\to \I{C}^{\op}\times\I{C}$ over any pair of objects $(c,d)\colon A\to \I{C}^{\op}\times\I{C}$ in context $A\in \BB$ is a large $\BB$-groupoid, hence the fibre can be computed as the fibre of the induced map of core $\BB$-groupoids, which is simply the pullback
\begin{equation*}
	\begin{tikzcd}
		\map{\I{C}}(c,d)\arrow[r]\arrow[d] & \I{C}_1\arrow[d]\\
		A\arrow[r, "{(c,d)}"] & \I{C}_0\times \I{C}_0.
	\end{tikzcd}
\end{equation*}
Using Proposition~\ref{prop:characterisationSmallFibrations}, we may therefore deduce:
\begin{proposition}
	\label{prop:locallySmallMappingGroupoid}
	A $\BB$-category $\I{C}$ is locally small if and only if for any pair of objects $(c,d)$ in $\I{C}\times\I{C}$ in context $A\in\BB$ the mapping $\BB$-groupoid $\map{\I{C}}(c,d)$ is contained in $\BB$.\qed
\end{proposition}
\begin{example}
	\label{ex:universeLocallySmall}
	The universe $\Univ$ for $\BB$-groupoids is locally small as its mapping groupoids can be identified with the internal mapping objects in the slice $\infty$-topoi over $\BB$ (cf.\ Proposition~\ref{prop:mappingObjectsInternalUniverse}).
\end{example}
\begin{proposition}
	\label{prop:locallySmallVsSmall}
	A locally small $\BB$-category $\I{C}$ is small if and only if $\I{C}^{\core}$ is a small $\BB$-groupoid.
\end{proposition}
\begin{proof}
	The condition is clearly necessary, so let us assume that $\I{C}$ is locally small and that $\I{C}^{\core}$ is a small $\BB$-groupoid. By making use of the Segal conditions, we need only show that $\I{C}_1$ is contained in $\BB$. Since $\I{C}_0\times\I{C}_0$ is an object of $\BB$, this follows from the observation that $\I{C}_1$ is recovered as the mapping $\BB$-groupoid of the pair $(\pr_0,\pr_1)\colon\I{C}_0\times\I{C}_0\to\I{C}_0\times\I{C}_0$.
\end{proof}
\begin{lemma}
	\label{lem:essentialImageSmall}
	Let  $f\colon \I{C}\to\I{D}$ be a functor such that $\I{C}$ is small and $\I{D}$ is locally small. Then the essential image $\I{E}$ of $f$ is small.
\end{lemma}
\begin{proof}
	Being a full subcategory of $\I{D}$, the $\BB$-category $\I{E}$ is locally small (using Proposition~\ref{prop:characterisationInternalFullyFaithfulMappingGroupoids}), hence Proposition~\ref{prop:locallySmallVsSmall} implies that $\I{E}$ is small whenever $\I{E}^{\core}$ is a small $\BB$-groupoid. 
	By Corollary~\ref{cor:esoCoverCore}, $\I{E}_0$ is the image of the map $f_0\colon \I{C}_0\to\I{D}_0$, hence one finds $\I{E}_0\simeq \colim_{n} \I{C}_0\times_{\I{D}_0}\cdots\times_{\I{D}_0}\I{C}_0$. As $\Delta$ is a small $1$-category, it suffices to show that for each $n\geq 0$ the $(n+1)$-fold fibre product $\I{C}_0\times_{\I{D}_0}\cdots\times_{\I{D}_0}\I{C}_0\in\BBB$ is contained in $\BB$. We may identify this object as the pullback of the map $\I{C}_0^{n+1}\to\I{D}_0^{n+1}$ along the diagonal $\I{D}_0\to\I{D}_0^{n+1}$. Since the map $\I{D}_0\to\I{D}_n$ is a monomorphism in $\BBB$, we obtain a monomorphism
	\begin{equation*}
	\I{C}_0\times_{\I{D}_0}\cdots\times_{\I{D}_0}\I{C}_0\into \map{\I{D}}(\pr_0^\ast f_0,\dots,\pr_n^\ast f_n)
	\end{equation*}
	where $\pr_i\colon \I{C}_0^{n+1}\to \I{C}_0$ denotes the $i$th projection. Since $\I{D}$ is by assumption locally small, the codomain of this map is contained in $\BB$, hence the result follows.
\end{proof}
\begin{proposition}
	\label{prop:functorCategoryLocallySmall}
	For any small $\BB$-category $\I{C}$ and any locally small $\BB$-category $\I{D}$, the functor $\BB$-category $\iFun{\I{C}}{\I{D}}$ is locally small as well.
\end{proposition}
\begin{proof}
	Using Proposition~\ref{prop:locallySmallMappingGroupoid}, we need to show that for any pair $(f,g)\colon A\to \iFun{\I{C}}{\I{D}}\times\iFun{\I{C}}{\I{D}}$ the mapping $\BB$-groupoid $\map{\iFun{\I{C}}{\I{D}}}(f,g)$ is small. Let $\I{E}$ be the essential image of $f$ when viewed as a functor $A\times \I{C}\to\I{D}$, and let $\I{E}^\prime$ be the essential image of $g$ when viewed as a functor $A\times\I{C}\to\I{D}$. Then both $\I{E}$ and $\I{E}^\prime$ are small $\BB$-categories by Lemma~\ref{lem:essentialImageSmall}. By the same argument, the image $E^{\prime\prime}$ of the functor $E\sqcup E^{\prime}\to \I{D}$ that is induced by the two inclusions is a small $\BB$-category that embeds fully faithfully into $\I{D}$. By construction, the map $(f,g)\colon A\to \iFun{\I{C}}{\I{D}}\times\iFun{\I{C}}{\I{D}}$ factors through the full embedding $\iFun{\I{C}}{\I{E}^{\prime\prime}}\times\iFun{\I{C}}{\I{E}^{\prime\prime}}\into \iFun{\I{C}}{\I{D}}\times\iFun{\I{C}}{\I{D}}$. As $\iFun{\I{C}}{\I{E}^{\prime\prime}}$ is small, the (a priori large) mapping $\BB$-groupoid $\map{\iFun{\I{C}}{\I{D}}}(f,g)\simeq\map{\iFun{\I{C}}{\I{E}^{\prime\prime}}}(f,g)$ must be contained in $\BB$.
\end{proof}

If $\I{C}$ is an arbitrary locally small $\BB$-category, applying Theorem~\ref{thm:internalGrothendieck} to the left fibration $\Tw(\I{C})\to\I{C}^{\op}\times\I{C}$ gives rise the \emph{mapping $\BB$-groupoid functor}
\begin{equation*}
	\map{\I{C}}(-,-)\colon \I{C}^{\op}\times \I{C}\to \Univ.
\end{equation*}
By transposing across the adjunction $\I{C}^{\op}\times -\dashv \iFun{\I{C}^{\op}}{-}$, this functor determines the \emph{Yoneda embedding}
\begin{equation*}
	h\colon \I{C}\to \IPSh_{\Univ}(\I{C})=\iFun{\I{C}^{\op}}{\Univ}.
\end{equation*}
\begin{remark}
	\label{rem:mappingFunctorExplicit}
	Let $\I{C}$ be a locally small $\BB$-category, let $A\in\BB$ be an arbitrary object and let $c$ be an object in $\I{C}$ in context $A$. Let us furthermore fix a map $f\colon d\to e$ in $\I{C}$ in context $A$. Applying the mapping $\BB$-groupoid functor $\map{\I{C}}(-,-)$ to the pair $(\id_{c}, f)$  of maps in $\I{C}$ then results in a morphism $f_!\colon\map{\I{C}}(c,d)\to\map{\I{C}}(c,e)$ in $\Over{\BB}{A}$. Explicitly, this map is given by applying the chain of equivalences $\LFib^{\bU}(\Delta^1\otimes A)\simeq\Fun_{\BB}(\Delta^1\otimes A, \Univ)\simeq \Fun_{\SS}(\Delta^1,\Over{\BB}{A})$ to the left fibration $\I{P}\to \Delta^1\otimes A$ that arises as the pullback
	\begin{equation*}
		\begin{tikzcd}
			\I{P}\arrow[d]\arrow[r] & \Under{\I{C}}{c}\arrow[d]\\
			\Delta^1\otimes A\arrow[r, "{(\pr_1, f)}"] & A\times\I{C} 
		\end{tikzcd}
	\end{equation*}
	in which $\pr_1\colon \Delta^1\otimes A\to A$ denotes the projection. By construction of the equivalence of $\infty$-categories $\LFib^{\bU}(\Delta^1\otimes A)\simeq\Fun_{\SS}(\Delta^1,\Over{\BB}{A})$, one now sees that the map $f_!$ fits into the commutative diagram
	\begin{equation*}
		\begin{tikzcd}[column sep={4em,between origins}, row sep={3em,between origins}]
			& \map{\I{C}}(c,d)\arrow[rrrr, "f_!", bend left]\arrow[from=rr, "\simeq"']\arrow[dd] \arrow[dl]&& Z\arrow[rr]\arrow[dd]\arrow[dl] && \map{\I{C}}(c,e)\arrow[dd]\arrow[dl]\\
			(\Under{\I{C}}{c})_0\arrow[from=rr, "d_1"', crossing over, near start]\arrow[dd] && (\Under{\I{C}}{c})_1\arrow[rr, "d_0", crossing over, near end] \arrow[dd]&& (\Under{\I{C}}{c})_0\arrow[dd] &\\
			& A\arrow[from=rr, "\id"', near start]\arrow[dl, "d"] && A\arrow[rr, "\id", near end]\arrow[dl, "f"] && A\arrow[dl, "e"] \\
			\I{C}_0\arrow[from=rr, "d_1"'] && \I{C}_1\arrow[rr, "d_0"] && \I{C}_0. &
		\end{tikzcd}
	\end{equation*}
	Let $g\colon c\to d$ be an arbitrary map in $\I{C}$ in context $A$, and let $\sigma\colon\Delta^2\otimes A\to \I{C}$ be the $2$-morphism that is encoded by the commutative diagram
	\begin{equation*}
		\begin{tikzcd}
			c\arrow[r, "g"] \arrow[dr, "fg"'] & d\arrow[d, "f"]\\
			& e.
		\end{tikzcd}
	\end{equation*}
	Let furthermore $\tau\colon \Delta^1\otimes A\to\I{C}$ be the $2$-morphism that is determined by the commutative diagram
		\begin{equation*}
		\begin{tikzcd}
			c\arrow[r, "\id"] \arrow[dr, "fg"'] & c\arrow[d, "fg"]\\
			& e.
		\end{tikzcd}
	\end{equation*}
	On account of the decomposition $\Delta^1\times\Delta^1\simeq\Delta^2\sqcup_{\Delta^1}\Delta^2$, the pair $(\tau, \sigma)$ gives rise to a map $(\Delta^1\times\Delta^1)\otimes A\to\I{C}$ that by construction defines a section $A\to Z$. Furthermore, the composition $A\to Z\simeq \map{\I{C}}(c,d)$ recovers $g$ and the composition $A\to Z\to\map{\I{C}}(c,e)$ recovers $fg$. Therefore, the map $f_!$ acts by sending a map $g\colon c\to d$ to the composition $fg\colon c\to d\to e$. By a dual argument, the map $f^\ast\colon\map{\I{C}}(e,c)\to\map{\I{C}}(d,c)$ that is determined by applying the mapping $\BB$-groupoid functor to the pair $(f,\id_{c})$ sends a map $g\colon e\to c$ to the composition $gf\colon d\to e\to c$.
\end{remark}
If $\I{C}$ is a $\BB$-category, let us denote by $\ev\colon \I{C}^{\op}\times \IPSh_{\Univ}(\I{C})\to\Univ$ the evaluation functor, i.e.\  the counit of the adjunction $\I{C}^{\op}\times- \dashv \iFun{\I{C}^{\op}}{-}$.
\begin{theorem}[Yoneda's lemma]
	\label{thm:YonedaLemma}
	For any $\BB$-category $\I{C}$, there is a commutative diagram
	\begin{equation*}
		\begin{tikzcd}
			\I{C}^{\op}\times \IPSh_{\Univ}(\I{C})\arrow[dr, "\ev"'] 
			\arrow[r, "h\times \id"] & {\IPSh_{\Univ}(\I{C})^{\op}}\times\IPSh_{\Univ}(\I{C})\arrow[d, "{\map{\IPSh_{\Univ}(\I{C})}(-,-)}"]  \\
			& \Univ
		\end{tikzcd}
	\end{equation*}
	in $\Cat(\BBB)$.
\end{theorem}

The proof of Theorem~\ref{thm:YonedaLemma} employs a strategy that is similar to the one used by Cisinski in~\cite{cisinski2019a} for a proof of Yoneda's lemma for $\infty$-categories. We begin with the following lemma:
\begin{lemma}
	\label{lem:YonedaCoreLemmaA}
	Let
	\begin{equation*}
		\begin{tikzcd}
			\I{P}\arrow[d, "p"]\arrow[r, "f"] & \I{Q}\arrow[d, "q"]\\
			\I{C}\times\I{C}\arrow[r, "\id\times g"] & \I{C}\times \I{D}
		\end{tikzcd}
	\end{equation*}
	be a commutative diagram in $\Cat(\BB)$ such that the maps $p$ and $q$ are left fibrations, and suppose that for any object $c\colon A\to \I{C}$ the induced map $f\vert_c\colon \I{P}\vert_c\to \I{Q}\vert_c$ is initial. Then $f$ is initial.
\end{lemma}
\begin{proof}
	By Remark~\ref{rem:covariantEquivalencesExplicitly}, $q$ being a left fibration implies that it suffices to show that $f$ is a covariant equivalence over $\I{C}\times\I{D}$. Using Proposition~\ref{prop:characterisationCovariantEquivalences}, it is moreover enough to show that for any object $d\colon A\to \I{D}$ the induced map $\Over{f}{d}$ (that is obtained by pulling back $f$ along $\id\times(\pi_d)_!\colon \I{C}\times\Over{\I{D}}{d}\to\I{C}\times\I{D}$) is a covariant equivalence over $\I{C}$. In fact, if this is the case, then Proposition~\ref{prop:characterisationCovariantEquivalences} implies that for any object $c\colon A\to \I{C}$ the induced map $\Over{(\Over{f}{d})}{c}$ becomes an equivalence after applying the groupoidification functor. Now it is straightfoward to see that this map is equivalently given by the pullback $\Over{f}{(c,d)}$ of $f$ along the right fibration $\Over{(\I{C}\times\I{D})}{(c,d)}\to \I{C}\times\I{D}$ that is determined by the object $(c,d)\colon  A\to \I{C}\times\I{D}$,
	hence another application of Proposition~\ref{prop:characterisationCovariantEquivalences} implies that $f$ is a covariant equivalence over $\I{C}\times\I{D}$.
	
	Now since the projections $\I{C}\times \Over{\I{C}}{d}\to \I{C}$ and $\I{C}\times \Over{\I{D}}{d}\to \I{C}$ are smooth, the diagonal maps in the induced commutative diagram
	\begin{equation*}
		\begin{tikzcd}[column sep=small]
			\Over{\I{P}}{d}\arrow[dr] \arrow[rr, "\Over{f}{d}"] && \Over{\I{Q}}{d}\arrow[dl] \\
			& \I{C}
		\end{tikzcd}
	\end{equation*}
	are smooth (as left fibrations are smooth by the dual of Proposition~\ref{prop:rightFibrationProper}). As the induced map $\left(\Over{f}{d}\right)\vert_c$ on the fibres over $c\colon A\to \I{C}$ is a pullback of the initial functor $f\vert_c$ along a proper map, this functor must be initial as well, hence we may apply~\ref{prop:initialityFibrewiseSmooth} to deduce that $\Over{f}{d}$ is a covariant equivalence over $\I{C}$, as required.
\end{proof}
\begin{remark}
	\label{rem:YonedaCoreLemmaA}
	In the situation of Lemma~\ref{lem:YonedaCoreLemmaA}, let $\pi_A^\ast (f)\vert_c\colon \pi_A^\ast\I{P}\vert_c\to\pi_A^\ast\I{Q}$ be the functor in $\Cat(\Over{\BB}{A})$ that arises as the fibre of the map $\pi_A^\ast(f)\colon\pi_A^\ast(\I{P})\to\pi_A^\ast(\I{Q})$ over the object $c\colon 1\to\pi_A^\ast\I{C}$ that corresponds to $c\colon A\to\I{C}$ by transposition. Then $f\vert_c$ is obtained as the image of $\pi_A^\ast(f)\vert_c$ under the forgetful functor $(\pi_A)_!$, hence Remark~\ref{rem:initialObjectequivalentCondition} implies that $f\vert_c$ is initial if and only if $\pi_A^\ast(f)\vert_c$ is initial. Hence Lemma~\ref{lem:YonedaCoreLemmaA} implies that $f$ is an initial map if and only if for every object $A\in\BB$ the fibre of $\pi_A^\ast(f)$ over every \emph{global} object $c\colon 1\to\pi_A^\ast\I{C}$ is initial.
\end{remark}

\begin{lemma}
	\label{lem:BCComma}
	For any $A\in\BB$ there is a natural equivalence 
	\begin{equation*}
		\pi_A^\ast(\Comma{-}{-}{-})\simeq \Comma{\pi_A^\ast(-)}{\pi_A^\ast(-)}{\pi_A^\ast(-)}
	\end{equation*}
	of functors $\Fun(\Lambda^2_2, \Cat(\BB))\to\Fun(\Lambda^2_0,\Cat(\Over{\BB}{A}))$. Dually, there is a natural equivalence of functors
	\begin{equation*}
			\pi_A^\ast(\Cocomma{-}{-}{-})\simeq \Cocomma{\pi_A^\ast(-)}{\pi_A^\ast(-)}{\pi_A^\ast(-)}.
	\end{equation*}
\end{lemma}
\begin{proof}
	As $\pi_A^\ast$ commutes with both limits and colimits and on account of Lemma~\ref{lem:BCFunctorCategory}, this is immediate.
\end{proof}
\begin{corollary}
	\label{cor:BCUniversalFibration}
	For any $A\in\BB$ the functor $\pi_A^\ast$ carries the universal left fibration in $\BB$ to the universal left fibration in $\Over{\BB}{A}$.
\end{corollary}
\begin{proof}
	Combine Remark~\ref{rem:BCUniverse} with Lemma~\ref{lem:BCComma} and the evident fact that $\pi_A^\ast$ carries $1_{\Univ}$ to the final object $1_{\Univ[\Over{\BB}{A}]}$.
\end{proof}

\begin{lemma}
	\label{lem:BCTw}
	For any $\BB$-category $\I{C}$, there is a canonical commutative square
	\begin{equation*}
	\begin{tikzcd}
	\Tw(\pi_A^\ast\I{C})\arrow[d, "p_{\pi_A^\ast\I{C}}"] \arrow[r, "\simeq"] & \pi_A^\ast\Tw(\I{C})\arrow[d, "\pi_A^\ast(p_{\I{C}})"]\\
	\pi_A^\ast\I{C}^{\op}\times\pi_A^\ast\I{C}\arrow[r, "\simeq"] & \pi_A^\ast(\I{C}^{\op}\times\I{C})
	\end{tikzcd}
	\end{equation*}
	and therefore an equivalence of bifunctors $\pi_A^\ast\map{\I{C}}\simeq \map{\pi_A^\ast\I{C}}$.
\end{lemma}
\begin{proof}
	The first statement follows immediately from Lemma~\ref{lem:BCFunctorCategory}. The second statement follows from the first and Corollary~\ref{cor:BCUniversalFibration}.
\end{proof}

\begin{lemma}
	\label{lem:BCYoneda}
	For any $\BB$-category $\I{C}$ and any object $A\in\BB$, there is a commutative diagram
	\begin{equation*}
		\begin{tikzcd}
		\pi_A^\ast\I{C}\arrow[r, "\pi_A^\ast h_{\I{C}}"]\arrow[dr, "h_{\pi_A^\ast\I{C}}"'] & \pi_A^\ast\IPSh_{\Univ}(\I{C})\arrow[d, "\simeq"]\\
		& \IPSh_{\Univ[\Over{\BB}{A}]}(\pi_A^\ast\I{C}).
		\end{tikzcd}
	\end{equation*}
\end{lemma}
\begin{proof}
	By construction, the Yoneda embedding $h_{\I{C}}$ is the image of $\map{\I{C}}$ under the equivalence
	\begin{equation*}
	\map{\Cat(\BBB)}(\I{C}^{\op}\times\I{C},\Univ)\simeq \map{\Cat(\BBB)}(\I{C},\IPSh_{\Univ}(\I{C})).
	\end{equation*}
	Moreover, if $\epsilon\colon (\pi_A)_!\pi_A^\ast\to\id_{\Cat(\BB)}$ denotes the adjunction counit, one can construct a commutative diagram
	\begin{equation*}
	\begin{tikzcd}
		\map{\Cat(\BBB)}(\I{C}, \IPSh_{\Univ}(\I{C}))\arrow[r, "\simeq"]\arrow[d, "\pi_A^\ast"] & \map{\Cat(\BBB)}(\I{C}^{\op}\times\I{C},\Univ)\arrow[d, "(\id\times\epsilon_{\I{C}})^\ast"]\\
		\map{\Cat(\Over{\BBB}{A})}(\pi_A^\ast\I{C}, \pi_A^\ast\IPSh_{\Univ}(\I{C})) \arrow[r, "\simeq"] \arrow[d, "\simeq"]& \map{\Cat(\BBB)}(\I{C}^{\op}\times(\pi_A)_!\pi_A^\ast\I{C}, \Univ)\arrow[d, "\simeq"]\\
		\map{\Cat(\Over{\BBB}{A})}(\pi_A^\ast\I{C}, \IPSh_{\Univ[\Over{\BB}{A}]}(\pi_A^\ast\I{C}))\arrow[r, "\simeq"] & \map{\Cat(\BBB)}((\pi_A)_!(\pi_A^\ast\I{C}^{\op}\times_A\pi_A^\ast\I{C}), \Univ)
	\end{tikzcd}
	\end{equation*}
	(cf.\ Lemma~\ref{lem:BCFunctorCategory} and Remark~\ref{rem:BCUniverse}). By chasing through this diagram, it therefore suffices to show that the bifunctor $\map{\I{C}}(-, \epsilon_{\I{C}}(-))$ is equivalent to the composition
	\begin{equation*}
	\I{C}^{\op}\times (\pi_A)_!\pi_A^\ast\I{C}\xrightarrow{\simeq} (\pi_A)_!(\pi_A^\ast\I{C}^{\op}\times_A\pi_A^\ast\I{C})\xrightarrow{(\pi_A)_!\map{\pi_A^\ast\I{C}}} (\pi_A)_!\pi_A^\ast\Univ\xrightarrow{\epsilon_{\Univ}} \Univ,
	\end{equation*}
	which follows from Lemma~\ref{lem:BCTw}.
\end{proof}

\begin{lemma}
	\label{lem:BCEvaluation}
	For any two $\BB$-categories $\I{C}$ and $\I{D}$, there is a commutative diagram
	\begin{equation*}
	\begin{tikzcd}
	 \iFun{\pi_A^\ast\I{C}}{\pi_A^\ast\I{D}}\times\pi_A^\ast\I{C}\arrow[dr, "\ev_{\pi_A^\ast\I{C}}"']\arrow[r, "\simeq"] & \pi_A^\ast(\iFun{\I{C}}{\I{D}}\times\I{C})\arrow[d, "\pi_A^\ast\ev_{\I{C}}"]\\
	 & \pi_A^\ast\I{D}.
	\end{tikzcd}
	\end{equation*} 
\end{lemma}
\begin{proof}
	If $\epsilon$ denotes the counit of the adjunction $(\pi_A)_!\dashv\pi_A^\ast$ the transpose of the functor $\pi_A^\ast\ev_{\I{C}}$ is given by the composition
	\begin{equation*}
		(\pi_A)_!\pi_A^\ast(\iFun{\I{C}}{\I{D}}\times\I{C})\xrightarrow{\epsilon} \iFun{\I{C}}{\I{D}}\times\I{C}\xrightarrow{\ev_{\I{C}}} \I{D}.
	\end{equation*}
	Note that functoriality of $\epsilon$ gives rise to a commutative diagram
	\begin{equation*}
	\begin{tikzcd}
		(\pi_A)_!\pi_A^\ast(\iFun{\I{C}}{\I{D}}\times\I{C})\arrow[r, "\simeq"]\arrow[d, "\epsilon"] & (\pi_A)_!(\pi_A^\ast\iFun{\I{C}}{\I{D}}\times_A\pi_A^\ast\I{C})\arrow[d, "\simeq"]\\
		\iFun{\I{C}}{\I{D}}\times\I{C}\arrow[from=r, "\epsilon\times\id_{\I{C}}"] & (\pi_A)_!\pi_A^\ast\iFun{\I{C}}{\I{D}}\times\I{C}.
	\end{tikzcd}
	\end{equation*}
	As a consequence, by chasing the identity $\id_{\iFun{\I{C}}{\I{D}}}$ through the upper horizontal map and the right column in the commutative diagram
	\begin{equation*}
	\begin{tikzcd}
		\map{\Cat(\BB)}(\iFun{\I{C}}{\I{D}},\iFun{\I{C}}{\I{D}})\arrow[r, "\simeq"]\arrow[d, "\pi_A^\ast"] & \map{\Cat(\BB)}(\iFun{\I{C}}{\I{D}}\times\I{C},\I{D})\arrow[d, "(\epsilon\times\id)^\ast"]\\
		\map{\Cat(\Over{\BB}{A})}(\pi_A^\ast\iFun{\I{C}}{\I{D}},\pi_A^\ast\iFun{\I{C}}{\I{D}})\arrow[d, "\simeq"]\arrow[r,"\simeq"] & \map{\Cat(\BB)}((\pi_A)_!\pi_A^\ast\iFun{\I{C}}{\I{D}}\times\I{C},\I{D})\arrow[d, "\simeq"]\\
		\map{\Cat(\Over{\BB}{A})}(\pi_A^\ast\iFun{\I{C}}{\I{D}},\iFun{\pi_A^\ast\I{C}}{\pi_A^\ast\I{D}})\arrow[r, "\simeq"]\arrow[d, "\simeq"] & \map{\Cat(\BB)}((\pi_A)_!(\pi_A^\ast\iFun{\I{C}}{\I{D}}\times_A\pi_A^\ast\I{C}),\I{D})\arrow[d, "\simeq"]\\
		\map{\Cat(\Over{\BB}{A})}(\iFun{\pi_A^\ast\I{C}}{\pi_A^\ast\I{D}},\iFun{\pi_A^\ast\I{C}}{\pi_A^\ast\I{D}})\arrow[r, "\simeq"] & \map{\Cat(\BB)}((\pi_A)_!(\iFun{\pi_A^\ast\I{C}}{\pi_A^\ast\I{D}}\times_A\pi_A^\ast\I{C}),\I{D})
	\end{tikzcd}
	\end{equation*}
	(see Lemma~\ref{lem:BCFunctorCategory}), we end up with the transpose of the composition
	\begin{equation*}
	\iFun{\pi_A^\ast\I{C}}{\pi_A^\ast\I{D}}\times\pi_A^\ast\I{C}\xrightarrow{\simeq} \pi_A^\ast(\iFun{\I{C}}{\I{D}}\times\I{C})\xrightarrow{\pi_A^\ast\ev_{\I{C}}} \pi_A^\ast\I{D}.
	\end{equation*}
	On the other hand, chasing $\id_{\iFun{\I{C}}{\I{D}}}$ through the left column and the lower horizontal map in the above diagram yields the transpose of $\ev_{\pi_A^\ast\I{C}}$, hence these two maps must be equivalent.
\end{proof}

We are finally ready to prove Yoneda's lemma for $\BB$-categories:
\begin{proof}[{Proof of Theorem~\ref{thm:YonedaLemma}}]
	Let $\int \ev\to \I{C}^{\op}\times \IPSh_{\Univ}(\I{C})$ be the left fibration that classifies the evaluation functor $\ev$. Using Theorem~\ref{thm:internalGrothendieck}, it suffices to show that there is a cartesian square
	\begin{equation*}
		\begin{tikzcd}
			\int\ev\arrow[d]\arrow[r] & \Tw(\IPSh_{\Univ}(\I{C}))\arrow[d]\\
			\I{C}^{\op}\times \IPSh_{\Univ}(\I{C})\arrow[r, "h\times\id"] & {\IPSh_{\Univ}(\I{C})}^{\op}\times \IPSh_{\Univ}(\I{C}).
		\end{tikzcd}
	\end{equation*}
	Note that by definition of the evaluation functor, there is a cartesian square
	\begin{equation*}
		\begin{tikzcd}
			\Tw(\I{C})\arrow[d]\arrow[r, "f"] & \int \ev\arrow[d]\\
			\I{C}^{\op}\times\I{C}\arrow[r, "\id\times h"] &\I{C}^{\op}\times \IPSh_{\Univ}(\I{C}).
		\end{tikzcd}
	\end{equation*}
	Moreover, using functoriality of the twisted arrow $\BB$-category construction, we may construct a commutative diagram
	\begin{equation*}
		\begin{tikzcd}
			\Tw(\I{C})\arrow[d]\arrow[r, "g"]\arrow[rr, bend left, "\Tw(h)"] & \I{P}\arrow[d]\arrow[r] & \Tw({\IPSh_{\Univ}(\I{C})})\arrow[d]\\
			\I{C}^{\op}\times\I{C}\arrow[r, "\id\times h"]\arrow[rr, bend right, "h\times h"] &\I{C}^{\op}\times \IPSh_{\Univ}(\I{C}) \arrow[r, "h\times\id"] & {\IPSh_{\Univ}(\I{C})}^{\op}\times \IPSh_{\Univ}(\I{C})
		\end{tikzcd}
	\end{equation*}
	in which the right square is cartesian. As a consequence, one obtains a commutative square
	\begin{equation*}
		\begin{tikzcd}
			\Tw(\I{C})\arrow[d, "g"]\arrow[r, "f"] & \int \ev\arrow[d]\\
			\I{P}\arrow[r] & \I{C}^{\op}\times \IPSh_{\Univ}(\I{C}).
		\end{tikzcd}
	\end{equation*}
	To complete the proof, it therefore suffices to produce a lift $\I{P}\to \int \ev$ in the previous square and to show that this map is an equivalence. This is possible once we verify that the two maps $f$ and $g$ are initial.
	
	In order to show that the map $g$ is initial, note that we are in the situation of Lemma~\ref{lem:YonedaCoreLemmaA}, which means that it suffices to show that for any object $c\colon A\to \I{C}$ the induced functor $g\vert_c\colon \Tw(\I{C})\vert_c\to \I{P}\vert_c$ is initial. By construction of $\I{P}$, this map is equivalent to the map $\Tw(h)\vert_c\colon \Tw(\I{C})\vert_c\to \Tw(\IPSh_{\Univ}(\I{C}))\vert_{h(c)}$, and by using Proposition~\ref{prop:fibresTwistedArrowFibration} this map can be identified with the functor
	\begin{equation*}
		\Under{\I{C}}{c}\to \Under{{\IPSh_{\Univ}(\I{C})}}{h(c)}.
	\end{equation*}
	This map is initial as it sends the initial section $\id_c\colon A\to \Under{\I{C}}{c}$ to the initial section $\id_{{h(c)}}\colon A\to \Under{{\IPSh_{\Univ}(\I{C})}}{h(c)}$.
	
	In order to prove that the map $f\colon \Tw(\I{C})\to \int\ev$ is initial, we employ Lemma~\ref{lem:YonedaCoreLemmaA} once more to conclude that it will be sufficient to show that the map $f\vert_c$ in the induced cartesian square
	\begin{equation*}
		\begin{tikzcd}
			\Under{\I{C}}{c}\arrow[r, "f\vert_c"] \arrow[d]& \int\ev\vert_{c}\arrow[d]\\
			A\times \I{C}\arrow[r, "\id\times h"] & A\times{\IPSh_{\Univ}(\I{C})}
		\end{tikzcd}
	\end{equation*}
	is initial. Remark~\ref{rem:YonedaCoreLemmaA} implies that we may regard this square as a diagram in $\Cat(\Over{\BBB}{A})$. By combining Lemma~\ref{lem:BCTw}, Lemma~\ref{lem:BCYoneda} and Lemma~\ref{lem:BCEvaluation}, we may thus assume without loss of generality $A\simeq 1$.

	Applying Proposition~\ref{prop:YonedaCoreLemmaB} to the factorisation $1\to \Over{\I{C}}{c}\to\I{C}$ of $c$ into a final map followed by a right fibration (cf.\ Corollary~\ref{cor:factorisationInternalObject}), we obtain the cartesian square
	\begin{equation*}
		\begin{tikzcd}
			\Under{{\IPSh_{\Univ}(\I{C})}}{h(c)}\arrow[r]\arrow[d] & {\UnivHat}\arrow[d]\\
			{\IPSh_{\Univ}(\I{C})}\arrow[r, "c^\ast"] & {\Univ}.
		\end{tikzcd}
	\end{equation*}
	Observe that the functor $c^\ast$ is equivalent to the composition $\ev\circ (c\times\id)$. Hence there is an equivalence $\rho\colon \int\ev\vert_{c}\simeq \Under{\IPSh_{\Univ}(\I{C})}{h(c)}$ over $\IPSh_{\Univ}(\I{C})$. In order to show that $f\vert_c$ is initial, it
	now suffices to verify that the image of $\id_c$ along the induced map $\rho f\vert_c\colon\Under{\I{C}}{c}\to\Under{\IPSh_{\Univ}(\I{C})}{h(c)}$ is equivalent to $\id_{h(c)}$. As both objects define sections over $h(c)\colon 1\to\IPSh_{\Univ}(\I{C})$, this follows once we show that they both give rise to equivalent sections $1\rightrightarrows\UnivHat$ over $\map{\I{C}}(c,c)\colon 1\to\Univ$.
	Now the proof of Proposition~\ref{prop:YonedaCoreLemmaB} shows that the image of $\id_{h(c)}$ along $\Over{\IPSh_{\Univ}(\I{C})}{h(c)}\to\UnivHat$ is given by the composition of the two upper horizontal maps in the diagram
	\begin{equation*}
		\begin{tikzcd}
			1\arrow[r, "\id_c"] &\left(\Over{\I{C}}{c}\right)^{\op}\arrow[d]\arrow[r] & \UnivHat\arrow[d]\\
			& \I{C}^{\op}\arrow[r, "h(c)"] & \Univ
		\end{tikzcd}
	\end{equation*}
	in which the square is the pullback diagram that is induced by Proposition~\ref{prop:fibresTwistedArrowFibration}. As the image of $\rho f\vert_c(\id_c)$ in $\UnivHat$ is given by the image of $\id_c$ along the composition $\Under{\I{C}}{c}\to\Tw(\I{C})\to \UnivHat$ in which the first map is the one that is determined by the pullback square from Proposition~\ref{prop:fibresTwistedArrowFibration}, the claim now follows from the observation that \emph{both} the map $\id_c\colon 1\to \left(\Over{\I{C}}{c}\right)^{\op}\to \Tw(\I{C})$ and the map $\id_c\colon 1\to \Under{\I{C}}{c}\to \Tw(\I{C})$ correspond to the composite $s_0 c\colon 1\to \I{C}_0\to \I{C}_1$.
\end{proof}

\begin{corollary}
	\label{cor:YonedaEmbedding}
	For any $\BB$-category $\I{C}$, the Yoneda embedding $h\colon \I{C}\to \IPSh_{\Univ}(\I{C})$ is fully faithful.
\end{corollary}
\begin{proof}
	By Theorem~\ref{thm:YonedaLemma} the canonical square
	\begin{equation*}
		\begin{tikzcd}
			\Tw(\I{C})\arrow[d]\arrow[r] & \Tw(\IPSh_{\Univ}(\I{C}))\arrow[d]\\
			\I{C}^{\op}\times\I{C}\arrow[r, "h^{\op}\times h"] & \IPSh_{\Univ}(\I{C})^{\op}\times\IPSh_{\Univ}(\I{C})
		\end{tikzcd}
	\end{equation*}
	that is obtained by functoriality of the twisted arrow construction is cartesian, which proves the claim upon applying the core $\BB$-groupoid functor.
\end{proof}

\begin{corollary}
	\label{cor:functorEquivalencesObjectwise}
	Let $\I{C}$ and $\I{D}$ be $\BB$-categories and let $\alpha\colon \Delta^1\otimes A\to \iFun{\I{C}}{\I{D}}$ be a morphism in $\iFun{\I{C}}{\I{D}}$. Then $\alpha$ is an equivalence if and only if for all $c\colon B\to \I{C}$ the map $\alpha(c)\colon \Delta^1\otimes(A\times B)\to \I{D}$ is an equivalence in $\I{D}$.
\end{corollary}
\begin{proof}
	The condition is clearly necessary, so suppose that $\alpha(c)$ is an equivalence for every object $c\colon B\to \I{C}$. By replacing $\BB$ with $\Over{\BB}{A}$, we may assume without loss of generality $A\simeq 1$. By Corollary~\ref{cor:YonedaEmbedding}, the functor $h_\ast\colon\iFun{\I{C}}{\I{D}}\into\iFun{\I{C}}{\IPSh_{\Univ}(\I{D})}\simeq\iFun{\I{C}\times\I{D}^{\op}}{\Univ}$ is fully faithful and therefore in particular conservative. It therefore suffices to show that the map $h_\ast\alpha$ is an equivalence. For any $(c,d)\colon 1\to\I{C}\times\I{D}^{\op}$, the map $h_\ast\alpha(c,d)$ corresponds to the image of $\alpha(c)$ along the functor $\I{D}\into\IPSh_{\Univ}(\I{D})\to\Univ$
	in which the second arrow is given by evaluation at $d$. As a consequence, the map $h_\ast\alpha(c,d)$ must be an equivalence in $\Univ$. By replacing $\BB$ with $\Over{\BB}{B}$ and using Lemma~\ref{lem:BCEvaluation} and Lemma~\ref{lem:BCYoneda}, the same is true when $(c,d)$ is in arbitrary context $B\in\BB$. By replacing $\I{C}$ with $\I{C}\times\I{D}^{\op}$, we may therefore assume without loss of generality $\I{D}\simeq\Univ$. In this case, the desired result follows from Proposition~\ref{prop:equivalenceLeftFibrationsFibrewise}.
\end{proof}
\begin{definition}
	\label{def:representablePresheaf}
	Let $\I{C}$ be a $\BB$-category. Then a presheaf $f\colon A\times\I{C}^{\op}\to\Univ$ is said to be \emph{representable} by an object $c\colon A\to \I{C}$ if there is an equivalence $f\simeq h(c)$, where $h\colon \I{C}\into\IPSh_{\Univ}(\I{C})$ denotes the Yoneda embedding.
\end{definition}
\begin{remark}
	\label{rem:representableFibration}
	If $p\colon\I{P}\to A\times\I{C}$ is a right fibration between $\BB$-categories (i.e.\ an object $A\to \IRFib_{\I{C}}$), we may say that $p$ is \emph{representable} if the associated presheaf $A\times\I{C}^{\op}\to \Univ$ that classifies $p$ is representable in the sense of Definition~\ref{def:representablePresheaf}. Equivalently, this means that there is an object $c\colon A\to \I{C}$ and an equivalence $\Over{\I{C}}{c}\simeq \I{P}$ over $A\times\I{C}$.
\end{remark}
\begin{proposition}
	\label{prop:representabilityFibration}
	A left fibration $p\colon\I{P}\to\I{C}\times A$ between $\BB$-categories is representable by an object $c\colon A\to \I{C}$ if and only if there is an initial section $A\to \I{P}$ over $A$.
\end{proposition}
\begin{proof}
	If $p$ is representable by an object $c\colon A\to \I{C}$ then there is an equivalence $\I{P}\simeq \Under{\I{C}}{c}$ over $\I{C}\times A$, hence Proposition~\ref{prop:initialityCanonicalSection} implies that there is an initial section $A\to \I{P}$ over $A$. Conversely, if there is such an initial section $s\colon A\to \I{P}$ and if $c\colon A\to \I{C}$ denotes the image of $s$ along the functor $\pr_0p\colon\I{P}\to\I{C}\times A\to \I{C}$, the lifting problem
	\begin{equation*}
		\begin{tikzcd}
			A\arrow[d, "s"]\arrow[r, "\id_c"] & \Under{\I{C}}{c}\arrow[d, "(\pi_c)_!"]\\
			\I{P}\arrow[ur, dotted] \arrow[r, "p"] & \I{C}\times A
		\end{tikzcd}
	\end{equation*}
	admits a unique solution which is necessarily an equivalence since $\id_c$ is initial and $p$ is a left fibration. 
\end{proof}

\bibliographystyle{halpha}
\bibliography{references.bib}

\end{document}